\documentclass[reqno]{amsart} 

\usepackage{color}
\usepackage{xcolor}

\usepackage{svg} %%%%%% not sure if this is needed outside overleaf???

\usepackage{ifpdf}
\ifpdf 
\usepackage{graphicx}   % to include graphics
\usepackage[pdftex,     % sets up hyperref to use pdftex driver
plainpages=false,   % allows page i and 1 to exist in the same document
breaklinks=true,    % link texts can be broken at the end of line
colorlinks=true,
linkcolor=blue,
 citecolor=green,
pdftitle={On the transport of currents and geometric Rademacher-type Theorems},
pdfauthor={Paolo Bonicatto, Giacomo Del Nin, Filip Rindler}
]{hyperref} 
\else 
\usepackage{graphicx}       % to include graphics
\usepackage{hyperref}       % to simplify the use of \href
\fi 

\usepackage{amsfonts,amsmath}	
\usepackage{amssymb}
\usepackage{verbatim}
\usepackage{amsopn}
\usepackage[english]{babel}
\usepackage{amsthm}
\usepackage{enumerate}
\usepackage{mathrsfs}	
\usepackage{mathtools}
\usepackage{harpoon}
\usepackage{esint}
\usepackage{todonotes}

\usepackage{stmaryrd}
\usepackage{bm}
\usepackage{bbm}
\usepackage{caption}
\usepackage{subcaption}
\usepackage{nicefrac}
\usepackage[a4paper,left=35mm,right=35mm,top=34mm,bottom=40mm,marginpar=25mm]{geometry}

\newcommand{\Crm}{\mathrm{C}}

\newcommand{\Irm}{\mathrm{I}}

\newcommand{\Lrm}{\mathrm{L}}
\newcommand{\Mrm}{\mathrm{M}}
\newcommand{\Nrm}{\mathrm{N}}

\newcommand{\Dcal}{\mathcal{D}}

\newcommand{\Ical}{\mathcal{I}}

\newcommand{\Lcal}{\mathcal{L}}

\newcommand{\Fbf}{\mathbf{F}}

\newcommand{\Mbf}{\mathbf{M}}

\newcommand{\Fbb}{\mathbb{F}}

\newcommand{\Nbb}{\mathbb{N}}

\newcommand{\norm}[1]{\|#1\|}

\newcommand{\abs}[1]{|#1|}

\newcommand{\absBB}[1]{\biggl|#1\biggr|}

\newcommand{\altnorm}[1]{{\left\vert\kern-0.25ex\left\vert\kern-0.25ex\left\vert #1 \right\vert\kern-0.25ex\right\vert\kern-0.25ex\right\vert}}
\newcommand{\eps}{\varepsilon}

\newcommand{\dprb}[1]{\bigl\langle #1 \bigr\rangle}

\newcommand{\ibf}{\mathbf{i}}
\newcommand{\tbf}{\mathbf{t}}
\newcommand{\pbf}{\mathbf{p}}
\newcommand{\qbf}{\mathbf{q}}
\newcommand{\rbf}{\mathbf{r}}

\newcommand{\Var}{{\mathrm{Var}}}

\newcommand{\pV}{\mathrm{pV}}
\newcommand{\eV}{\mathrm{eV}}
\newcommand{\dash}{\text{-}}

\theoremstyle{definition} \newtheorem{definition}{Definition}[section]
\theoremstyle{definition} \newtheorem{remark}[definition]{Remark}
\theoremstyle{plain} \newtheorem{lemma}[definition]{Lemma}
\theoremstyle{plain} \newtheorem{proposition}[definition]{Proposition}
\theoremstyle{plain} \newtheorem{theorem}[definition]{Theorem}
\theoremstyle{plain} \newtheorem{corollary}[definition]{Corollary}
\theoremstyle{definition} \newtheorem{example}[definition]{Example}
\theoremstyle{plain} 
\theoremstyle{definition} 
\theoremstyle{plain} 
\theoremstyle{plain}

\DeclareMathOperator{\AC}{AC}
\DeclareMathOperator{\BV}{BV}
\DeclareMathOperator{\SBV}{SBV}
\DeclareMathOperator{\dive}{div}

\DeclareMathOperator{\Lip}{Lip}

\DeclareMathOperator{\curl}{curl}
\DeclareMathOperator{\Wedge}{{\textstyle\bigwedge}}

\DeclareMathOperator*{\wslim}{w*-lim}

\newcommand{\ee}{\mathrm{e}}

\newcommand{\sbullet}{\begin{picture}(1,1)(-0.5,-2.5)\circle*{2}\end{picture}}
\newcommand{\frarg}{\,\sbullet\,}

\newcommand{\Fbold}{\mathbf{F}}
\newcommand{\Mbold}{\Mbf}
\newcommand{\R}{\mathbb{R}}

\newcommand{\N}{\mathbb{N}}
\newcommand{\Z}{\mathbb{Z}}

\newcommand{\loc}{\text{\rm loc}}

\renewcommand{\L}{\mathscr L}

\newcommand{\1}{\mathbbm 1}

\newcommand{\dd}{\mathrm{d}}
\newcommand{\DD}{\mathrm{D}}

\newcommand{\Crit}{\mathrm{Crit}}

\newcounter{counter}

\renewcommand{\S}{S}
\newcommand{\T}{T}

\newcommand{\weaksto}{\overset{*}{\rightharpoonup}}
\newcommand{\Fhom}{\mathbb{F}}
\newcommand{\Dscr}{\mathscr{D}}

\renewcommand{\L}{\mathscr L}
\renewcommand{\H}{\mathscr H}

\theoremstyle{plain} \newtheorem*{theorem*}{Theorem}
\theoremstyle{plain} 
\theoremstyle{plain} \newtheorem*{mthm*}{Main Theorem}
\theoremstyle{plain} \newtheorem*{conjecture*}{Conjecture}
\theoremstyle{plain} 
\theoremstyle{plain} \newtheorem*{problem*}{Problem}

\newcommand{\bbb}[1]{\llbracket #1 \rrbracket}

\numberwithin{equation}{section}

\usepackage{color}
\usepackage{graphicx}
\usepackage{tikz}
\usepackage{framed}
\definecolor{shadecolor}{rgb}{0.94, 0.97, 1.0}

\usepackage{pict2e}
\makeatletter
\DeclareRobustCommand{\intprod}{%
  \mathbin{\mathpalette\int@prod{(0.1,0)(0.9,0)(0.9,0.8)}}}
\DeclareRobustCommand{\restrict}{%
  \mathbin{\mathpalette\int@prod{(0.1,0.8)(0.1,0)(0.9,0)}}}	
\newcommand{\int@prod}[2]{%
  \begingroup
  \sbox\z@{$\m@th#1+$}%
  \setlength\unitlength{\wd\z@}%
  \begin{picture}(1,1)
  \roundcap
  \polyline#2
  \end{picture}%
  \endgroup
}
\makeatother

\usepackage[normalem]{ulem}

\title[Transport of currents]{Transport of currents and\\geometric Rademacher-type Theorems}

\author{Paolo Bonicatto}
\address[P.\ Bonicatto]{Mathematics Institute, University of Warwick,
Zeeman Building, CV4 7AL Coventry, UK}
\email{Paolo.Bonicatto@warwick.ac.uk}

\author{Giacomo Del Nin}
\address[G.\ Del Nin]{Mathematics Institute, University of Warwick,
Zeeman Building, CV4 7AL Coventry, UK}
\email{Giacomo.Del-Nin@warwick.ac.uk}

\author{Filip Rindler}
\address[F.\ Rindler]{Mathematics Institute, University of Warwick,
Zeeman Building, CV4 7AL Coventry, UK}
\email{F.Rindler@warwick.ac.uk}

\begin{document}

\maketitle

%\hrule\vspace{1pt}
%\begin{center}
%\textbf{\large
%DRAFT -- Version of \today}
%\end{center}
%\hrule
%\vspace{10mm}

\begin{abstract}
The transport of many kinds of singular structures in a medium, such as vortex points/lines/sheets in fluids, dislocation loops in crystalline plastic solids, or topological singularities in magnetism, can be expressed in terms of the geometric (Lie) transport equation
\[
  \frac{\mathrm{d}}{\mathrm{d} t} T_t + \mathcal{L}_{b_t} T_t = 0
\]
for a time-indexed family of integral or normal and usually boundaryless $k$-currents $t \mapsto T_t$ in an ambient space $\mathbb{R}^d$ (or a subset thereof). Here, $b_t = b(t,\cdot)$ is the driving vector field and $\mathcal{L}_{b_t} T_t$ is the Lie derivative of $T_t$ with respect to $b_t$ (introduced in this work via duality). Written in coordinates for different values of $k$, this PDE encompasses the classical transport equation ($k = d$), the continuity equation ($k = 0$), as well as the equations for the transport of dislocation lines in crystals ($k = 1$) and membranes in liquids ($k =d-1$). The top-dimensional and bottom-dimensional cases have received a great deal of attention in connection with the DiPerna--Lions and Ambrosio theories of Regular Lagrangian Flows. On the other hand, very little is rigorously known at present in the intermediate-dimensional cases. This work thus develops the theory of the geometric transport equation for arbitrary $k$ and in the case of boundaryless currents $T_t$, covering in particular existence and uniqueness of solutions, structure theorems, rectifiability, and a number of Rademacher-type differentiability results. The latter yield, given an absolutely continuous (in time) path $t \mapsto T_t$, the existence almost everywhere of a ``geometric derivative'', namely a driving vector field $b_t$. This subtle question turns out to be intimately related to the critical set of the evolution, a new notion introduced in this work, which is closely related to Sard's theorem and concerns singularities that are ``smeared out in time''. Our differentiability results are sharp, which we demonstrate through an explicit example of a heavily singular evolution, which is nevertheless Lipschitz-regular in time.
\vspace{4pt}



\noindent\textsc{Date:} \today{}.
\end{abstract}

\vspace{20pt}

\setcounter{tocdepth}{1}
\tableofcontents

\newpage

\section{Introduction}
Transport phenomena are ubiquitous in physics and engineering: For instance, given a bounded (time-dependent) vector field $b_t = b(t,\frarg) \colon \R^d \to \R^d$, $t \in [0,1]$ (the time interval being $[0,1]$ throughout this work for reasons of simplicity only), the \emph{transport equation}
\begin{equation} \label{eq:transport}
  \frac{\dd}{\dd t} u_t + b_t \cdot \nabla u_t = 0,  
\end{equation}
describes the transport of \emph{scalar fields} $u_t = u(t,\frarg) \colon \R^d \to \R$ (e.g., electrical potentials), while the \emph{continuity equation}
\begin{equation} \label{eq:continuity}
  \frac{\dd}{\dd t} \mu_t + \dive( b_t \mu_t) = 0,  
\end{equation}
describes the transport of \emph{densities}, or, more generally, of measures $\mu_t=\mu(t,\frarg)$ representing (possibly singular) mass distributions.

Another example of a transport phenomenon is the motion of \emph{dislocations}, which constitutes the main mechanism of plastic deformation in solids composed of crystalline materials, e.g. metals~\cite{AbbaschianReedHill09,HullBacon11book,AndersonHirthLothe17book} (also see~\cite{HudsonRindler21?} for an extensive list of further references). Dislocations are topological defects in the lattice of the crystal material which carry an orientation and a ``topological charge'', called the \emph{Burgers vector}. If one considers \emph{fields} $\tau_t = \tau(t,\frarg) \colon \R^3 \to \R^3$ of continuously-distributed dislocations (for a fixed Burgers vector) transported by a velocity field $b_t$, one obtains the \emph{dislocation-transport} equation
\begin{equation} \label{eq:transport_dislfields}
  \frac{\dd}{\dd t}\tau_t+\curl(b_t \times \tau_t) = 0.
\end{equation}
This equation follows more or less directly from the Reynolds transport theorem for $1$-dimensional quantities. A theory of plasticity based on dislocation transport are the \emph{field dislocation mechanics} developed by Acharya, see, for instance,~\cite{AA} and~\cite{AT}, and the recent variational model in~\cite{HudsonRindler21?, Rindler21a?, Rindler21b?}.

Further, also the movement of membranes in a medium can be described by a suitable transport equation, which, when formulated for the normal vector $\vec{\alpha}_t = \vec{\alpha}(t,\frarg) \colon \R^d \to \R^d$ to the surface, reads formally as
\begin{equation} \label{eq:transport_membranes}
  \frac{\dd}{\dd t} \vec{\alpha}_t + \nabla (\vec{\alpha}_t \cdot b_t) = 0.
\end{equation}

In all of the above equations, the case of ``singular'' objects being transported is just as natural as the case of fields. Besides the dislocations mentioned already, moving point masses, lines, or sheets are particularly relevant in fluid mechanics when considering concentrated vorticity. Intermediate-dimensional structures also appear in the setting of Ginzburg--Landau energies, even in the static case, see \cite{ABO, JS, BCL}. On the other hand, the continuum case corresponds to ``fields'' of such points, lines, membranes, etc., and should arise via homogenisation. However, the rigorous justification of this limit passage is missing in many cases, e.g., in the theory of dislocations, where it constitutes one of the most important open problems. This is partly due to the present lack of understanding of the singular versions of the equations above.

The starting point for the present work is the observation that the transport equation~\eqref{eq:transport}, the continuity equation~\eqref{eq:continuity}, the dislocation-transport equation~\eqref{eq:transport_dislfields} and the equation for the transport of membranes~\eqref{eq:transport_membranes}, as well as several other transport-type equations, are all special cases of the \emph{geometric transport equation}
\begin{equation} \label{eq:geom_transport} \tag{GTE}
  \frac{\dd}{\dd t} T_t + \Lcal_{b_t} T_t = 0
\end{equation}
for families of normal or integral $k$-currents $t \mapsto T_t \in \Nrm_k(\R^d)$, $t \in [0,1]$ (some notation for differential forms and currents is recalled in Section~\ref{sc:prelims}). We understand this equation in a weak sense, that is, for every $\psi\in \Crm^\infty_c((0,1))$ and every smooth $k$-form $\omega \in \Dscr^k(\R^d)$ it needs to hold that
\begin{equation} \label{eq:geom_transport_weak}
  \int_0^1 \dprb{\T_t,\omega} \psi'(t) - \dprb{ \Lcal_{b_t}\T_t, \omega} \psi(t) \; \dd t=0.
\end{equation}
Here, we define the \emph{Lie derivative} $\Lcal_{b_t} T_t$ of $T_t$ with respect to $b_t$ as the current given by 
\[
  \Lcal_{b_t} T_t := -b_t \wedge \partial T_t - \partial (b_t \wedge T_t), 
\]
which is obtained by duality via Cartan's formula for differential forms. 
In order to make sense of~\eqref{eq:geom_transport_weak}, it is further natural to assume the following integrability condition for $b_t$ relative to the family $(T_t)_t$, which forms part of our notion of weak solution:
\[
  \int_0^1 \int_{\R^d} |b_t| \; \dd(\|T_t\|+\|\partial T_t\|) \; \dd t < \infty.
\]

Consider first the case $k = 0$ for~\eqref{eq:geom_transport} and assume that we are looking for solutions with finite mass. It is well-known that $0$-currents $T_t$ ($t \in [0,1]$) with finite mass are actually measures, $T_t = \mu_t$, and in coordinates~\eqref{eq:geom_transport} turns out to be precisely the continuity equation~\eqref{eq:continuity} for $\mu_t$ (see Section~\ref{sc:particular} for this and the following computations). This is natural since the continuity equation expresses \emph{mass} transport. In this way,~\eqref{eq:geom_transport} can be seen as a generalisation of the continuity equation.

In the case $k = d$ for~\eqref{eq:geom_transport}, the moving objects are now top-dimensional currents, which we can think of as assigning a (signed) volume or (scalar) field strength to every point. In the case of a smooth volume or scalar field $u \colon [0,1] \times \R^d \to \R$, we may write
\[
  T_t = u_t(x) \, (\ee_1 \wedge \cdots \wedge \ee_d) \, \L^d \quad\text{as measures},
\]
where $u_t := u(t,\frarg)$. Then, we obtain the transport equation~\eqref{eq:transport} as the coordinate expression of~\eqref{eq:geom_transport}. 

Finally, \eqref{eq:transport_dislfields} is the coordinate expression of~\eqref{eq:geom_transport} for $k = 1$ and~\eqref{eq:transport_membranes} is the coordinate expression of~\eqref{eq:geom_transport} for $k = d-1$ (see again Section~\ref{sc:particular} for the details of the computations). 

Deep investigations have been conducted into the transport and continuity equations, that is, the top-dimensional and bottom-dimensional cases of the geometric transport equation~\eqref{eq:geom_transport}. In particular, the theory of Regular Lagrangian Flows pioneered by DiPerna--Lions and Ambrosio, see, e.g.,~\cite{DPL, AmbrosioBV, AmbrosioCrippa} and the references contained therein, has enabled many new and surprising insights into the equations of fluids. On the contrary, almost nothing rigorous seems to be known about the intermediate-dimensional cases in~\eqref{eq:geom_transport}, not even in the case of smooth vector fields $b_t$.

This paper will mostly focus on the case of \emph{boundaryless} integral (or, in some cases, normal) currents. Indeed, singular structures in applications are often prevented from ending within the medium; for instance, this is the case for dislocations in single crystals (stemming from atomistic reasons~\cite{HullBacon11book}). Moreover, if boundaries do happen to occur, then they themselves are transported, but by a transport-type equation in one dimension fewer and, potentially, a different driving vector field. The case of transported currents with boundary thus involves further complications and is not considered in this work (besides some results that are indifferent to the presence of boundaries).

\subsection*{Space-time solutions and rectifiability}

When considering the transport of a family of currents $t \mapsto T_t$, it turns out that another notion of solution is, in a sense, more natural than the weak solutions considered above, namely the \emph{space-time solutions}. This concept builds on the theory of space-time currents, introduced in \cite{Rindler21a?}, and can be explained, in the case of integral currents, as follows: Let $S$ be a $(1+k)$ integral current in $[0,1] \times \R^d$. Denote by $S|_t$ the slice of $S$ at time $t$ (with respect to the time projection $\tbf(t,x):=t$) and by
\[
  S(t) := \pbf_*(S|_t)
\]
its pushforward under the spatial projection $\pbf(t,x) := x$. Standard theory gives that $S(t)$ is an integral $k$-current in $\R^d$ and that the orienting map $\vec{S} \in \Lrm^\infty(\norm{S};\Wedge_k (\R\times\R^d))$ (with $\abs{\vec{S}} = 1$ $\norm{S}$-a.e.) decomposes orthogonally as
\[
  \vec{S} = \xi \wedge \vec{S}|_t,  \qquad \xi \perp \vec{S}|_t,
\]
where $\vec{S}|_t$ is the orienting map of the slice $S|_t$, and
\[
  \xi(t,x) := \frac{\nabla^S \tbf(t,x)}{\abs{\nabla^S \tbf(t,x)}}.
\]
Here, $\tbf(t,x) := t$ is the temporal projection and $\nabla^S \tbf$ is its approximate gradient with respect to $S$, i.e., the projection of $\nabla \tbf$ onto the approximate tangent space $\mathrm{Tan}_{(t,x)} S$ (more precisely, ``$S$'' should be replaced by the $\H^{1+k}$-rectifiable carrier set of $S$ here, but we consider this to be implicit). See Figure~\ref{fig:spacetime} for an illustration.

\begin{figure}
         \centering
         \includegraphics[width=0.45\textwidth]{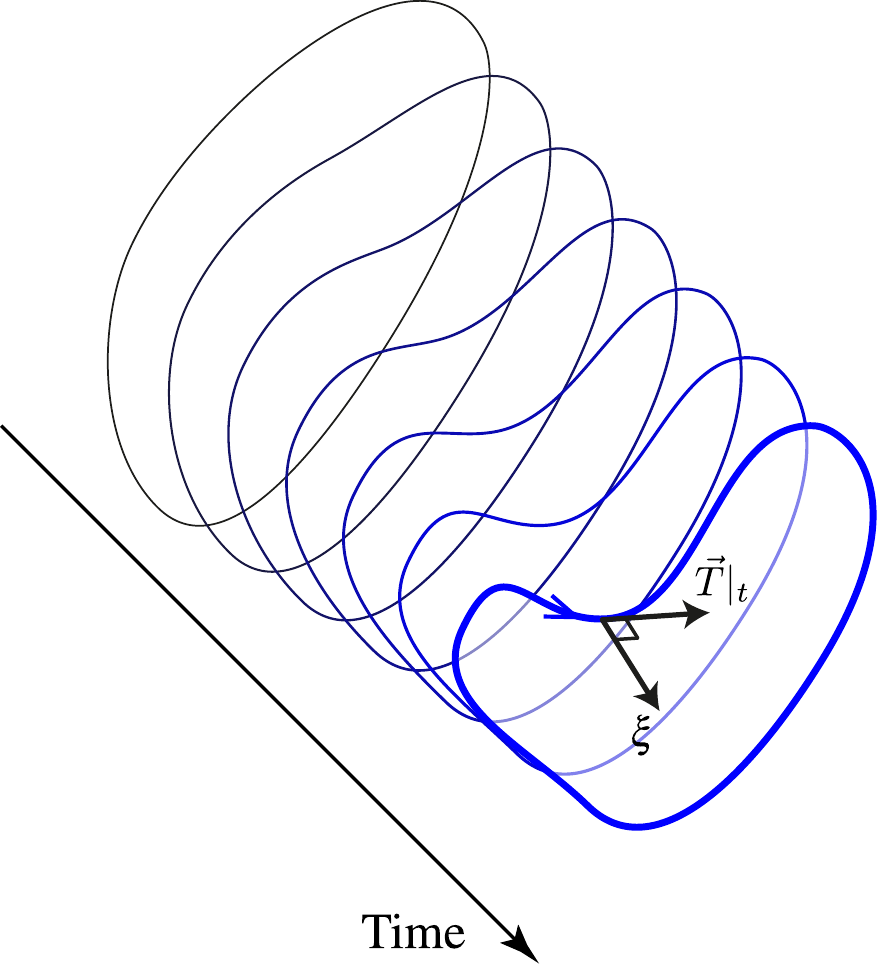}
         \caption{An illustration of an evolution of an integral $1$-current described via a space-time integral $2$-current.}
         \label{fig:spacetime}
\end{figure} 

We can now define the \emph{geometric derivative} of $S$ as the (normal) change of position per time of a point travelling on the current being transported, that is,
\[
  \frac{\DD}{\DD t} S(t,x) := \frac{\pbf(\xi(t,x))}{\abs{\tbf(\xi(t,x))}}
  = \frac{\pbf(\xi(t,x))}{\abs{\nabla^S \tbf(t,x)}}=\pbf\left(\frac{\nabla^S\tbf(t,x)}{|\nabla^S\tbf(t,x)|^2}\right)
\]
for $\norm{S}$-a.e.\ $(t,x)$. Clearly, the geometric derivative exists only outside the \emph{critical set}
\[
	\Crit(S) := \bigl\{(t,x) \in [0,1] \times \R^d:  \nabla^S \tbf (t,x)= 0 \bigr\},
\]
which is related to Sard's theorem and which turns out to play a major role in this work.

We then say that a space-time current $S$ as above is a \emph{space-time solution} of~\eqref{eq:geom_transport} if
\begin{equation} \label{eq:geom_transport_spacetime}
  \frac{\DD}{\DD t} S(t,x) = b(t,x)  \qquad\text{for $\norm{S}$-a.e.\ $(t,x)$.}
\end{equation}
In fact, this is the approach taken in the modelling of dislocation movements contained in~\cite{HudsonRindler21?,Rindler21a?,Rindler21b?}. There, the geometric derivative can be identified with the (normal) dislocation velocity, which is a key quantity in any theory of plasticity driven by dislocation motion.

One can see without too much effort that space-time solutions give rise to weak solutions: The projected slices $S(t) := \pbf_*(S|_t)$ of an integral $(1+k)$-current $S$ satisfying~\eqref{eq:geom_transport_spacetime} solve~\eqref{eq:geom_transport}. The converse question, that is, when a collection of currents $\{T_t\}_t$ solving~\eqref{eq:geom_transport} can be realised as the slices of a space-time current lies much deeper. A positive result on this \emph{space-time rectifiability} will be a principal theorem of this paper.

\subsection*{Main results}

This work develops a general theory of the geometric transport equation~\eqref{eq:geom_transport} in the case of transported integral (sometimes only normal) $k$-currents, including the case of intermediate dimensions ($k \neq 0, d$). Before stating our main results, it is worthwhile to note that the primary goal of our analysis does not lie in proving existence and uniqueness of solutions (even though we prove such a theorem for the sake of completeness), but in analysing the structure and properties of solutions, such as rectifiability and (geometric) differentiability, as well as in understanding the relation between weak solutions and space-time solutions. Indeed, in applications the existence of solutions, usually to much more complicated systems with the transport equation being only one ingredient, can often be established in a variational way by minimising over all possible paths of currents, similar to the classical minimising movement scheme; for instance, this is the approach taken in~\cite{Rindler21b?}. On the other hand, the questions investigated in the present work go to the heart of the physical modelling, such as establishing the existence of the geometric derivative (which is the ``slip velocity'' in the theory of dislocations) and establishing in what sense one may understand weak solutions to~\eqref{eq:transport} as ``physical''.

To wit, we prove the following main theorems:

\begin{itemize}
  \item \emph{Existence \& Uniqueness Theorem~\ref{thm:smooth}}: In the case where the driving vector field is assumed smooth, we show the existence and uniqueness of a path of integral currents solving the geometric transport equation.
  \item \emph{Disintegration Structure Theorem~\ref{thm:struct}:} For a space-time current, this theorem details the structure of its slices. In particular, it clarifies the role of ``critical points'' of the currents, which turn out to be central.
  \item \emph{Equality of Variations Theorem~\ref{thm:Var_equals_TV}}: We prove the equality between two different notions of variation for boundaryless space-time currents: the \emph{pointwise total variation} (with respect to the homogeneous flat norm) of the map $t \mapsto S\vert_t$ and the \emph{space-time variation} introduced in \cite{Rindler21a?}.  
  \item \emph{Rectifiability Theorem~\ref{thm:rect}:} In a sense a converse to the disintegration structure theorem, this result shows that if a time-indexed collection of boundaryless integral $k$-currents satisfies suitable continuity conditions, then these currents are in fact the slices of a space-time integral current. In this sense the path is ``space-time rectifiable'' (i.e., rectifiable of dimension $1+k$).
  \item \emph{Advection Theorem~\ref{thm:advection}:} We show that a boundaryless space-time current satisfies the \emph{negligible criticality condition} (meaning that critical points are negligible for the mass measure of the current) if and only if its slices are advected by some vector field.
  \item \emph{Weak and Strong Rademacher-type Differentiability Theorems~\ref{thm:weak_rademacher},~\ref{thm:strong_rademacher}:} These results show that a time-indexed family of boundaryless integral currents, satisfying suitable Lipschitz-continuity (or even absolute continuity) assumptions, is a solution to the geometric transport equation for some driving vector field.
\end{itemize}

We call the last results ``Rademacher-type'' theorems because they show the differentiability of Lipschitz-continuous evolutions of currents with respect to the time variable. However, the notion of derivative is not the classical one (e.g., in the Gateaux-sense with respect to the linear structure of the space of normal currents), but the geometric one defined above.

\subsection*{Technical aspects}

On a technical side let us point out three important facts that play a central role in our theory: The first point is that, for boundaryless integral currents, there are two possible definitions of the homogeneous (boundaryless) Whitney flat norm (we always work in spaces of trivial homology):
\begin{align*}
  \Fbb(T) &:= \inf\bigl\{ \Mbf(Q) : Q\in\Nrm_{k+1}(\R^d), \, T=\partial Q \bigr\}, \\
  \Fbb_{\Irm}(T) &:= \inf\bigl\{ \Mbf(Q) : Q\in\Irm_{k+1}(\R^d), \, T=\partial Q\},
\end{align*}
meaning that in $\Fbb(T)$ we use normal test currents and in $\Fbb_{\Irm}(T)$ we use integral test currents. It is known that for $k\in \{0,d-2,d-1,d\}$ these two flat norms coincide, but for other $k$ their equivalence seems to be unknown (in fact, we establish some new information on this question in Proposition~\ref{prop:asymptotic_equivalence}). A number of our results (e.g., the Rademacher-type differentiability theorems) are quite sensitive to the type of flat norm used and we need to carefully distinguish them in this work.

The second point is that one can define the ``variation'' of a path of integral or normal currents in several different ways and we compare these different ways in Section~\ref{sc:variation}. It is then a consequence of the (non-trivial) Rectifiability Theorem~\ref{thm:rect} that the space-time variation and the essential variation with respect to $\Fbb_{\Irm}$ are in fact equal, answering in particular a question left open in~\cite{Rindler21a?, Rindler21b?}.

\begin{figure}
         \centering
         \includegraphics[width=0.45\textwidth]{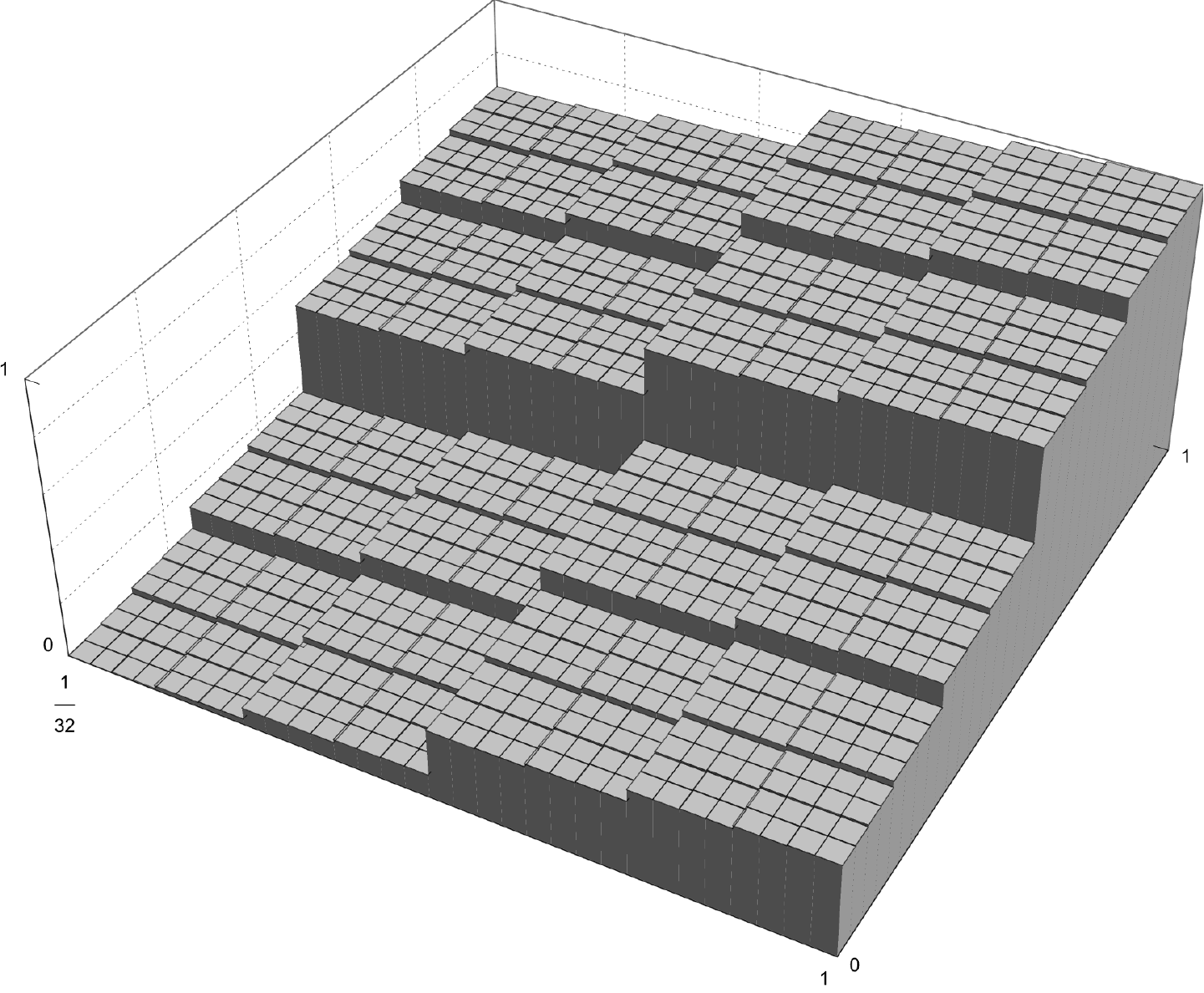}
         \caption{A step in the construction of the ``Flat Mountain'', which has many critical points, but whose slices (horizontal level-sets in this picture as the time direction is upwards) are nevertheless Lipschitz-regular in time.}
         \label{fig:Sard_intro}
\end{figure} 

As the third point we record the phenomenon that, unlike in the classical theory of BV-functions on a time interval, a path of integral $k$-currents may have a diffusely-concentrated or fractal structure \emph{in space only}, while being regular in time.
This phenomenon is closely related to Sard's theorem and to the presence of critical points smeared out in time. We give an explicit example of such behaviour in Section~\ref{sc:counterexample} (see Figure~\ref{fig:Sard_intro} for an illustration of one of the steps in the construction). In fact, the constructed current, which we call the ``Flat Mountain'', is even Lipschitz in time, so this effect is not linked to the time-regularity of a path. This example in particular shows the sharpness of our differentiability results.

\subsection*{Acknowledgements}

This project has received funding from the European Research Council (ERC) under the European Union's Horizon 2020 research and innovation programme, grant agreement No 757254 (SINGULARITY). The authors would like to thank Giovanni Alberti, Gianluca Crippa, Thomas Hudson, and Fanghua Lin for discussions related to this work.

\section{Notation and preliminaries} \label{sc:prelims}

This section fixes our notation and recalls some basic facts. We refer the reader to \cite{Federer69book,KP} for notation and the main results we use about differential forms and currents.

\subsection{Linear and multilinear algebra.}
Let $d \ge 1$ be an integer. We will often use the projection maps $\tbf \colon\R\times \R^d\to \R$, $\pbf \colon\R\times \R^d\to \R^d$ from the (Euclidean) space-time $\R\times\R^d$ onto the time and space variables, respectively, which are given as
\[
\tbf(t,x) = t,  \qquad
\pbf(t,x) = x.
\]

If $V$ is a (finite-dimensional, real) vector space, for every $k \in \N$, we let $\Wedge^k V$ be the space of \textbf{$k$-covectors} on $V$, i.e., the space of real-valued functions $\alpha \colon V \times \cdots \times V \to \R$ which are linear and alternating, 
\begin{equation*}
\alpha(v_{\sigma(v_1)}, \ldots , v_{\sigma(v_k)}) = \text{sgn}(\sigma) \cdot \alpha (v_1, \ldots, v_k)
\end{equation*}
for every $v_1, \ldots , v_k \in V$ and for every permutation $\sigma$ of the set $\{1, \ldots , k\}$. By duality, we define the space of \textbf{$k$-vectors} on $V$ as $\Wedge_k V := \Wedge^k(V^*)$. A $k$-vector $\eta\in \Wedge_k V$ is called \textbf{simple} if $\eta=v_1\wedge\ldots\wedge v_k$, where $v_i\in V$ and $\wedge$ denotes the exterior \textbf{wedge product}. Recall also that the duality between $V$ and $V^*$ gives us also a duality between $\Wedge^k V$ and $\Wedge_k V$, the duality pairing being uniquely determined by
\begin{equation*}
	\langle  v_1 \wedge \ldots \wedge v_k,\alpha\rangle  = \alpha(v_1, \ldots , v_k), \qquad  v_1 \wedge \ldots \wedge v_k  \in \Wedge_k V, \; \alpha \in \Wedge^k V,
\end{equation*}
and then extended by linearity.

Whenever $V$ is an inner product space, we can endow $\Wedge_k V$ with an inner product (Euclidean) norm $|\cdot|$ by declaring $\ee_I:=\ee_{i_1}\wedge\ldots \wedge \ee_{i_k}$, as $I$ varies in the $k$-multi-indices of $\{1,\ldots, n\}$, as \emph{orthonormal} whenever $\ee_1,\ldots,\ee_n$ are an orthonormal basis of $V$. A simple $k$-vector $\eta\in\Wedge_k V$ is called \textbf{unit} if there exists an orthonormal family $v_1,\ldots,v_k$ such that $\eta=v_1\wedge\ldots\wedge v_k$, or equivalently if its Euclidean norm $|\eta|$ equals 1. We define the \textbf{comass} of $\alpha\in\Wedge^kV$ as
\[
\|\alpha\|:=\sup\bigl\{\langle \eta,\alpha\rangle: \eta\in\Wedge_k V, \text{ simple, unit}\bigr\}
\]
and the \textbf{mass} of $\eta\in\Wedge_kV$ as
\[
\|\eta\|:=\sup\bigl\{\langle \eta,\alpha\rangle: \alpha\in\Wedge^k V,\,\|\alpha\|\leq 1\bigr\}.
\]
We have that $\|\eta\|=|\eta|$ if and only if $\eta$ is simple (and the same holds for covectors). From the definition it follows that
\begin{equation}\label{eq:cauchy_schwartz_mass_comass}
|\langle \eta,\alpha\rangle|\leq \,\|\eta\|\|\alpha\|,\qquad\text{for every $\eta\in\Wedge_k V$, $\alpha\in\Wedge^k V$.}
\end{equation}
For $\eta \in \Wedge_{k}V$ and $\alpha\in \Wedge^\ell V$, we furthermore define the \textbf{interior products} $\eta\intprod \alpha \in \Wedge^{\ell-k}V$ (whenever $\ell\ge k$) and $\eta \restrict \alpha \in \Wedge_{k-\ell}V$ (whenever $\ell\le k$) via 
\[
\left \langle\zeta,\eta\intprod \alpha\right\rangle :=\langle \zeta\wedge\eta ,\alpha\rangle\qquad\text{for every $\zeta\in\Wedge_{\ell-k}V$}
\]
and 
\[
\left \langle\eta\restrict \alpha ,\beta\right\rangle :=\langle \eta, \alpha \wedge \beta \rangle\qquad\text{for every $\beta\in\Wedge^{k-\ell}V$}.
\]
In the following, we will often use the interior product $\intprod$ between a $k$-covector $\omega \in \Wedge^k V$ and a 1-vector $b \in \Wedge_1 V$, and in this case write it as a \textbf{contraction}, setting (in analogy to what is usually done in differential geometry, see, e.g.,~\cite{Lee13book})
\[
i_{b}\omega := (-1)^{k} b \intprod \omega, 
\]
that is, for every $k$-vector $v \in \Wedge_k V$,  
\begin{equation}\label{eq:contraction}
\langle v,i_{b}\omega\rangle :=
\langle  b \wedge v,\omega\rangle.
\end{equation}
A \textbf{differential $k$-form} $\omega$ on $\R^d$ is a covector field defined on $\R^d$, i.e., a map that to each $x \in \R^d$ associates a covector $\omega(x) \in \Wedge^k\R^d$. We denote by $\Dscr^k(\R^d)$ the space of smooth $k$-forms on $\R^d$ with compact support.

Given a linear map $S\colon V\to W$, we define $\Wedge^k S \colon \Wedge_k V\to \Wedge_k W$ by
\[
( \Wedge^k S) (v_1\wedge \ldots\wedge v_k)= (Sv_1)\wedge\ldots\wedge (Sv_k),
\]
on simple vectors and then we extend this definition by linearity to all of $\Wedge_k V$. If there is no risk of confusion, we will often write simply $S$ instead of $\Wedge^{k}S$ to denote the extension of the map $S$ to $\Wedge^k V$.  
The \textbf{pullback of a covector} $\alpha\in\Wedge^kV$ with respect to a linear map $S:V\to W$ is given by
\[
\left \langle  v_1\wedge\ldots\wedge v_k,S^* \alpha\right\rangle:=\langle  (Sv_1)\wedge\ldots\wedge (Sv_k),\alpha\rangle
\]
on simple $k$-vectors and then extended by linearity. Therefore,
\[
\langle \eta,S^*\alpha\rangle= \bigl\langle(\Wedge^k S) \eta,\alpha\bigr\rangle.
\]
If $f:\R^d\to \R^d$ is a differentiable map and $\omega\in \Dscr^k(\R^d)$ we define the \textbf{pullback} $f^*\omega$ to be the differential form given by
\[
\langle v,(f^*\omega)(x)\rangle:= \bigl\langle (\Wedge^k df(x) )v,\omega(f(x))\bigr\rangle\qquad\text{for every $v\in \Wedge_k(\R^d)$.}
\]

The \textbf{comass of a form} $\omega\in \Dscr^k(\R^d)$ is
\[
\|\omega\|_\infty:= \sup_{x\in \R^d} \|\omega(x)\|.
\]
In particular, we see from~\eqref{eq:cauchy_schwartz_mass_comass} that for every $k$-vector $v \in \Wedge_k \R^d$,
\begin{equation}\label{eq:pullback_mass_comass}
|\langle v,(f^*\omega )(x)\rangle|\leq \bigl\|(\Wedge^k df(x))v\bigr\|\cdot\|\omega (f(x))\| , \qquad  x \in \R^d.
\end{equation}

\subsection{Currents} We refer to~\cite{Federer69book} for a comprehensive treatment of the theory of currents, summarising here only the main notions that we will need. The space of \textbf{$k$-dimensional currents} $\mathscr D_k(\R^d)$ is defined as the dual of $\Dscr^k(\R^d)$, where the latter space is endowed with the locally convex topology induced by local uniform convergence of all derivatives.
As for distributions, we are interested in the notion of (sequential weak*) convergence: 
\begin{equation*}
	T_n \weaksto T \text{ in the sense of currents } \Longleftrightarrow \langle T_n, \omega\rangle \to \langle T, \omega\rangle \quad  \text{for all $\omega \in \Dscr^k(\R^d)$.}
\end{equation*}
The \textbf{boundary} of a current is defined as the adjoint of De Rham's differential: if $T$ is a $k$-current, then $\partial T$ is the $(k-1)$-current given by 
\begin{equation*}
	\langle \partial T, \omega \rangle = \langle T, d\omega \rangle,   \qquad  \omega \in \Dscr^{k-1}(\R^d). 
\end{equation*}
We denote by $\Mrm_k(\R^d)$ the space of $k$-\textbf{currents with finite mass} in $\R^d$, where the \textbf{mass} of a current $T \in \mathscr{D}_k(\R^d)$ is defined as 
\begin{equation*}
	\mathbf{M}(T):= \sup \left\{ \langle T,\omega \rangle:  \omega \in \Dscr^k(\R^d), \| \omega\|_\infty \le 1 \right\}.
\end{equation*}

Let $\mu$ be a finite measure on $\R^d$ and let ${\tau} \colon \R^d \to \Wedge_k(\R^d)$ be a map in ${\rm L}^1(\mu)$. Then we define the current $T:={\tau}\mu$ as 
\begin{equation*}
	\langle T, \omega \rangle = \int_{\R^d} \langle \tau(x),\omega(x) \rangle\; \dd \mu(x).
\end{equation*}
We recall that all currents with finite mass can be represented as $T= \tau \mu$ for a suitable pair $\tau, \mu$ as above. In the case when $\|\tau\|=1$ $\mu$-a.e., we denote $\mu$ by $\|T\|$ and we call it the \textbf{mass measure} of $T$. As a consequence, we can write $T=\vec{T}\|T\|$, where $\|\vec{T}\|=1$ $\|T\|$-almost everywhere. One can check that, if $T=\tau\mu$ with $\tau\in \Lrm^1(\mu)$, then $\|T\|=\|\tau\| \mu$, hence
\[
\mathbf{M}(T) = \int_{\R^d} \|\tau(x)\|\; \dd \mu(x).
\]
Given a current $T=\tau \mu \in \Dscr_k(\R^d)$ with finite mass and a vector field $v:\R^d\to\R^d$ defined $\|T\|$-a.e., we define the \textbf{wedge product}
\[
  v\wedge T:=(v\wedge \tau) \mu\in \Dscr_{k+1}(\R^d).
\]
Whenever $T$ is a current with finite mass we can extend its action on forms that are merely $\|T\|$-summable, and in particular we can consider for every set $A\subseteq\R^d$ the current $T\restrict A$ given by
\[
\langle T\restrict A,\omega\rangle := \langle T,\1_A \omega\rangle\qquad\text{for every $\omega\in\Dscr^k(\R^d)$.}
\]
If $T\in \Dscr_k(\R^d)$ and $\omega\in \Dscr^\ell(\R^d)$, with $\ell\le k$, we similarly define the \textbf{restriction} $T\restrict \omega\in \Dscr_{k-\ell}(\R^d)$ by
\[
\langle T\restrict \omega,\eta\rangle:=\langle T,\omega\wedge \eta\rangle,\qquad\text{for every $\eta\in\Dscr^{k-\ell}(\R^d)$.}
\]
The \textbf{pushforward} of $\T$ with respect to a $\Crm^1$-map $f:\R^d\to\R^d$ is defined by
\[
  \langle f_* T,\omega\rangle=\langle \T,f^* \omega\rangle. 
\]
In view of~\eqref{eq:pullback_mass_comass} we have the estimate 
\begin{equation}\label{eq:mass_of_pushforward}
\Mbf(f_{*} T) \leq \int_{\R^d} \bigl\| (\Wedge^k df(x))\tau(x) \bigr\| \; \dd \|T\|(x)
\leq \|df\|_\infty^k \Mbf(T).
\end{equation}
If $T$ is \textbf{simple}, i.e., $\vec{T}$ is a simple $k$-vector $\|T\|$-almost everywhere, then the same inequality holds with the mass norm $\|\cdot\|$ replaced by the Euclidean norm $|\cdot|$.

Given two currents $T_1\in \mathscr{D}_{k_1}(\R^{d_1})$ and $T_2\in\mathscr{D}_{k_2}(\R^{d_2})$, their \textbf{product} $T_1\times T_2$ is a well-defined current in $\mathscr{D}_{k_1+k_2}(\R^{d_1}\times\R^{d_2})$. The boundary of the product is given by
\begin{equation}\label{eq:boundary_of_product}
\partial (T_1\times T_2)=\partial T_1 \times T_2 + (-1)^{k_1} T_1\times \partial T_2.
\end{equation}

A $k$-current on $\R^d$ is said to be \textbf{normal} if both $T$ and $\partial T$ have finite mass. The space of normal $k$-currents is denoted by $\mathrm N_k(\R^d)$.
The weak* topology on the space of (normal) currents has good properties of compactness and lower semicontinuity: if $(T_j)_j$ is a sequence of currents with $\Mbf(T_j) + \Mbf(\partial T_j) \le C < +\infty$ for every $j \in \N$, then there exists a normal current $T$ such that, up to a subsequence, $T_j \weaksto T$. Furthermore, 
\begin{equation*}
\Mbf(T) \le\liminf_{j \to +\infty} \Mbf(T_j) ,  \qquad
\Mbf(\partial T) \le\liminf_{j \to +\infty} \Mbf(\partial T_j).
\end{equation*}
For a normal current $T\in \Nrm_k(\R^d)$, it is possible to define the pushforward $f_*T$ when $f:\R^d\to\R^d$ is merely Lipschitz~\cite[4.1.14]{Federer69book}. 
We also mention that, by~\cite[4.1.20]{Federer69book}, if $T\in\Nrm_k(\R^d)$ then $\|T\|\ll\H^k$.

An \textbf{integer-multiplicity rectifiable} $k$-current $T$ is a
$k$-current of the form 
\[
	T = m \, \vec{T} \, \mathscr H^k \restrict R,
\]
where:
\begin{enumerate}
\item $R \subset \R^d$ is countably $\mathscr H^k$-rectifiable (that is, it can be covered up to a $\H^k$-null set by countably many images of Lipschitz functions from $\R^k$ to $\R^d$) with $\mathscr H^k(R \cap K) < \infty$ for all compact sets $K \subset \R^d$;
\item $\vec{T} \colon R \to \Wedge_k \R^d$ is $\mathscr H^k$-measurable and for $\mathscr H^k$-a.e.\ $x \in R$ the $k$-vector $\vec{T}(x)$ is simple, unit ($|\vec{T}(x)| = 1$), and its span coincides with the approximate tangent space $\mathrm{Tan}_x R$ to $R$ at $x$;
\item $m \in {\rm L}^1_\loc(\mathscr H^k \restrict R;\Z)$;
\end{enumerate}
The map $\vec{T}$ is called the \textbf{orientation map} of $T$ and $m$ is the \textbf{multiplicity}.
Let $T = \vec{T} \| T\|$ be the Radon--Nikod\'{y}m decomposition of $T$ with the \textbf{total variation measure} $\|T\| = |m| \, \mathscr H^k \restrict R \in \mathscr M^+_{\rm loc}(\R^d)$. Then we have %
\[
\Mbf(T) = \| T\|(\R^d) = \int_R |m(x)| \; \dd \mathscr H^k(x).
\]
We then define the space of \textbf{integral} $k$-currents ($k \in \N \cup \{0\}$):
\[
\Irm_k(\R^d) := \bigl\{ \text{$T$ integer-multiplicity rectifiable $k$-current}:  \Mbf(T) + \Mbf(\partial T) < \infty \bigr\}. 
\]

For $F \subset \R^d$ closed, the subspaces $\Irm_k(F)$, $\Nrm_k(F)$ are defined as the spaces of all $T \in \Irm_k(\R^d)$, or $T \in \Nrm_k(\R^d)$, respectively, with support (in the sense of measures) in $F$. Since $F$ is closed, these subspaces are weakly* closed.

An important property of integral currents is the \emph{Federer-Fleming compactness theorem}~\cite[Theorems 7.5.2, 8.2.1]{KP}: Let $(T_j)_j \subset \mathrm I_k(\R^d)$ with 
\[
\sup_{j \in \N} \,(\Mbf(T_j) + \Mbf(\partial T_j)) < \infty. 
\]
Then, there exists a (not relabeled) subsequence and a $T \in \mathrm I_k(\R^d)$ such that $T_j \weaksto T$ in the sense of currents.

\subsection{Flat norms}

For $T \in \Irm_k(\R^d)$, the \textbf{(Whitney) flat norm} is given by
\begin{equation}\label{eq:def_flat_norm_new}
\Fbf (\T):=\inf\{\Mbf(Q)+\Mbf(R): Q\in\mathrm{N}_{k+1}(\R^d),R\in\mathrm{N}_k(\R^d),\, T=\partial Q+R\}
\end{equation}
and one can also consider the \textbf{flat convergence} $\mathbf{F}(T-T_j) \to 0$ as $j \to \infty$. Under the mass bound $\sup_{j\in\N} \, \big( \mathbf{M}(T_j) + \mathbf{M}(\partial T_j) \big) < \infty$, this flat convergence is equivalent to weak* convergence (see, for instance,~\cite[Theorem~8.2.1]{KP} for a proof). The flat norm admits also a dual representation (see~\cite[4.1.12]{Federer69book}) as
\begin{equation}\label{eq:flat_norm_dual}
\Fbf(T)=\sup\{\langle T,\omega\rangle\colon \, \omega\in\Dscr^k(\R^d), \|\omega\|_\infty\leq 1,\,\|d\omega\|_\infty\leq 1\}.
\end{equation}
When $\partial T=0$, one can also consider the \textbf{homogeneous flat norm}
\begin{equation}\label{eq:def_hom_flat_norm}
\Fhom (\T):=\inf\{\Mbf(Q): Q\in\mathrm{N}_{k+1}(\R^d), \, T=\partial Q\},
\end{equation}
which also admits a dual representation as
\begin{equation}\label{eq:hom_flat_norm_dual}
\Fhom(T)=\sup \{\langle T,\omega\rangle:\, \omega\in\Dscr^k(\R^d),\, \|d\omega\|_\infty\leq 1\}.
\end{equation}

If $T$ is integral, one can also consider the corresponding integral versions of~\eqref{eq:def_flat_norm_new} and~\eqref{eq:def_hom_flat_norm}, called \textbf{integral flat norm} and \textbf{integral homogeneous flat norm} respectively:
\begin{align*}
\Fbf_{\Irm} (T)&:=\inf\{\Mbf(Q)+M(R): Q\in\Irm_{k+1}(\R^d),R\in\Irm_k(\R^d), \,T=\partial Q+R\},\\
\Fhom_{\Irm} (T)&:=\inf\{\Mbf(Q): Q\in\Irm_{k+1}(\R^d), \,T=\partial Q\}.
\end{align*}
These, however, do not admit a dual representation as in~\eqref{eq:flat_norm_dual} and~\eqref{eq:hom_flat_norm_dual}. Notice that these are not proper norms because $\Irm_k(\R^d)$ is not a vector space. Moreover, quite surprisingly, they can even fail to be positively $\mathbb{Z}$-homogeneous: there are examples of $S\in\Irm_1(\R^4)$ such that $\Fhom(2S)<2\Fhom(S)$ (see \cite{Young63,White84}, and \cite{Young18} for a recent advancement). Observe that, by compactness and lower semicontinuity, all the above infima in the definition of flat norms are attained. In the following, we will also consider the homogeneous flat norms $\Fhom, \Fhom_{\Irm}$ on the whole $\Nrm_{k}(\R^d)$ or $\Irm_{k}(\R^d)$, in which case they are understood to be $+\infty$ on currents that are not boundaryless.

\subsection{Slicing and coarea formula for integral currents} \label{ss:slicing} 

Given a Lipschitz function $f:\R^n\to \R$ and $S\in\Nrm_{k+1}(\R^n)$, the \textbf{slicing} of $S$ at level $t$ is defined by the following, which will be referred to as the \textbf{cylinder formula}:
\begin{equation}\label{eq:definition_slicing}
	S|_t:= \partial (S\restrict \{f<t\})-(\partial S)\restrict \{f<t\}.
\end{equation}
The slices with respect to $f$ are also characterised by the following property: 
\begin{equation}\label{eq:slice_general}
	\int  \S|_t \psi(t) \; \dd t=\S\restrict (\psi\circ f) df \qquad\text{for every $\psi\in \Crm^\infty_c(\R)$}.
\end{equation} 
see, e.g.,~\cite{AK} or also~\cite[4.3.2]{Federer69book}.
For integral currents the following coarea formula holds~\cite[Section 2.4]{Rindler21a?}: For every non-negative Borel function $g \colon \R^n \to \R$ we have
\begin{equation}\label{eq:coarea_integral}
\int_{\R^n} g(z) |\nabla^S f (z)| \; \dd \|S\|(z)=\int_{\R} \left(\int_{\R^n} g(z)\dd\|S|_t\|(z)\right) \; \dd t,
\end{equation}
where $\nabla^S f(z)$ denotes the tangential gradient of the map $f$ on the approximate tangent space to $S$ at $z$, that is, the projection of the vector $\nabla f(z)$ onto the approximate tangent space to $S$ at $z$ (see~\cite[Theorem 2.90]{AFP}). The equality~\eqref{eq:coarea_integral} holds also whenever $g \in {\rm L}^1(|\nabla^S f|\| S\|)$.

\subsection{Disintegration of measures}\label{ss:disintegration}

Given the product structure of the space-time $\R\times \R^d$, we will often work with product measures or \emph{generalised product measures} and we will consider the \emph{disintegration} of measures on $\R \times \R^d$ with respect to the map $\tbf$, for which we follow the approach of~\cite{ABC1}. Let $\{\mu_t\} = \{\mu_t: t \in \R \}$ be a family of finite (vector) measures on $\R \times \R^d$. We say that such a family is \textbf{Borel} if
\[
t \mapsto \int_{\R \times \R^d} \phi \; \dd \mu_t
\]
is Borel for every test function $\phi \in \Crm_c(\R \times \R^d)$. 
Given a measure $\nu$ on $\R$ and 
a family $\{\rho_t: t \in \R\}$ of measures on $\R^d$ such that 
\begin{equation*}
	\int_\R |\rho_t|(\R^d) \; \dd \nu(t) < \infty, 
\end{equation*}
we define the generalised product $\nu \otimes \rho_t$ as  the measure on $\R \times\R^d$ such that 
\begin{equation*}
	\int_{\R\times \R^d} \phi(t,x) \; \dd (\nu\otimes\rho_t)(t,x) = \int_\R\left( \int_{\R^d} \phi(t,x) \; \dd\rho_t(x) \right)\; \dd\nu(t)  
\end{equation*}
for every $\phi \in\Crm_c(\R^d)$. 

Let now $\mu$ be a (possibly vector-valued) measure in $\R \times \R^d$ and let $\nu$ be a measure on $\R$ such that $\tbf_{\#}\mu \ll \nu$. Then, there exists a Borel family $\{\mu_t: t \in \R\}$ of (possibly vector-valued) measures on $\R \times \R^d$ such that: 
\begin{enumerate}
	\item[\rm(i)] $\mu_t$ is supported on $\{t\} \times \R^d$ for $\nu$-a.e.\ $t \in \R$; 
	\item[\rm(ii)] $\mu$ can be decomposed as 
	\[
	\mu = \int_\R \mu_t \; \dd \nu(t),
	\]
	which means 
	\begin{equation}\label{eq:def_disint}
		\mu(A) = \int_\R \mu_t(A) \; \dd \nu(t),
	\end{equation}
	for every Borel set $A \subset \R \times \R^d$. 
\end{enumerate}
Any family $\{\mu_t\}$ satisfying the conditions {\rm(i)} and {\rm(ii)} above will be called a \textbf{disintegration of $\mu$ with respect to $\tbf$ and $\nu$}. We remark that, from~\eqref{eq:def_disint} we obtain
\[
\int_{\R \times \R^d} \phi \; \dd\mu  = \int_\R \left(\int_{\{t\}\times \R^d} \phi \; \dd \mu_t \right)\; \dd \nu(t)
\]
for every Borel function $\phi \colon \R \times \R^d \to [0,+\infty]$. 

In the context above, if $\mu = \nu \otimes \rho_t$ then it also holds that $|\mu| = \nu \otimes |\rho_t|$ (see, e.g.,~\cite[Corollary 2.29]{AFP}).

In what follows, most of the times $\nu$ will be taken to be the Lebesgue measure. The existence and uniqueness of the disintegration in the case where $\nu=\tbf_{\#}\mu$ is a standard result in measure theory. For the version cited above, we refer the reader to~\cite{ABC1} and to the references listed therein.

\begin{remark}\label{rem:disint_explicit} We point out that, with the notations above, the following formula characterises the measures $\mu_t$: For $\nu$-a.e.\ $t \in \R$
    \begin{equation}\label{eq:explicit_disintegration}
        \mu_t = \wslim_{h \to 0} \frac{\mu \restrict {(t-h,t+h)\times \R^d}}{\nu((t-h,t+h))} \qquad \text{in the sense of measures on $\R \times \R^d$.}
    \end{equation}
    Indeed, let us fix a countable dense set $\Dcal$ of functions $\phi \in \Crm_c(\R\times \R^d)$. For every $\phi \in\Dcal$ the function
    \[
    t \mapsto \langle \mu_t, \phi \rangle = \int_{\R \times \R^d} \phi \; \dd \mu_t
    \]
    belongs to ${\rm L}^1(\nu)$, and it is thus approximately continuous for $\nu$-a.e.\ $t$. This entails the existence of $N \subset \R$, $\nu(N)=0$ such that for every $t\in \R\setminus N$
	\begin{align}
	\langle  \mu_t,\phi \rangle  & = \lim_{h\to 0}\frac{1}{\nu((t-h,t+h))} \int_{t-h}^{t+h} \langle \mu_s,\phi \rangle \; \dd\nu(s) \notag\\ 
	& = \lim_{h\to 0} \frac{1}{\nu((t-h,t+h))} \int_{(t-h,t+h) \times \R^d} \phi \; \dd\mu \label{eq:coincidence} 
	\end{align}
	for all $\phi\in\Dcal$. By definition of $\nu$, the family of measures 
	\[
	\frac{\mu \restrict {(t-h,t+h)\times \R^d}}{\nu((t-h,t+h))}
	\]
	in~\eqref{eq:explicit_disintegration} has total mass uniformly bounded by $1$, and therefore admits a converging subsequence. It is readily seen that the limit must be $\mu_t$, since by~\eqref{eq:coincidence} they agree when evaluated on all $\phi\in\Dcal$ and consequently the limit in~\eqref{eq:explicit_disintegration} exists as $h\to 0$.
\end{remark}

\section{The geometric transport equation}

Throughout this work we will consider the time interval to be $[0,1]$, but all results hold, with due modifications, for a general interval. Given a path of currents $t \mapsto \T_t \in \mathrm{N}_k(\R^d)$ ($t \in [0,1]$) and a vector field $b \colon [0,1] \times \R^d \to \R^d$, the \textbf{geometric transport equation} reads as
\begin{equation}\label{eq:PDE} \tag{GTE}
    \frac{\dd}{\dd t} \T_t + \Lcal_{b_t} \T_t = 0, 
\end{equation}
where the Lie derivative $\Lcal_{b_t} \T_t$ is defined in the next section. As we shall see in a moment, in order to make sense of~\eqref{eq:PDE}, it would suffice to have $b_t$ defined for $\L^1$-a.e.\ $t \in \R$ and for $(\|T_t\|+\|\partial T_t\|)$-a.e.\ $x \in \R^d$. However, since $T_t$ will be unknown, we will assume that $b$ is an everywhere-defined Borel measurable map.

\subsection{The Lie derivative of a current}

Let $b \colon \R^d \to \R^d$ be a smooth globally-bounded \emph{autonomous} (i.e., time-independent) vector field. We recall that, under these assumptions, there exists a unique flow of $b$, that is, a  family of $\Crm^\infty$-maps $\Phi \colon \R \times \R^d \to \R^d$ such that 
\[
  \left\{
  \begin{aligned}
	\partial_t \Phi(t,x) &= b(\Phi(t,x)) \\
	\Phi(0,x) &= x,
  \end{aligned}
  \right. 
  \qquad
  (t,x) \in \R \times \R^d.
\]
We will often write $\Phi_t(x) := \Phi(t,x)$ and likewise for other time-dependent quantities. Standard facts ensure that the family $\{\Phi_t : t \in \R\}$ is a family of $\Crm^\infty$-diffeomorphisms of $\R^d$ to itself. It is worth noting that the family $\{\Phi_t : t \in \R\}$ induces a one-parameter semigroup: 
\begin{equation}\label{eq:semigroup}
    \Phi_t \circ \Phi_{s} = \Phi_{t+s} = \Phi_s \circ \Phi_t
    \qquad \text{for every $s,t \in \R$}. 
\end{equation}

We recall that the \textbf{Lie derivative} of a smooth $k$-form $\omega \in \Dscr^k(\R^d)$ with respect to the vector field $b$ is the smooth $k$-form
\[
  \Lcal_{b} \omega := \lim_{h \to 0} \frac{(\Phi_h)^* \omega - \omega}{h}. 
\]
Here the convergence is to be understood pointwise or in $\Crm^\infty$. The following identity, which is sometimes taken as definition of the Lie derivative, is called \textbf{Cartan's formula}:
\begin{equation}\label{eq:cartan_magic}
	\Lcal_{b}\omega=d(i_{b}\omega) + i_{b} (d\omega),
\end{equation}
where $i_{b}\omega$ denotes the contraction of $\omega$ with the vector field $b$ defined in~\eqref{eq:contraction}.

By duality, we now define the \textbf{Lie derivative of a current} $T\in \Dscr_k(\R^d)$ with respect to a smooth vector field $b$ via
\[
 \langle \Lcal_{b} \T, \omega \rangle := - \left \langle \T, \Lcal_{b} \omega \right \rangle, 
    \qquad \omega \in \Dscr^k(\R^d).
\]
By~\eqref{eq:cartan_magic} we have  
\begin{equation*}
    \langle \Lcal_{b} \T, \omega \rangle = - \left \langle \T, \Lcal_{b} \omega \right \rangle = - \langle \T, d(i_{b}\omega) + i_{b} (d\omega) \rangle = - \langle b \wedge \partial \T + \partial (b \wedge \T), \omega \rangle,  
\end{equation*}
and this gives the analogous \textbf{Cartan identity for currents}:
\[
\Lcal_{b}\T=-b\wedge \partial \T-\partial(b\wedge \T).
\]

In the following we will also work with \emph{non-autonomous} vector fields, i.e., $b \colon [0,1] \times \R^d \to \R^d$. In this case we write $b_t(x):=b(t,x)$ and we define the Lie derivative of $T$ with respect to $b_t$ as  
\[
\Lcal_{b_t}T:=-b_t\wedge \partial T-\partial(b_t\wedge T).
\]

\subsection{Equivalent formulations of the PDE}

In the following we will often work with \textbf{paths of currents} $t \mapsto T_t \in \Dscr_k(\R^d)$, $t \in [0,1]$. We will always tacitly assume the weak* measurability of the path, namely, for every $\omega\in\Dscr^k(\R^d)$ the map $t \mapsto \langle T_t,\omega \rangle$ is assumed to be Lebesgue-measurable.

The following lemma offers some equivalent conditions to the triviality of the path $t \mapsto T_t$ and will be used quite often throughout the work. 

\begin{lemma}\label{lemma:equivalent_being_zero}
Let $t \mapsto T_t$ ($t \in [0,1]$) be a path of currents with 
\begin{equation}\label{eq:mass_control}
\int_{K} \Mbf(T_t) \; \dd t < \infty, \qquad  K \Subset \R. 
\end{equation}
The following statements are equivalent: 
\begin{enumerate}
    \item[\rm{(i)}] $T_t = 0$ for $\L^1$-a.e.\ $t \in \R$.
    \item[\rm{(ii)}] For every $\eta \in \Dscr^k([0,1] \times \R^d)$, where $\iota_t$ denotes the embedding $x \mapsto (t,x)$, it holds that
    \[
    \int_0^1 \langle \delta_t \times T_t, \eta \rangle \; \dd t = \int_0^1 \langle T_t, \iota_{t}^* \eta \rangle \; \dd t = 0.
    \]
    
     \item[\rm{(iii)}] For every $\omega \in \Dscr^{k}(\R^d)$ and for every function $\phi \in \Crm_c^\infty(\R)$, it holds that
    \[
    \int_0^1 \phi(t) \langle T_t, \omega \rangle \; \dd t= 0.
    \]
     \end{enumerate}
Moreover, in {\rm(ii)} and {\rm(iii)} we may replace the spaces $\Dscr^k([0,1] \times \R^d)$, $\Dscr^k(\R^d)$ and $\Crm_c^\infty(\R)$ with countable dense subsets thereof. 
\end{lemma}

\begin{proof} It is clear that $ {\rm(i)} \implies {\rm(ii)} \implies  {\rm(iii)}$, therefore it suffices to show $ {\rm(iii)} \implies  {\rm(i)}$ and we prove directly the version where the condition in {\rm(iii)} holds for a countable dense set of $\omega$'s and $\phi$'s. Let indeed $\Dcal = \{\omega_n\}_{n \in \N}$ be a countable dense set in $\Dscr^k(\R^d)$. For each $n$, the map $t \mapsto \langle T_t,\omega_n \rangle$ is locally integrable by~\eqref{eq:mass_control}, and by assumption its integral against a countable dense set of test functions vanishes. Therefore, for each $n$, it holds that $\langle T_t,\omega_n \rangle = 0$ for $\L^1$-a.e.\ $t$, that is, we can find a full measure set $L_n \subset \R$ such that $\langle T_t,\omega_n\rangle=0$ for each $t \in L_n$. For every $t$ belonging to the full measure set $L:=\bigcap_n L_n$ we thus have 
\[
\langle T_t,\omega_n \rangle = 0 \qquad\text{for every $n \in\N$.}
\]
Therefore, by density, we conclude that {\rm(i)} holds.
\end{proof}

\begin{remark}\label{rmk:not_dt} We explicitly observe that in condition {\rm(ii)} in Lemma~\ref{lemma:equivalent_being_zero} it is equivalent to consider forms $\eta \in \Dscr^k([0,1] \times \R^d)$ of the kind $\eta(t,x) = \alpha(t)\beta(x)$, where $\alpha \in \Crm_{c}^\infty(\R)$ and $\beta \in \Dscr^k(\R^d)$ is given by $\beta = \sum_{I} \beta_{I} dx^{I}$ with $I$ a multi-index of order $k$ just in the spatial variables. In other words, it is enough to verify the condition on forms without terms involving $dt$, because $\iota_{t}^*\eta$ vanishes if $\eta$, written in coordinates $dt, dx^1, \ldots, dx^{d}$, only contains terms involving $dt$. 
\end{remark}

In order to make sense of the geometric transport equation~\eqref{eq:PDE}, we will assume that 
\begin{equation}\label{eq:uniform_bound_normal_mass}
\sup_{t\in[0,1]}\big[\Mbold(T_t)+\Mbold (\partial T_t)\big]<\infty 
\end{equation}
and that the following integrability condition for $b_t$ holds:
\begin{equation}\label{eq:basic_integrability_b}
\int_0^1 \int_{\R^d} |b_t(x)|\; \dd(\|T_t\|+\|\partial T_t\|)(x)\; \dd t<+\infty.
\end{equation}
We also set
\[
B(t):=\int_{\R^d}|b_t(x)|\; \dd(\|T_t\|+\|\partial T_t\|)(x).
\]

The following lemma explains the way in which we understand the equation~\eqref{eq:PDE}. 

\begin{lemma}\label{lemma:equiv_def_of_solutions} Suppose $t \mapsto T_t$ is a path of normal currents satisfying~\eqref{eq:uniform_bound_normal_mass} and let $b$ be a vector field satisfying~\eqref{eq:basic_integrability_b}. The following conditions are equivalent:
\begin{enumerate}
    \item[\rm{(i)}] For every $\psi\in \Crm^\infty_c((0,1))$ and every $\omega \in \mathscr{D}^k(\R^d)$ it holds that  
    \[
    \int_0^1 \langle \T_t,\omega\rangle \psi'(t)-\langle \Lcal_{b_t}\T_t,\omega\rangle \psi(t) \; \dd t=0. %-\langle T_0,\omega\rangle\psi(0).
    \]
    
     \item[\rm{(ii)}] For every $\omega \in \mathscr{D}^k(\R^d)$ the map $t \mapsto \langle T_t,\omega \rangle$ is absolutely continuous and the following equality holds for $\L^1$-a.e.\ $t \in (0,1)$: 
     \[
     \frac{\dd}{\dd t} \langle T_t,\omega \rangle = -\langle \Lcal_{b_t} T_t, \omega \rangle. 
     \]
\end{enumerate}
In addition, if $b$ is autonomous and smooth, {\rm(i)} and {\rm(ii)} are equivalent to:
\begin{enumerate}
	\item[\rm{(iii)}] For every $\eta \in \Dscr^{k}((0,1) \times \R^d)$ it holds that
\[
\int_{0}^1 \langle \delta_t \times T_t, \partial_t \eta + \Lcal_{b} \eta \rangle \; \dd t = 0. 
\]
\end{enumerate}
\end{lemma}

\begin{proof} We split the proof into several steps.

${\rm(i)} \implies  {\rm(ii)}$. Let us begin by showing that, for every fixed $\omega \in \Dscr^k(\R^d)$, the function $t\mapsto \langle \T_t,\omega\rangle$ is absolutely continuous. Indeed, by {\rm(i)} its distributional time derivative is 
\[ 
\frac{\dd}{\dd t}\langle \T_t,\omega\rangle = - \langle \Lcal_{b_t}\T_t,\omega\rangle=\langle b_t \wedge \T_t,d\omega\rangle +\langle b_t\wedge \partial \T_t,\omega\rangle
\]
and
\begin{align*}
\int_0^1 |\langle \Lcal_{b_t} \T_t,\omega\rangle| \; \dd t
&\leq \|d\omega\|_\infty \int_0^1\int_{\R^d} |b_t(x)|\; \dd\|T_t\|(x) \; \dd t+\|\omega\|_\infty \int_0^1 \int_{\R^d}|b_t(x)|\; \dd\|\partial \T_t\|(x)\; \dd t \\
&<\infty
\end{align*}
by~\eqref{eq:basic_integrability_b}. This yields the conclusion.

${\rm(ii)} \implies  {\rm(i)}$. This follows from the fact that, for an absolutely continuous function, the distributional derivative coincides with the pointwise one. 

${\rm(iii)} \implies {\rm(i)}$. It is enough to choose $\eta(t,x) = \psi(t) \omega(x)$. 

${\rm(i)} \implies  {\rm(iii)}$. By adapting the argument in~\cite[4.1.8]{Federer69book}, one can prove that the forms of the kind $\psi\omega:=(\tbf^*\psi)\wedge (\pbf^*\omega)$, with $\psi \in \Crm_c^\infty((0,1))$ and $\omega \in \Dscr^k(\R^d)$, generate a vector space which is  dense in the space of $k$-forms in $\R\times \R^d$ and that, when written in coordinates $dt,dx^1,\ldots,dx^d$, do not contain terms involving $dt$. The conclusion follows from Lemma~\ref{lemma:equivalent_being_zero}, taking into account Remark~\ref{rmk:not_dt}. 
\end{proof}

\begin{definition}\label{def:sol}
Let $t \mapsto T_t \in \Nrm_k(\R^d)$ be a path of normal currents satisfying~\eqref{eq:uniform_bound_normal_mass} and let $(b_t)_t$ be a family of (everywhere-defined) Borel vector fields satisfying~\eqref{eq:basic_integrability_b}. If $t \mapsto T_t$ satisfies one of the first two equivalent conditions listed in Lemma~\ref{lemma:equiv_def_of_solutions} then it is called a \textbf{(weak) solution} to the geometric transport equation~\eqref{eq:PDE}.
\end{definition}

The following lemma concerns the existence of a continuous representative.

\begin{lemma} \label{lemma:continuous_representative} Let $t \mapsto T_t \in \Nrm_k(\R^d)$ satisfy~\eqref{eq:uniform_bound_normal_mass} and let $(b_t)_t$ be a family of Borel vector fields satisfying~\eqref{eq:basic_integrability_b} such that $T_t$ is a weak solution to~\eqref{eq:PDE}. Then:
\begin{enumerate}[(i)]
    \item There exists a weakly*-continuous path $t \in [0,1] \mapsto \tilde{T}_t \in \mathrm{N}_k(\R^d)$ such that $T_t = \tilde{T}_t$ for a.e.\ $t \in [0,1]$.
    \item The path from {\rm(i)} is absolutely continuous as a map from $[0,1]$ to $\mathrm{N}_k(\R^d)$ equipped with the flat norm $\Fbf$.
\end{enumerate}
If, in addition, $B(t)$ is uniformly bounded on $[0,1]$ (which follows, for instance, if the field $b$ is uniformly bounded) then $t \mapsto \tilde{T}_t $ is Lipschitz with respect to $\Fbf$.
\end{lemma} 

Here, $t \in [0,1] \mapsto \tilde{T}_t \in \mathrm{N}_k(\R^d)$ is called \textbf{weakly*-continuous} if for every $t \in [0,1]$ and for every sequence $t_j \to t$ then $\tilde{T}_{t_j} \weaksto \tilde{T}_t$ in the sense of currents over $\R^d$.

\begin{proof} By {\rm(ii)} in Lemma~\ref{lemma:equiv_def_of_solutions}, the function $t\mapsto \langle \T_t,\omega\rangle$ is absolutely continuous. Let now $\Dcal = \{\omega_n\}_{n \in \N}$ be a countable dense set in $\Dscr^k(\R^d)$. For each $n$, we can find a set $L_n \subset [0,1]$ such that $t \mapsto \langle T_t,\omega_n\rangle$ is uniformly continuous on $L_n$ and $\L^1([0,1] \setminus L_n)=0$. Set $L:=\bigcap_n L_n$. The restriction $L\ni t  \mapsto T_t$ is a uniformly weakly*-continuous path because for each $\omega \in \Dcal$ and all $s, t \in [0,1]$ with $s < t$,
\begin{equation*}
    |\langle \T_t-\T_s,\omega\rangle| \leq \int_s^t \frac{\dd}{\dd r}\langle T_r,\omega\rangle \; \dd r\le (\|\omega\|_\infty + \|d\omega\|_\infty)\int_s^t B(r)\; \dd r ,
\end{equation*}
and this extends to every $\omega\in\mathscr{D}^k(\R^d)$ by density.  Moreover, recalling~\eqref{eq:flat_norm_dual}, the same inequality above shows that
\begin{equation}\label{eq:cauchy}
\Fbold(T_t-T_s)\leq 2\int_s^t B(r)\; \dd r\qquad\text{for $s,t\in L$.}
\end{equation}
This proves that $L\ni t\mapsto T_t$ is indeed absolutely continuous (or Lipschitz if $B(r)$ is uniformly bounded) with respect to the flat norm $\Fbf$. It remains to show that we can extend this map to the whole interval $[0,1]$. 
Let $t$ be any point in $[0,1]$, and let us choose any sequence $t_i\to t$ with $t_i\in L$. Then, $T_{t_i}$ is a Cauchy sequence by~\eqref{eq:cauchy}, hence $T_{t_i}\weaksto \tilde{T}_t$ for some $\tilde{T}_t\in\Nrm_k(\R^d)$, and $\tilde{T}_t$ does not depend on the sequence $t_i$. In this way we can extend $t\mapsto T_t$ to a weakly*-continuous map $t\mapsto \tilde T_t$ on $[0,1]$. This proves {\rm(i)}.

To establish {\rm(ii)}, let now $s,t\in[0,1]$ with $s_i,t_i\in L$ and $s_i\to s$, $t_i\to t$. Then,
\begin{align*}
\Fbf(\tilde{T}_t-\tilde{T}_s)&\leq \Fbf(\tilde{T}_t-T_{t_i})+\Fbf(T_{t_i}-T_{s_i})+\Fbf(T_{s_i}-\tilde{T}_s)\\
& \leq \Fbf(\tilde{T}_t-T_{t_i})+2\int_{s_i}^{t_i} B(r)\; \dd r+\Fbf(T_{s_i}-\tilde{T}_s)
\end{align*}
and sending $t_i\to t$ and $s_i\to s$ we obtain that~\eqref{eq:cauchy} holds for $t\mapsto\tilde T_t$ for all $s,t\in[0,1]$ with $s < t$.
\end{proof}

As a consequence of Lemma~\ref{lemma:continuous_representative} we obtain that, for vector fields satisfying~\eqref{eq:basic_integrability_b}, one may complement the PDE~\eqref{eq:PDE} with an initial condition: Given $\overline{T} \in \mathrm{N}_{k}(\R^d)$ we can consider the initial-value problem
\begin{equation}\label{eq:initial_value_problem}
  \left\{ \begin{aligned}
    \frac{\dd}{\dd t} \T_t + \Lcal_{b_t} \T_t &= 0, \\ 
    \T_{0} &= \overline{T},
  \end{aligned} \right.
\end{equation}
where the second condition is to be understood by employing the continuous representative of $t \mapsto \T_t$. Equivalently, one can understand~\eqref{eq:initial_value_problem} in the following way: For every $\omega \in \Dscr^k(\R^d)$ and for every $\psi \in C_c^\infty([0,1))$ it must hold that 
\[
    \int_0^1 \langle \T_t,\omega\rangle \psi'(t) - \langle \Lcal_{b_t}\T_t, \omega\rangle \psi(t) \; \dd t = -\langle  \overline{T},\omega\rangle\psi(0).
\]
Observe that in general, if $\psi \in C_c^\infty([0,1])$ and $\omega \in \Dscr^k(\R^d)$, it holds that 
\begin{equation}\label{eq:non_compact_support}
    \int_0^1 \langle \T_t,\omega\rangle \psi'(t)-\langle \Lcal_{b_t}\T_t,\omega\rangle \psi(t) \; \dd t
    = \langle  \tilde{T}_1, \omega\rangle\psi(1) - \langle  \tilde{T}_0,\omega\rangle\psi(0),
\end{equation}
where $t \mapsto \tilde{T}_t$ denotes the continuous representative defined in Lemma~\ref{lemma:continuous_representative}. Similarly, if $b$ is smooth and autonomous, for every $\eta \in C^{\infty}([0,1] \times \R^d; \Wedge^k(\R \times \R^d))$ it holds that 
\begin{equation}\label{eq:with_boundary_terms}
    \int_{0}^1 \langle \delta_t \times T_t, \partial_t \eta + \Lcal_{b} \eta \rangle \; \dd t = \langle \delta_1 \times T_1, \eta \rangle - \langle \delta_0 \times T_0, \eta \rangle = \langle T_1, \iota_{1}^*\eta \rangle - \langle T_0, \iota_0^* \eta \rangle . 
\end{equation}

\subsection{Existence and uniqueness in the smooth framework}

When our vector field $b$ is autonomous and sufficiently regular it is straightforward to see that we have existence and uniqueness of solutions to~\eqref{eq:initial_value_problem}:

\begin{theorem}[Existence \& uniqueness] \label{thm:smooth}
Let $b \colon \R^d \to \R^d$ be a globally bounded, $\Crm^\infty$ vector field and let $\Phi_t$ be its (unique) flow. Let $\overline{T} \in \Nrm_k(\R^d)$ be a given initial value. Then, the path of normal currents defined by
\begin{equation*}
    t \mapsto \T_t := (\Phi_t)_* \overline{T}
\end{equation*}
is a weak solution to~\eqref{eq:initial_value_problem}. Moreover, the solution is unique in the class of normal currents.
\end{theorem}

\begin{proof}
We first prove the existence statement, for which we notice that the path of currents $t \mapsto T_t$ is indeed a path of normal currents and the bound~\eqref{eq:uniform_bound_normal_mass} is satisfied by~\eqref{eq:mass_of_pushforward}. The bound~\eqref{eq:basic_integrability_b} follows as a consequence of the global boundedness of $b$. 
To show the claim, in view of Lemma~\ref{lemma:equiv_def_of_solutions}, it therefore suffices to compute the pointwise derivative of the map  $t \mapsto \langle T_t, \omega \rangle$, where $\omega \in \Dscr^k(\R^d)$ is a fixed smooth $k$-form. From~\eqref{eq:mass_of_pushforward} we see that this map is absolutely continuous (even Lipschitz): For $s < t$ it holds that
\begin{align*}
|\langle T_t, \omega \rangle - \langle T_s, \omega\rangle| & =\bigl|\bigl\langle  (\Phi_t)_* \overline{T} -  (\Phi_s)_* \overline{T}, \omega \bigr\rangle\bigr| \\ 
& =\bigl|\bigl\langle  (\Phi_t - \Phi_s)_* \overline{T}, \omega \bigr\rangle\bigr|\\ 
& \le \|\omega\|_{\infty} \int_{\R^d} \bigl\|\Wedge^k d(\Phi_t - \Phi_s)(x)\, \vec{\overline{T}}(x) \bigr\| \; \dd\|\overline{T}\|(x)\\
& \le C\Mbf(\overline{T}) \|\omega\|_{\infty} \|d(\Phi_t - \Phi_s))\|_{\infty}^k  \\ 
& \le C|t-s|   
\end{align*}
since $t \mapsto \Phi_t$ is Lipschitz with values in $\Crm^1$ by classical results. On the other hand, by the regularity of the flow map again,
\begin{align*}
    \langle \Lcal_b T_t, \omega\rangle & = - \langle T_t, \Lcal_b \omega\rangle \\ 
    & = - \left \langle T_t, \lim_{h \to 0} \frac{(\Phi_h)^* \omega - \omega}{h} \right \rangle \\ & = - \left \langle (\Phi_t)_* \overline{T}, \lim_{h \to 0} \frac{(\Phi_h)^* \omega - \omega}{h} \right \rangle \\
    & = - \lim_{h \to 0} \frac{1}{h} \left \langle \overline{T}, (\Phi_{t+h})^{*} \omega - (\Phi_t)^{*} \omega \right \rangle \\
    & = - \frac{\dd}{\dd t} \langle \overline{T}, (\Phi_t)^{*}\omega\rangle\\
    & =- \frac{\dd}{\dd t} \langle T_t, \omega\rangle , 
\end{align*}
where we have used the semigroup property~\eqref{eq:semigroup}. The validity of the initial condition is obvious and therefore the proof of existence is complete. 

For the uniqueness part of the theorem, we will employ a duality strategy, which is inspired by~\cite[Proposition 8.1.7]{AGS} or~\cite[Theorem II.6]{DPL}. By the linearity of the problem, it is enough to show that if $\overline{T} = 0$ then for the solution $t \mapsto T_t$ it holds that $T_t=0$ for $\L^1$-a.e.\ $t \in (0,1)$.  Let us fix $\eta \in \Dscr^{k}([0,1] \times \R^d)$. By Lemma~\ref{lemma:equivalent_being_zero}, our claim is equivalent to \[
\int_0^1 \langle \delta_t \times T_t, \eta \rangle \; \dd t = 0. 
\]
Consider the following family of
forms on $\R^d$ indexed by $t \in (0,1)$: 
\[
\omega_t := -\int_{t}^1 (\Phi_{s-t})^*\eta_s \; \dd s,
\]
where $\eta_s := \iota_s^{*}\eta \in \Dscr^k(\R^d)$ is the pullback via the immersion $\iota_s \colon x \mapsto (s,x)$. We will think of $\omega(t,x) := \omega_t(x)$ as a $k$-form defined on $(0,1) \times \R^d$ and a direct computation shows that 
\[
\partial_t \omega (t,x) = \eta(t,x) - \int_{t}^1 \frac{\dd}{\dd t}(\Phi_{s-t})^*\eta_s(x) \; \dd s, \qquad  (t,x) \in (0,1) \times \R^d.
\]
On the other hand, by linearity of the pullback and dominated convergence,
\begin{align*}
\Lcal_{b}\omega (t,x) & = \lim_{h \to 0} \frac{(\Phi_h)^{*}\omega_t - \omega_t}{h}(x) \\ 
& = - \int_{t}^1 \lim_{h \to 0} \frac{(\Phi_h)^{*} (\Phi_{s-t})^* \eta_s - (\Phi_{s-t})^{*} \eta_s}{h}(x)\; \dd s \\
& =\int_{t}^1 \frac{\dd}{\dd t}(\Phi_{s-t})^*\eta_s(x) \; \dd s,
\end{align*} 
whereby
\begin{align*}
    \partial_t \omega + \Lcal_{b}\omega = \eta \qquad \text{as $k$-forms on $[0,1] \times \R^d$}. 
\end{align*}
In addition, observe that $\omega_1 =0$. Hence,
\begin{align*}
\int_0^1 \langle \delta_t \times T_t, \eta \rangle \; \dd t & = \int_0^1 \langle \delta_t \times T_t, \partial_t \omega + \Lcal_{b}\omega \rangle \; \dd t = 0, 
\end{align*}
where we have used Lemma~\ref{lemma:equiv_def_of_solutions} and~\eqref{eq:with_boundary_terms} (observe that the boundary terms vanish: at $0$ because of $\overline{T}=0$, and at $1$ because of $\iota_1^* \omega = \omega_1=0$). Thus, $T_t=0$ for $\L^1$-a.e.\ $t \in (0,1)$ and this concludes the proof. 
\end{proof}

\subsection{Particular cases} \label{sc:particular}

We now consider the geometric transport equation~\eqref{eq:PDE} in some specific dimensions $k\in\{0,\ldots,d\}$ and see how it reduces to well-known PDEs.

\medskip

\textbf{$0$-currents:} In the case $k=0$, $T_t$ is a distribution with finite mass and can be identified with a measure $\mu_t$. Therefore,~\eqref{eq:PDE} becomes
	\[
	\frac{\dd}{\dd t} \mu_t - \partial(b_t \wedge\mu_t)=0,
	\]
	where $b_t \wedge\mu_t$ can be identified with the vector-valued measure $b_t\mu_t$.
	A generic test $0$-form is given by $\omega\in \Crm^\infty_c(\R^d)$, and for such a function, $d\omega=\sum_i \frac{\partial \omega}{\partial x_i}(x)\, d x_i$. Then,
	\[
	\langle - \partial(b_t \wedge\mu_t),\omega\rangle =\langle - b_t\wedge \mu_t,d \omega\rangle= - \int_{\R^d} b_t(x)\cdot \nabla \omega(x) \; \dd\mu_t(x).
	\]
	Hence, the PDE in coordinates reads as
	\[
	\frac{\dd}{\dd t} \mu_t+\dive(b_t \mu_t)=0,
	\]
	which is the \textbf{continuity equation}.

\medskip
	
\textbf{$1$-currents:} Assume that we have a family of normal 1-currents $t \mapsto T_t \in \Nrm_1(\R^d)$ in $\R^d$ satisfying
\begin{equation*}
\frac{\dd}{\dd t} {T_t} + \Lcal_{b_t} T_t = 0 
\end{equation*}
and suppose that the $T_t$ are diffuse, i.e., $\|T_t\| \ll \L^d$, so that they can be identified with a time-dependent (1-)vector field $\tau_t = (\tau_1, \ldots ,\tau_d)$ on $\R^d$. A straightforward computation yields that, for every smooth $1$-form $\omega = \sum_{i=1}^d \omega_i dx^i$, it holds that 
\begin{equation*}
- \langle \partial (b_t\wedge T), \omega \rangle = - \langle T, i_{b_t} (d\omega) \rangle = \int_{\R^d} \sum_{j=1}^d \dive(\tau_t \otimes b_t - b_t\otimes \tau_t)_j(x) \omega_j(x) \; \dd x
\end{equation*}
and 
\begin{equation*}
	- \langle b_t\wedge \partial T, \omega \rangle = \int_{\R^d}  \dive(\tau_t) \sum_{j=1}^d \ b_j(t,x) \cdot \omega_j(x) \; \dd x.
\end{equation*}
We conclude that~\eqref{eq:PDE} in this case reads as
\begin{equation*}\
	\frac{\dd}{\dd t}\tau_t + \dive(\tau_t \otimes b_t - b_t\otimes \tau_t) + b_t \dive \tau_t = 0. 
\end{equation*}
If we further assume that every $T_t$ is boundaryless (equivalently, $\dive\tau_t = 0$ for every $t$), we obtain 
\begin{equation*}\
	\frac{\dd}{\dd t}\tau_t + \dive(\tau_t \otimes b_t - b_t\otimes \tau_t) = 0,  
\end{equation*}
which features for instance in~\cite{Brenier2, Brenier1} in connection with the ``magnetic relaxation equations'' and some models of incompressible magnetohydrodynamics. We observe further that, in $\R^3$, by standard vector-calculus identities, it can be easily shown that for every pair of vector fields $u,w$ it holds that 
\[
\dive\left(u \otimes w - w \otimes u \right) = \curl(w \times u), 
\]
where $\times$ denotes the vector product. We conclude that, for absolutely continuous $1$-currents $\tau_t\L^3$ in $\R^3$, the PDE~\eqref{eq:PDE} can also be written as 
\[
\frac{\dd}{\dd t}\tau_t+\curl(b_t\times \tau_t)=0.
\]
which features, e.g., in the ``field dislocation mechanics'' developed in~\cite[Eq.~27a]{AA},~\cite[Eq.~1]{AT} and, implicitly, in~\cite{HudsonRindler21?, Rindler21a?, Rindler21b?}.

\medskip 

\textbf{$(d-1)$-currents:} We first recall a general duality between $(d-1)$-currents and $1$-forms (cf., e.g., the appendix to~\cite{HudsonRindler21?}): Given $\alpha \in \mathscr D^{1}(\R^d)$, we associate to it the current $T_\alpha \in \Dscr_{d-1}(\R^d)$ defined by 
\[
\langle T_\alpha, \omega \rangle := \int_{\R^d} \alpha \wedge \omega = \int_{\R^d} \langle \ee_1 \wedge \ldots \wedge \ee_d, \alpha \wedge \omega \rangle \; \dd x, \qquad  \omega \in \Dscr^{d-1}(\R^d),
\]
where $\ee_1, \ldots, \ee_d$ is the standard basis in $\R^d$. We note the following identities, which follow from Leibniz rule and elementary computations:
\begin{equation*}
    \partial T_\alpha = T_{d\alpha}, \qquad \partial (b \wedge T_\alpha) = \partial T_{i_b(\alpha)}, 
\end{equation*}
for any vector field $b\colon \R^d \to \R^d$. Notice that given any $(d-1)$-current $T=\tau \L^d$ with $\tau \in \Crm_c^\infty(\R^d; \Wedge_{d-1}(\R^d))$, it holds
\[
T=T_{\alpha},
\]
where $\alpha \in \Crm_c^\infty(\R^d; \Wedge^{1}(\R^d))$ is the form defined by 
\[
\alpha := \tau \intprod (dx^1 \wedge \ldots \wedge dx^d). 
\]

The following is a geometric interpretation of this duality: identifying $\alpha=\sum_i \alpha_i dx^i$ with the vector field $\vec{\alpha}(x)=(\alpha_1(x),\ldots,\alpha_d(x))^\top$, the orienting $(d-1)$-vector $\tau(x)$ of $T_\alpha=\tau \L^d$ (when written as a simple vector) spans the $(d-1)$-plane $\vec{\alpha}^\perp(x)$. In other words, $\vec{\alpha}$ can be identified with the Hodge dual of $\tau$, taking into account the standard inner product structure of $\R^d$.

Given a (smooth) path of boundaryless $(d-1)$-currents $t \mapsto T_t=T_{\alpha(t)}=T_\alpha$ (where for ease of notation we suppress the time index $t$ in the form $\alpha$) and given a smooth velocity field $b \colon \R^d \to \R^d$, with $b = (b^1, \ldots, b^d)$, we have
\[
\langle \Lcal_b T_t, \omega \rangle = - \langle  \partial (b \wedge T_t), \omega \rangle = - \langle T_{i_b(\alpha)}, d\omega \rangle.
\]
Writing $\alpha=\sum_i \alpha_i dx^i$ and $\omega = \sum_j \omega_j \widehat{dx^j}$, where we have used the abbreviation $\widehat{dx^j} := dx^1\wedge\cdots \wedge dx^{j-1} \wedge dx^{j+1} \wedge \cdots\wedge dx^d$, we further have
\[
d \omega = \sum_j (-1)^{j+1} \frac{\partial \omega_j}{\partial x_j} dx^1 \wedge \ldots \wedge dx^d 
\]
and 
\[
i_b(\alpha) = \sum_j \alpha_j b^j.
\]
Therefore,
\begin{align*}
\langle \Lcal_b T_t, \omega \rangle &=  - \langle T_{i_b(\alpha)}, d\omega \rangle \\
& = - \int_{\R^d} i_b(\alpha) \wedge d\omega \\ 
& = - \int_{\R^d} (\vec{\alpha}\cdot b) \, \sum_j (-1)^{j+1} \frac{\partial \omega_j}{\partial x_j} \; \dd x \\ 
& = \int_{\R^d} \sum_j \frac{\partial}{\partial x_j} (\vec{\alpha}\cdot b) \,  (-1)^{j+1} \omega_j \; \dd x.
\end{align*} 
On the other hand, 
\[
\frac{\dd}{\dd t} \langle T_{\alpha}, \omega \rangle = \frac{\dd}{\dd t} \int_{\R^d} \alpha \wedge \omega = \frac{\dd}{\dd t} \int_{\R^d} \sum_{j} \alpha_j \omega_j (-1)^{j+1} \; \dd x ,
\]
so that the coordinate expression of~\eqref{eq:PDE} reads as
\[
\frac{\dd}{\dd t} \int_{\R^d} \sum_{j} \alpha_j \omega_j (-1)^{j+1} \; \dd x + \int_{\R^d} \sum_j \frac{\partial}{\partial x_j} (\vec{\alpha}\cdot b) \,  (-1)^{j+1} \omega_j \; \dd x = 0,  \qquad \omega\in\Dscr^{d-1}(\R^d).
\]
This is the weak formulation of the following system of conservation laws
\[
\frac{\dd}{\dd t} \vec{\alpha} + \nabla (\vec{\alpha} \cdot b) = 0. 
\]

\medskip
	
\textbf{$d$-currents:} Let us now consider a family of top-dimensional normal currents $T_t$, which in coordinates take the form $T_t = u_t(x) \, \ee_1 \wedge \ldots \wedge \ee_d \, \L^d$ for some scalar function $u_t(x)$, which will be assumed smooth and compactly supported. Then, equation~\eqref{eq:PDE} becomes
	\[
	\frac{\dd}{\dd t}T_t - b_t\wedge \partial T_t=0.
	\]
	We need to test this equation with a generic $d$-form $\omega(x)=\bar\omega(x) d x_1\wedge\ldots\wedge d x_d$. We compute
	\[
	i_{b_t} \omega(t,x)=\sum_j (-1)^{j-1} b_t^j(x)\bar\omega(x) \, \widehat{dx^j},
	\]
	where we have again used the abbreviation $\widehat{dx^j} := dx^1\wedge\cdots \wedge dx^{j-1} \wedge dx^{j+1} \wedge \cdots\wedge dx^d$.
	Then,
	\begin{align*}
	\langle b_t\wedge \partial T_t,\omega\rangle&=\langle  T_t, d(i_{b_t} \omega)\rangle\\
	&=\int_{\R^d} u_t(x)\sum_j (-1)^{j-1}\frac{\partial}{\partial x_j}\big((-1)^{j-1}b_t^j(x)\bar \omega (x)\big) \; \dd x\\
	&=-\int_{\R^d}\nabla u_t(x)\cdot b_t(x) \bar\omega(x) \; \dd x. 
	\end{align*}
	Therefore, in this case the equation reduces to
	\[
	\frac{\dd}{\dd t}u_t+ b_t  \cdot \nabla u_t=0,
	\]
	which is the \textbf{transport equation}.

\begin{remark}
One could also consider the transport and continuity equations as special cases of the general transport equation for differential $k$-forms $\omega(t,\frarg) = \omega_t(\frarg)$,
\[
  \frac{\dd}{\dd t} \omega_t + \Lcal_{b} \omega_t = 0.
\]
Then we would again obtain the transport equation as the special case of $0$-forms (scalars) and the continuity equation as the special case of top-dimensional forms (assigning volumes) and this is essentially the point of view of~\cite{bouchut_slides}. However, the transport equation for $k$-forms requires more regularity on $\omega$. For instance, one can model continuously-distributed dislocation via the so-called Weitzenb\"{o}ck approach (see, e.g.,~\cite{EpsteinKupfermanMaor20} for an overview) and consider the torsion $2$-form of a material connection expressing a distribution of dislocations. However, this only makes sense in the smooth setting, and for singular torsions, most notably the torsion concentrated on discrete dislocation lines, one needs to adopt a dual approach, as detailed in the appendix to~\cite{HudsonRindler21?}, leading eventually to the transport of singular $1$-currents. The geometric transport equation then gives a rigorous interpretation of this ``transport of singular torsion''.
\end{remark}

\subsection{Product solutions}

For technical reasons we also need to consider products of solutions. Denote by $\qbf \colon \R^\ell \times \R^m \to \R^\ell$ and $\rbf \colon \R^\ell \times \R^m \to \R^m$ the projections onto the first and second factors. We write a generic point in $\R^\ell\times \R^m$ as $(x,y)$, where $x$ denotes the coordinates in $\R^\ell$ and $y$ the coordinates in $\R^m$. 

\begin{proposition} \label{prop:cartesian_product}
Let $t \mapsto T_t \in \mathscr{D}_{k}(\R^\ell)$, $t \in [0,1]$, be a path of currents with $\partial T_t=0$ such that there exists a vector field $v \in {\rm L}^1(\L^1(\dd t)\otimes \|T_t\|; \R^\ell)$ with 
 \[
 \frac{\dd}{\dd t} T_t + \Lcal_{v_t}T_t = 0.
 \]
Let $t \mapsto U_t \in \mathscr{D}_{h}(\R^m)$ be a path of currents with $\partial U_t =0$ such that there exists a vector field $w \in {\rm L}^1(\L^1(\dd t)\otimes \|T_t\|; \R^m)$ with
 \[
 \frac{\dd}{\dd t} U_t + \Lcal_{w_t} U_t = 0.
 \]
Assume that there is a constant $C > 0$ such that
 \[
 \Mbf(T_t)+\Mbf(U_t)\le C, \qquad  t\in[0,1].
 \]
 Then, the path $t \mapsto P_t$ of the product currents $P_t:=T_t \times U_t \in \mathscr{D}_{k+h}(\R^\ell \times \R^m)$ solves
 \[
 \frac{\dd}{\dd t} P_t + \Lcal_{B_t} P_t = 0,
 \]
where the vector field $B_t \colon \R^\ell\times \R^m \to \R^\ell\times \R^m$ is defined as $B_t := (v_t \circ \qbf,  w_t \circ \rbf)$, i.e., 
\[
B_t(x,y) = (v_t(x), w_t(y)), \qquad (x,y) \in \R^\ell \times \R^m.  
\]
\end{proposition}

\begin{proof} We first show that $B \in {\rm L}^1(\L^1(\dd t) \otimes \|P_t\|)$:
\begin{align*}
    &\int_0^1 \int_{\R^\ell \times \R^m} |(v_t(x), w_t(y))| \; \dd \|P_t\|(x,y)\; \dd t \\
    &\qquad \le \int_0^1 \int_{\R^m} \int_{\R^\ell} |(v_t(x), w_t(y))| \; \dd \|T_t\|(x) \; \dd \|U_t\|(y) \; \dd t \\
    &\qquad \le \int_0^1 \int_{\R^m} \int_{\R^\ell} (|v_t(x)| +| w_t(y)|) \; \dd \|T_t\|(x) \; \dd \|U_t\|(y) \; \dd t\\
    &\qquad \le C\int_0^1 \int_{\R^\ell}|v_t(x)|\; \dd \|T_t\|(x) \; \dd t + C\int_0^1 \int_{\R^m}|w_t(x)|\; \dd \|U_t\|(x) \; \dd t. 
\end{align*}

By~\cite[4.1.8]{Federer69book} the forms of the kind $\qbf^*\alpha \wedge \rbf^{*}\beta$, where $\alpha \in \mathscr{D}^r(\R^\ell)$ and $\beta \in \Dscr^s(\R^m)$, with $r+s = h+k$, generate a vector space that is dense in $\Dscr^{h+k}(\R^\ell \times \R^m)$. Therefore, it is enough to prove that the equation is satisfied pointwise on these forms.  To this end, we will use the following identities:
\begin{equation}\label{eq:contraction_distributive}
    i_{B_t} (\alpha\wedge \beta)=(i_{B_t} \alpha)\wedge \beta+(-1)^r \alpha\wedge(i_{B_t} \beta),\qquad \alpha\in \Dscr^r(\R^d), \; \beta\in\Dscr^s(\R^d),
\end{equation}
and
\begin{align}
\begin{aligned}\label{eq:comm_p_and_contraction} 
    i_{B_t} (\qbf^{*}\alpha) &= \qbf^{*}(i_{v}\alpha),\qquad &&\alpha \in \Dscr^{r}(\R^\ell),\\
    i_{B_t} (\rbf^{*}\beta) &= \rbf^{*}(i_{w}\beta),\qquad &&\beta \in \Dscr^{s}(\R^m)
    \end{aligned}
\end{align}
Identity~\eqref{eq:contraction_distributive} follows from~\cite[Lemma~13.11(b)]{Lee13book}. Let us now prove~\eqref{eq:comm_p_and_contraction}. For every multi-index $I$ of order $r$ in $\{1, \ldots ,\ell+m\}$ and every $\omega \in \Dscr^{r+s-1}(\R^\ell \times \R^m)$, we have for every $x,y$:
\begin{align*}
     \langle \ee_I ,i_{B_t} (\qbf^{*}\omega)(x,y)\rangle & = \langle  B_t \wedge \ee_I,(\qbf^{*}\omega)(x,y) \rangle \\
     &  = \langle  (v \circ \qbf, 0) \wedge \ee_I + (0, w \circ \rbf) \wedge \ee_I ,(\qbf^{*}\omega)(x,y)\rangle \\ 
     &  = \langle (v \circ \qbf,0) \wedge \ee_I ,(\qbf^{*}\omega)(x,y)\rangle \\
     & = \langle \Wedge^r \qbf [(v \circ \qbf,0) \wedge \ee_I] , \omega(x)\rangle \\
     & = \langle  v \wedge \Wedge^{r-1}\qbf \,\ee_I ,\omega(x)\rangle \\
     & = \langle \Wedge^{r-1}\qbf \,\ee_I ,i_v\omega(x)\rangle \\
     & = \langle \ee_I ,(\qbf^{*} i_v\omega)(x,y)\rangle. 
\end{align*}
This proves~\eqref{eq:comm_p_and_contraction}.

Now, for every $\alpha \in \mathscr{D}^r(\R^\ell)$ and $\beta \in \Dscr^s(\R^m)$ we have 
\begin{align}
\langle \Lcal_{B_t} P_t, \qbf^*\alpha \wedge \rbf^{*}\beta \rangle & =  -\langle \partial (B_t \wedge P_t), \qbf^*\alpha\wedge \rbf^{*}\beta\rangle \nonumber\\ 
& = - \langle B_t \wedge P_t, \qbf^*(d\alpha) \wedge \rbf^{*}\beta + (-1)^{r}\qbf^{*}\alpha \wedge \rbf^{*}(d\beta)
\rangle\nonumber\\ 
& = - \langle P_t, i_{B_t}(\qbf^*(d\alpha) \wedge \rbf^{*}\beta) + (-1)^r i_{B_t}(\qbf^{*}\alpha \wedge \rbf^{*}(d\beta)) \rangle. \label{eq:last_line}
\end{align}
We study the terms separately. For the first one, applying~\eqref{eq:contraction_distributive} and~\eqref{eq:comm_p_and_contraction}, 
\begin{align*}
     i_{B_t}(\qbf^*(d\alpha) \wedge \rbf^{*}\beta) & = i_{B_t}(\qbf^*(d\alpha)) \wedge \rbf^{*}\beta + (-1)^{r+1} \qbf^*(d\alpha) \wedge i_{B_t}(\rbf^{*}\beta) \rangle \\  
     & = \qbf^*(i_{v_t} d\alpha) \wedge \rbf^{*}\beta + (-1)^{r+1} \qbf^*(d\alpha) \wedge \rbf^{*}(i_{w_t} \beta),
\end{align*}
while for the second one we have 
\begin{align*}
    (-1)^r i_{B_t}(\qbf^{*}\alpha \wedge \rbf^{*}(d\beta)) & = (-1)^r \left[ i_{B_t}(\qbf^*\alpha) \wedge \rbf^{*}d\beta + (-1)^r \qbf^*\alpha \wedge i_{B_t}(\rbf^{*}d\beta) \right] \\ 
    & = (-1)^r\qbf^*(i_{v_t} \alpha) \wedge \rbf^{*}d\beta + \qbf^*\alpha \wedge \rbf^{*}(i_{w_t} d\beta). 
\end{align*}
We now observe that $P_t$ vanishes on the form $\qbf^*(d\alpha) \wedge \rbf^{*}(i_{w_t} \beta)$: Indeed, if $r+1 \ne k$,
\[
\langle P_t, \qbf^*(d\alpha) \wedge \rbf^{*}(i_{w_t} \beta)\rangle = 0. 
\]
If instead $r+1 = k$, then 
\[
\langle P_t, \qbf^*(d\alpha) \wedge \rbf^{*}(i_{w_t} \beta)\rangle = \langle T_t, d\alpha \rangle \langle U_t, i_{w_t} \beta \rangle = 0,  
\]
because $\partial T_t = 0$. A similar computation shows that $P_t$ vanishes also on the form $\qbf^*(i_{v_t} \alpha)) \wedge \rbf^{*}d\beta$ and therefore from~\eqref{eq:last_line} we obtain 
\begin{align*}
\langle \Lcal_{B_t} P_t, \qbf^*\alpha \wedge \rbf^{*}\beta \rangle =  - \langle P_t, \qbf^*(i_{v_t} d\alpha) \wedge \rbf^{*}\beta + \qbf^*\alpha \wedge \rbf^*(i_{w_t} d\beta) \rangle. 
\end{align*}
Now, if $r \ne k$ then 
\[
\langle \Lcal_{B_t} P_t, \qbf^*\alpha \wedge \rbf^{*}\beta \rangle = 0 = \langle P_t, \qbf^*\alpha \wedge \rbf^{*}\beta \rangle
\]
and therefore the equation is satisfied. If instead $r=k$, we have 
\begin{align*}
\langle \Lcal_{B_t} P_t, \qbf^*\alpha \wedge \rbf^{*}\beta \rangle & = -\langle P_t, \qbf^*(i_{v_t} d\alpha) \wedge \rbf^{*}\beta + \qbf^*\alpha \wedge \rbf^*(i_{w_t} d\beta) \rangle \\ 
& = -\langle T_t, i_{v_t} d\alpha \rangle \langle U_t, \beta \rangle  - \langle T_t, \alpha \rangle \langle U_t, i_{w_t} d\beta \rangle \\ 
& = \langle \Lcal_{v_t} T_t, \alpha \rangle \langle U_t, \beta \rangle + \langle T_t, \alpha \rangle \langle \Lcal_{w_t} U_t, \beta \rangle \\ 
& = - \frac{\dd}{\dd t}\langle T_t, \alpha \rangle\langle U_t, \beta \rangle - \langle T_t, \alpha \rangle  \frac{\dd}{\dd t}\langle U_t, \beta \rangle \\ 
& = -\frac{\dd}{\dd t} [\langle T_t, \alpha\rangle \langle U_t,\beta\rangle ] \\ 
& = - \frac{\dd}{\dd t} \langle P_t, \qbf^*\alpha \wedge \rbf^{*}\beta \rangle, 
\end{align*}
and the equation is again satisfied. Notice that in the last computation we have used that the maps $t \mapsto\langle T_t, \alpha\rangle$ and $t\mapsto \langle U_t, \alpha\rangle$ are absolutely continuous (by the very definition of solution, see Lemma~\ref{lemma:equiv_def_of_solutions}~{\rm(ii)}) and we have applied the Leibniz rule for absolutely continuous maps (see, e.g.,~\cite[Exercise 3.17]{AFP}). By resorting once again to Lemma~\ref{lemma:equiv_def_of_solutions}~{\rm(ii)} we can therefore conclude that the path $t \mapsto P_t$ solves the equation with vector field $B_t$ and this concludes the proof. 
\end{proof} 

\begin{corollary}\label{corollary:product} 
Let $t \mapsto T_t \subset \mathscr{D}_{k}(\R^d)$, $t \in [0,1]$, be a path of currents in $\R^d$ with $\partial T_t=0$ such that there exists a vector field $b \in {\rm L}^1(\L^1(\dd t)\otimes \|T_t\|; \R^d)$ with 
 \[
 \frac{\dd}{\dd t} T_t + \Lcal_{b_t}T_t = 0.
 \]
Then, the currents $S_t:=\delta_{t}\times T_t$  satisfy
\[
\frac{\dd}{\dd t} S_t + \Lcal_{(1,b_t)}S_t = 0.
\]
\end{corollary}

\begin{proof}
It is sufficient to observe that the path $t\mapsto \delta_{t}$ satisfies the equation
\[
\frac{\dd}{\dd t}\delta_t+\Lcal_1 \delta_t=0
\]
and to apply Proposition~\ref{prop:cartesian_product}.
\end{proof}

\section{Disintegration structure}

As was explained in the introduction, one needs to consider also space-time solutions of the geometric transport equation. To this end, we need to precisely describe the disintegration structure of space-time integral currents, which is the topic of this section and in particular of Theorem~\ref{thm:struct}. We will also introduce crucial technical tools and explore the relationship between two different decompositions of the mass measure: the one given by the slicing and the one given by standard disintegration of measures with respect to the time projection $\tbf$.

The following definition is fundamental: Let $S \in \Irm_{k+1}([0,1] \times \R^d)$ and let $\|S\|$ denote its mass measure. We define the \textbf{critical set} of $S$ as
\begin{equation}\label{eq:definition_critical_set}
	\Crit(S) := \bigl\{(t,x) \in [0,1] \times \R^d:  \nabla^S \tbf (t,x)= 0 \bigr\}.
\end{equation}

\subsection{The coarea formula revisited}

We start by recording the following two observations on the disintegration via the coarea formula.

\begin{lemma}\label{lemma:disint_bella}
Let $S \in \Irm_{k+1}([0,1] \times \R^d)$ and define the measure 
\[
\gamma := |\nabla^S \tbf| \, \| S\|. 
\]
The following statements hold: 
\begin{enumerate}
    \item[\rm{(i)}] $\tbf_{\#} \gamma \ll \L^1$ and the disintegration of $\gamma$ with respect to $\tbf$ and $\L^1$ is given by
        \begin{equation*}
        \gamma =  \int_0^1  \| S\vert_t \| \; \dd t; 
        \end{equation*}
    \item[\rm{(ii)}] for $\L^1$-a.e.\ $t \in \R$ it holds that $\|S\vert_t \|(\Crit(S)) =0$;
    
    \item[\rm{(iii)}] for any set $A \subset [0,1] \times \R^d$ with $\|S\|(A) <\infty$ it holds that
    \[
    \int_{A} \frac{1}{|\nabla^S \tbf|} \; \dd \|S\vert_t\| <\infty
    \]
    for $\L^1$-a.e.\ $t \in \R$.
\end{enumerate}
\end{lemma}
    
\begin{proof}
	Assertion {\rm(i)} can be seen as a reformulation of the coarea formula~\eqref{eq:coarea_integral}. First, let $N \subset \R$ be an $\L^1$-null set. Then, applying~\eqref{eq:coarea_integral} with $g := \1_{N\times \R^d}$ we obtain  
	\begin{align*}
		\tbf_{\#} \gamma (N) &= \gamma(N \times \R^d) \\
		& = \int_{[0,1] \times \R^d} \1_{N\times \R^d} |\nabla^S \tbf| \; \dd \| S\| \\ 
		& = \int_0^1 \left(\int_{\{t\} \times \R^d} \1_{N\times \R^d}\; \dd\|S|_t\|\right)\; \dd t \\ 
		& =  \int_0^1 \1_{N}(t) \left(\int_{\{t\} \times \R^d}  \; \dd\|S|_t\|\right)\; \dd t \\
		&= 0, 
	\end{align*} 	
	because $N$ is $\L^1$-negligible. Hence, $\tbf_{\#} \gamma \ll \L^1$.

By the disintegration procedure recalled in Section~\ref{ss:disintegration}, the disintegration $\{\gamma_t\}$ of $\gamma$ with respect to $\tbf$ and $\L^1$ satisfies, for every Borel $B\subseteq\R$,
	\begin{equation*}
		|\nabla^S\tbf| \, \|S\|\restrict (B\times\R^d) = \gamma\restrict(B\times\R^d) = \int_{B} \gamma_t \; \dd t
	\end{equation*}
	as measures.
	Comparing this with the coarea formula~\eqref{eq:coarea_integral}, we infer 
	\begin{equation*}
		\int_{A}  \|S|_t\| \; \dd t  = \int_{A} \gamma_t \; \dd t.
	\end{equation*}
	Hence, since $B$ is arbitrary, {\rm(i)} follows from the uniqueness of the disintegration.
	
	Claim {\rm(ii)} is again a consequence of the coarea formula~\eqref{eq:coarea_integral}, now applied with $g := \1_{C}$, where $C:=\Crit(S)$. This yields 
	\begin{equation*}
	0 =\int_{C}|\nabla^S\tbf |	\; \dd \|S\| = \int_0^1 \left(\int_{\{t\} \times \R^d} \1_{C}\; \dd\|S|_t\| \right)\; \dd t = \int_0^1 \|S\vert_t \| (C)\; \dd t,
	\end{equation*}
	whence $\|S\vert_t \| (C) = 0$ for $\L^1$-a.e.\ $t \in \R$.
	
	It remains to show~{\rm(iii)}. Observe that, as a consequence of~{\rm(ii)}, it holds that $\1_{C}(t,x)=0$ for $ (\int_0^1  \| S\vert_t \| \; \dd t)$-a.e.\ $(t,x)$. By yet another application the coarea formula~\eqref{eq:coarea_integral}, this time with $g = |\nabla^S \tbf|^{-1} \1_{A\setminus C}$, we have 
	\begin{align*}
	+\infty &>
	\int_{A \setminus C} \; \dd \| S\| \\ 
	& = \int_{[0,1] \times \R^d} \frac{1}{|\nabla^S \tbf|}\1_{A\setminus C}\; \dd\gamma \\ 
	& = \int_0^1 \int_{\{t\} \times \R^d} \frac{1}{|\nabla^S \tbf|}\1_{A\setminus C}\; \dd\|S\vert_t\| \; \dd t\\ 
	& = \int_0^1 \int_{A} \frac{1}{|\nabla^S \tbf|}\; \dd\|S\vert_t\| \; \dd t,
	\end{align*}
	thus concluding the proof.
\end{proof}

We now define the vector field 
\begin{equation}\label{eq:def_eta}
    \xi(t,x) := \frac{\nabla^\S \tbf(t,x)}{|\nabla^\S \tbf(t,x)|}\1_{\Crit(S)^c}(t,x).  
\end{equation} 
Notice that, as a consequence of Lemma~\ref{lemma:disint_bella}~{\rm(ii)}, it holds that
\[
  \xi = \frac{\nabla^\S \tbf}{|\nabla^\S \tbf|}  \quad \text{a.e.\ with respect to the measure $ \int_0^1  \| S\vert_t \| \; \dd t$.}
\]

\begin{lemma} \label{lemma:disint_bella_vettoriale}
Let $S \in \Irm_{k+1}([0,1] \times \R^d)$ and define the vector-valued measure 
\[
\Gamma := \nabla^S \tbf \, \| S\|.
\]
Then, $\tbf_{\#} |\Gamma| \ll \L^1$ and the disintegration of $\Gamma$ with respect to $\tbf$ and $\L^1$ is
\begin{equation*}
    \Gamma = \int_0^1 \xi \, \|S\vert_t\| \; \dd t. 
\end{equation*}
\end{lemma}
	
\begin{proof}
Observe that $|\Gamma| = \gamma$, where $\gamma$ is the measure defined in Lemma~\ref{lemma:disint_bella}. Therefore, $\tbf_{\#} |\Gamma| = \tbf_{\#}\gamma\ll \L^1$ by said lemma. Set again $C := \Crit(S)$. For the second claim it is enough to observe that, for every vector-valued test function $\phi \in \Crm_c([0,1] \times \R^d; \R \times \R^d)$ it holds that 
\begin{align*}
\int_{[0,1] \times \R^d} \phi \cdot \dd \Gamma & = \int_{[0,1] \times \R^d} \phi \cdot \nabla^S \tbf \; \dd \|S\| \\ 
& = \int_{([0,1] \times \R^d) \cap C^c} \phi \cdot \nabla^S \tbf \; \dd \|S\| \\ 
& = \int_{([0,1] \times \R^d) \cap C^c} \phi \cdot \xi |\nabla^S \tbf|\; \dd \|S\| \\
& = \int_{[0,1] \times \R^d} \phi \cdot \xi |\nabla^S \tbf|\; \dd \|S\| \\ 
& = \int_0^1 \left( \int_{\{t\}\times \R^d} \phi \cdot \xi \; \dd\|S|_t\| \right) \dd t,
\end{align*}
where we have used Lemma~\ref{lemma:disint_bella} again in the last line. We conclude via the uniqueness of the disintegration. 
\end{proof}

\subsection{Disintegration of the mass measure}

Given a space-time integral current $S \in \Irm_{k+1}([0,1] \times \R^d)$, we now turn to investigate the disintegrations of the mass measure $\|\S\|$ and of $S$ itself, seen as a multivector-valued measure, with respect to the map $\tbf$ and the measure
\[
  \lambda := \L^1 +(\tbf_{\#}\|S\|)^s,
\]
where $(\tbf_{\#}\|S\|)^s$ denotes the singular part (with respect to $\L^1$) of $\tbf_{\#}\|S\|$. For the rest of this section, and in fact the whole article, we will denote by $\{\mu_t\}$ the disintegration of $\|S\|$ with respect to $\tbf$ and $\lambda$, i.e., 
\begin{equation}\label{eq:disintegration}
	\| S \| = \int_0^1 \mu_t \; \dd \lambda(t) = \int_0^1 \mu_t \; \dd t  +  \int_0^1 \mu_t \; \dd \lambda^s(t).
\end{equation} 
The next goal is to investigate the measures $\mu_t$, which will be achieved in several steps. We begin with the following lemma, which characterises the disintegration of the non-critical mass, that is, the mass restricted to the complement of the critical set.

\begin{lemma}\label{lemma:disint_non_critical_mass} 
Let $S \in \Irm_{k+1}([0,1] \times \R^d)$ and define the measure
\[
\sigma := \| S\| \restrict \Crit(S)^c.
\]
Then, $\tbf_{\#} \sigma \ll \L^1$ and the disintegration of $\sigma$ with respect to $\tbf$ and $\L^1$ is given by 
\[
	\sigma = \int_0^1 \frac{1}{|\nabla^S \tbf|}\| S\vert_t\|\; \dd t.  
\]
\end{lemma}

\begin{proof}
	Let $N \subset \R$ be an $\L^1$-negligible set. Applying the coarea formula~\eqref{eq:coarea_integral} with $g = |\nabla^S \tbf|^{-1}\1_{(N \times \R^d) \cap C^c}$, where $C:=\Crit(S)$, we obtain   
	\begin{align*}
		\tbf_{\#} \sigma (N) 
		& = \int_{[0,1] \times \R^d} \1_{N \times \R^d}\; \dd\sigma \\ 
		& = \int_{C^c} \1_{N \times \R^d} \; \dd\|S\| \\ 
		& = \int_0^1 \left(\int_{\{t\} \times \R^d}  |\nabla^S \tbf|^{-1} \1_{(N \times \R^d) \cap C^c}\; \dd\|S|_t\| \right)\; \dd t \\ 
		& =  \int_0^1 \left(\int_{\{t\} \times \R^d}   |\nabla^S \tbf|^{-1} \1_{N \times \R^d}\; \dd\|S|_t\| \right)\; \dd t \\ 
		& =  \int_0^1 \1_{N}(t) \left(\int_{\{t\} \times \R^d}   |\nabla^S \tbf|^{-1} \dd\|S|_t\| \right)\; \dd t \\ 
		& = 0, 
	\end{align*} 	
	where we have exploited that $\1_{C^c} (t,x)= 1$ for $( \int_0^1  \| S\vert_t \| \; \dd t)$-a.e.\ points $(t,x)$ by virtue of Lemma~\ref{lemma:disint_bella}~(ii). Thus, $\tbf_{\#} \sigma \ll \L^1$. An analogous argument also yields
	$$
	\tbf_{\#} \sigma = \left[ \int_{\{t\} \times \R^d} |\nabla^S \tbf|^{-1} \; \dd\| S\vert_t \| \right] \L^1(\dd t)
	= \left(|\nabla^S \tbf|^{-1}\| S\vert_t\|\right)(\R\times \R^d) \, \L^1(\dd t)
	$$ 
	as measures.
	
	By the disintegration theorem, the disintegration $\{\sigma_t\}$ of $\sigma$ with respect to $\tbf$ and $\L^1$ satisfies
	\begin{equation*}
		\sigma = \int_0^1 \sigma_t  \; \dd t.
	\end{equation*}
 	For every Borel set $A \subset [0,1] \times \R^d$, we have by the coarea formula with $g := |\nabla^S \tbf|^{-1} \1_{A \setminus C}$ that
 	\begin{align*}
	 \sigma(A)
	 &= \| S \|(A \setminus C) \\ 
	& = \int_0^1 \int_{\{t\} \times \R^{d}}|\nabla^S \tbf|^{-1} \1_{A\setminus C} \; \dd \|S\vert_t\| \; \dd t \\ 
	& = \int_0^1 \int_{\{t\} \times \R^{d}}|\nabla^S \tbf|^{-1} \1_{A} \; \dd \| S\vert_t\| \; \dd t, 
	\end{align*}
	where we have once again exploited that $\1_{A\setminus C} (t,x)= \1_A(t,x)$ for $( \int_0^1  \| S\vert_t \| \; \dd t)$-a.e.\ $(t,x)$. As $A \subset [0,1] \times \R^d$ was arbitrary, we therefore obtain   
	\begin{equation*}
		\sigma_t =  |\nabla^S \tbf|^{-1}\| S\vert_t\|
		\qquad \text{for $\L^1$-a.e.\ $t\in \R$ as measures on $\R\times \R^d$,}
	\end{equation*}
	which concludes the proof. 
\end{proof}

As before, we also have the multivector-valued analogue of Lemma~\ref{lemma:disint_non_critical_mass}. 

\begin{lemma}\label{lemma:disint_non_critical_current} 
 Let $S \in \Irm_{k+1}([0,1] \times \R^d)$ and define the $(k+1)$-vector-valued measure 
\[
\Sigma := S \restrict {\Crit(S)^c}.
\]
Then, $\tbf_{\#} |\Sigma|\ll \L^1$ and the disintegration of $\Sigma$ with respect to $\tbf$ and $\L^1$ is
\[
	\Sigma = \int_0^1 \frac{1}{|\nabla^S \tbf|} \xi \wedge  S\vert_t\; \dd t. 
\]
\end{lemma}
\begin{proof} Observe that $|\Sigma| = \sigma$, where $\sigma$ is the measure defined in Lemma~\ref{lemma:disint_non_critical_mass}, and so,
$\tbf_{\#} |\Sigma| = \tbf_{\#}\sigma\ll \L^1$. Set $C := \Crit(S)$. For the second claim, observe that, for every vector-valued test function $\phi \in \Crm_c([0,1] \times \R^d; \R \times \R^d)$, it holds that 
\begin{align*}
\int_{[0,1] \times \R^d} \phi \; \dd \Sigma  & = \int_{([0,1] \times \R^d) \cap C^c} \phi \; \dd S \\ 
& = \int_{([0,1] \times \R^d) \cap C^c} \phi \cdot \vec{S} \; \dd \|S\| \\ 
& = \int_{[0,1] \times \R^d} \phi \cdot \vec{S} \; \dd\sigma.
\end{align*}
By Lemma~\ref{lemma:disint_non_critical_mass} we thus obtain for the disintegration $\{\Sigma_t\}$ of $\Sigma$ with respect to $\tbf$ and $\L^1$ that
\begin{equation*}
	\Sigma_t = \vec{S} |\nabla^S \tbf|^{-1}\| S\vert_t\|
	\qquad \text{for $\L^1$-a.e.\ $t \in \R$} 
\end{equation*}
as $(k+1)$-vector-valued measures on $\R\times \R^d$. 

Observe that
\[
  \xi(t,x)\wedge \vec{S}|_t(t,x)=\vec{S}(t,x)  \qquad\text{for $( \textstyle\int_0^1  \| S\vert_t \| \; \dd t)$-a.e.\ $(t,x)$,}
\]
which follows from \cite[Lemma 3.5]{Rindler21a?} (see also \cite[1.5.3]{Federer69book})
and therefore for $\L^1$-a.e.\ $t$ we have
\begin{align*}
	\Sigma_t & = \bigl( \xi \wedge \vec{S}|_t\bigr)\,  |\nabla^S \tbf|^{-1}\| S\vert_t\|
	= 
	\frac{1}{|\nabla^S \tbf|} \xi \wedge  S\vert_t 
\end{align*}
as $(k+1)$-vector-valued measures on $\R\times \R^d$. This finishes the proof. 
\end{proof}

\subsection{Structure theorem}

Summing up the results from the previous section, we have the following formulae for the disintegrations with respect to the map $\tbf$ and the measure $\L^1$: 
\begin{align*}
|\nabla^S \tbf| \| S\| &= \int_0^1 \| S\vert_t \| \; \dd t, & \| S\|\restrict \Crit(S)^c  &= \int_0^1 |\nabla^S \tbf|^{-1} \| S\vert_t\| \; \dd t,  \\
\nabla^S \tbf \| S\| &= \int_0^1 \xi \| S\vert_t \| \; \dd t, &  S\restrict \Crit(S)^c &= \int_0^1 |\nabla^S \tbf|^{-1} \xi \wedge S\vert_t \; \dd t. 
\end{align*}

We are now ready to describe the structure of the disintegration $\{\mu_t\}_t$ appearing in~\eqref{eq:disintegration}. 

\begin{theorem}[Disintegration structure] \label{thm:struct}
Let $S \in \Irm_{k+1}([0,1] \times \R^d)$ and write
\[
  \| S \| = \int_0^1 \mu_t \; \dd \lambda(t) = \int_0^1 \mu_t \; \dd t  +  \int_0^1 \mu_t \; \dd \lambda^s(t),
\]
where $\lambda := \L^1 + \lambda^s := \L^1 + (\tbf_{\#}\|S\|)^s$, as in~\eqref{eq:disintegration}. Then the following statements hold:
	\begin{enumerate}
		\item[\rm (i)] For $\lambda^s$-a.e.\ $t \in \R$ the measure $\mu_t$ is concentrated on $\Crit(S)$.
		\item[\rm (ii)] For $\L^1$-a.e.\ $t \in \R$ the measure $\mu_t$ can be decomposed as
		\begin{equation*}
		    \mu_t =|\nabla^S\tbf|^{-1} 
		    \| S\vert_t \| + \mu^s_t,
		\end{equation*}
		where $\mu^s_t$ is a measure which is concentrated on $\Crit(S)$ and is singular with respect to $|\nabla^S\tbf|^{-1} \|S\vert_t\|$ and also with respect to $\mathscr H^k$.
		\end{enumerate}
\end{theorem}

The disintegration of the mass measure $\| S \|$ with respect to the map $\tbf$ and $\L^1$ therefore has the following structure:
\begin{equation}\label{eq:structure_disint}
	\| S \| = \int_0^1 \left( |\nabla^S\tbf|^{-1}  \| S\vert_t \| + \mu^s_t  \right) \; \dd t  +  \int_0^1 \mu_t \; \dd \lambda^s(t), 
\end{equation}
where for $\lambda^s$-a.e.\ $t\in \R$ the measure $\mu_t$ is concentrated on $\Crit(S)$, while for $\L^1$-a.e.\ $t \in \R$ the measure $\mu^s_t$ is concentrated on $\Crit(S)$ and is singular with respect to $|\nabla^S\tbf|^{-1} \|S\vert_t\|$ and with respect to $\mathscr H^k$.

\begin{proof} The strategy of the proof consists of decomposing $\|S\| = \sigma + \nu$, where $\sigma := \|S\|\restrict C^c$ and $\nu:=\|S\|\restrict C$ (as usual, $C:=\Crit(S)$), and then applying the disintegration theorem separately to $\sigma$ and $\nu$.

First, invoking Lemma~\ref{lemma:disint_non_critical_mass}, we have that the disintegration of $\sigma$ with respect to $\tbf$ and $\lambda$  satisfies 
	\begin{align*}
	    & \sigma_t = |\nabla^S \tbf|^{-1} \|S\vert_t\| & &\text{as measures on $[0,1] \times \R^d$ for $\L^1$-a.e.\ $t \in \R$,} \\ 
	    & \sigma_t  = 0 & &\text{as measures on $[0,1] \times \R^d$ for $\lambda^s$-a.e.\ $t\in \R$.}
	\end{align*}
Denoting by $\{\nu_t\}$ the disintegration of $\nu$ with respect to $\tbf$ and $\lambda$, we obtain 
    \begin{align*}
		\|S\| & =\sigma +\nu \\ 
		& =  \int_0^1 |\nabla^S \tbf|^{-1}\| S\vert_t\| \; \dd t + \int_0^1 \nu_t \dd \lambda(t)\\ 
		& =  \int_0^1 \left(  |\nabla^S \tbf|^{-1}\| S\vert_t\|+ \nu_t \right)\; \dd t + \int_0^1 \nu_t \; \dd \lambda^s(t). 
	\end{align*}
	We therefore conclude, in view of~\eqref{eq:disintegration}, that 
	\begin{align*}
	    & \mu_t = |\nabla^S \tbf|^{-1}\| S\vert_t\|+ \nu_t &  &\text{as measures on $[0,1] \times \R^d$ for $\L^1$-a.e.\ $t \in \R$,} \\
	    & \mu_t = \nu_t & &\text{as measures on $[0,1] \times \R^d$ for $\lambda^s$-a.e.\ $t \in \R$.} 
	\end{align*}
	In particular, since $\{\nu_t\}$ is the disintegration of $\nu$, which is concentrated on $\Crit(S)$, we conclude that the same holds for $\mu_t$ for $\lambda^s$-a.e.\ $t \in \R$; this yields~{\rm(i)}.
	
	Comparing the absolutely continuous parts, we see that 
	\begin{equation*}
		\mu_t =|\nabla^S\tbf|^{-1} \| S\vert_t \| + \mu^s_t
	\end{equation*}
	for $\L^1$-a.e.\ $t \in \R$, where 
	\begin{equation}\label{eq:def_of_sing_sard}
	\mu^s_t  := \nu_t
	\end{equation}
	is concentrated on $\Crit(S)$ and is therefore singular with respect to $|\nabla^S\tbf|^{-1} \| S\vert_t \|$ (for which we recall the set $\Crit(S)$ is negligible).
	
	Resorting to the coarea formula for rectifiable sets (see, e.g.,~\cite[Theorem 2.93]{AFP}) one can show that $\mu_t^s \perp \mathscr H^k$. Indeed, notice that the set $C:=\Crit(S)$ is $(k+1)$-rectifiable, being a subset of (the carrier of) $S$. Then, by the coarea formula, one infers that for $\L^1$-a.e.\ $t$ the set 
    \[
    C_t := C \cap (\{t\} \times \R^d)
    \] 
    is $k$-rectifiable and it holds that 
    \[
    \int_0^1 \mathscr H^{k}(C_t) \; \dd t = \int_C |\nabla^S\tbf| \; \dd \mathscr H^{k+1} = 0,
    \]
    whence $\mathscr H^{k}(C_t) = 0$ for $\L^1$-a.e.\ $t \in \R$. Since, by definition, $\mu_t^s$ is concentrated on $C \cap (\{t\} \times \R^d)=C_t$, we conclude that $\mu_t^s \perp \mathscr H^k$ for $\L^1$-a.e.\ $t \in \R$, thus finishing the proof. 
	\end{proof}

\begin{remark}\label{remark:coarea_di_alberti_mu_t_perp}
From the proof above it follows that $\mu_t^s$ is concentrated on $\Crit(S)\cap (\{t\}\times\R^d)$, which is $\H^k$-null for $\L^1$-a.e.\ $t$. Hence, by the Besicovitch differentiation theorem (see, e.g.,~\cite[Thm.~2.22]{AFP}), $\mu_t^s$ can be identified (for $\L^1$-a.e.\ $t$) with the restriction of $\mu_t$ to the set
\[
\left\{(t,x):\limsup_{\rho\to 0}\frac{\mu_t(B_\rho(x))}{\rho^k} = \infty\right\},
\]
which conveys the idea that $\mu_t^s$ is more concentrated than $\mathscr H^k$. 
\end{remark}

We conclude this section observing that, in general, all terms in the disintegration~\eqref{eq:structure_disint} can be non-zero. The measure $\lambda^s$ takes into account singular-in-time behaviour (in Section~\ref{ss:ac_currents} we will show a characterisation of currents for which $\lambda^s = 0$, namely $\lambda^s$ vanishes if and only if $S \in \Irm_{k+1}^{\AC}([0,1]\times \R^d)$).

\begin{example}
As an easy example where $\lambda^s \ne 0$, one can consider the current $S \in \Irm_{2}(\R\times \R^2)$ defined by $S=S_1+S_2$, where
\[
S_1 = \bbb{0,1} \times C_1, \qquad S_2 = \bbb{1,2} \times C_2 + \delta_{1} \times W,
\]
where $C_1$ and $C_2$ are circles of radius $1$ and $2$ respectively and $W$ is the annulus between them (so that $\partial W = C_2 - C_1$ and hence $\partial S\restrict {(0,1) \times \R^2} = 0$). In this case, one can see that $\tbf_{\#}\|S\| = \L^1\restrict {[0,2]} + \delta_{1}$ and hence $\lambda^s$ is non-zero.
\end{example}

On the other hand, the measures $\mu_t^s$ in~\eqref{eq:structure_disint} account for a completely different type of singularity, measuring a sort of \emph{diffuse concentration} that is ``smeared out in time'', which will be further investigated in the following sections.

\section{Notions of variation}\label{sc:variation}

A central quantity in our theory is the \emph{variation} of a time-indexed path of currents. It turns out that this notion can be defined in several distinct ways. One classic approach is to use a metric on the space of integral or normal currents (usually, the one induced by one of the flavours of flat norm introduced before) and define the variation in a pointwise way. This is the approach taken in the theory of rate-independent systems, cf.~\cite{MielkeRoubicek15book,ScalaVanGoethem19}. On the other hand, the variational approach to the motion of currents developed in~\cite{Rindler21a?} crucially rests on a notion of space-time variation, which is, roughly, the area traversed as the current moves. In preparation for the later developments in this work, it is crucial to compare these notions and to derive conditions for their equality or non-equality. 

As motivated in the introduction, from now onwards we almost exclusively consider evolutions of \emph{boundaryless} currents.

\subsection{Variations and AC integral currents}

Let $(X,\mathsf d)$ be a locally compact metric space. 
For a curve $t \mapsto \gamma(t) \in X$, $t \in [a,b]$ ($a < b$), we define the \textbf{pointwise variation} as 
\begin{equation}\label{eq:def_pV}
	{\mathsf d}\dash\pV(\gamma;[a,b]):=\sup\left\{ \sum_{i=1}^N \mathsf d(\gamma(t_{i+1}),\gamma(t_i)):   a\leq t_1\leq \ldots \leq t_N\leq b \right\} . 
\end{equation}
We further define the \textbf{essential variation} of the curve $\gamma$ as 
\[
	{\mathsf d}\dash\eV(\gamma;[a,b]) := \inf \left\{{\mathsf d}\dash\pV(\tilde{\gamma};[a,b]):  \gamma(t) = \tilde{\gamma}(t)  \, \text { for  $\L^1$-a.e.\ $t \in [a,b]$} \right\}. 
\]
We extend the same definition to curves $\gamma$ which are only defined for $\L^1$-a.e. $t \in [a,b]$. In this case, the supremum in \eqref{eq:def_pV} is taken over families of partitions such that $\gamma$ is defined at $t_i$ for every $i$. 
By~\cite[Remark 2.2]{AmbrosioAnnali}, the infimum in the definition of essential variation is achieved and therefore, if ${\mathsf d}\dash\eV(\gamma:[a,b]) < \infty$, then there exist two \textbf{good representatives}, the right-continuous representative $\gamma_+$ and the left-continuous representative $\gamma_-$ such that 
\[
{\mathsf d}\dash\eV(\gamma_{\pm}; [a,b]) = {\mathsf d}\dash\pV(\tilde{\gamma}; [a,b]).
\]
If ${\mathsf d}\dash\eV(\gamma;[a,b])<\infty$, then ${\mathsf d}\dash\eV(\gamma; \frarg)$ can be extended to a finite measure on the Borel subsets of $[a,b]$.

In this vein, for for $S\in\Irm_{k+1}([0,1] \times \R^d)$ and $U\in\Nrm_{k+1}([0,1] \times \R^d)$ with $\partial S \restrict (0,1) \times \R^d = \partial U \restrict (0,1) \times \R^d = 0$ we set, with a little abuse of notation,
\begin{align*}
\Fhom_\Irm\dash\eV(S;[a,b]) &:= \Fhom_\Irm\dash\eV(t \mapsto S(t);[a,b]), \\
\Fhom\dash\eV(U;[a,b]) &:= \Fhom\dash\eV(t \mapsto U(t);[a,b])
\end{align*}
for every closed interval $[a,b] \subset [0,1]$. Recall that we denote by $S(t):=\pbf_{*}(S|_t)$ the pushforward of the slice $S|_t$ onto $\R^d$.  Likewise we define the slices $U|_t$ and the pushforwards $U(t):=\pbf_{*}(U|_t)$ for $\L^1$-a.e.\ $t$. 

On the other hand, the work~\cite{Rindler21a?} introduced the \textit{(space-time) variation} of an integral space-time current. Given a current $S$ of finite mass, i.e., $S \in \Mrm_{k+1}([0,1] \times \R^d)$, we define the \textbf{(space-time) variation} of $S$ on the interval $[a,b]$ to be
\[
\Var(S; [a,b]):=\int_{[a,b]\times\R^d} \|\pbf (\vec{S})\|\; \dd\|S\|.
\]
Here and in the following, we will often write $\pbf (\vec{S})$ instead of $\Wedge^{k+1} \pbf(\vec{S})$ for ease of notation. We remark that, if $S$ is integral, then $\vec{S}$ is simple and so is $\pbf(\vec{S})$. Therefore, in this case, $ \|\pbf(\vec{S})\| = |\pbf(\vec{S})|$. One can further see that $\Var(S; \frarg)$ can be extended to all Borel sets (by the very same formula) to define a non-negative finite measure on $\R$, which will still be denoted by $\Var(S;\frarg)$.

We will mostly work with the following particular subclasses of integral currents, namely the \textbf{absolutely continuous ($\AC$)} integral space-time currents \[
\Irm^{\AC}_{1+k}([0,1] \times \R^d) := \bigl\{S \in \Irm_{1+k}([0,1] \times \R^d):  \Var(S;\frarg) \ll \L^1, \;
\Var(\partial S;\frarg) \restrict (0,1) \ll \L^1  \bigr\}
\]
and the \textbf{Lipschitz} integral space-time currents
\begin{align*}
\Irm^{\rm Lip}_{1+k}([0,1] \times \R^d) := \biggl\{S \in \Irm^{\AC}_{1+k}([0,1] \times \R^d): &\frac{\dd \Var(S;\frarg)}{\dd \L^1} \in \Lrm^\infty([0,1]), \\
  &\frac{\dd \Var(\partial S;\frarg) \restrict (0,1)}{\dd \L^1} \in \Lrm^\infty((0,1)) \biggr\}.
\end{align*}
This space of Lipschitz integral space-time currents was already introduced in~\cite{Rindler21a?}, but there an additional $\Lrm^\infty$-bound on the masses of the slices was included, which is not necessary here. We remark further that since the present paper mostly focusses on the transport of \emph{boundaryless} currents, that is, assuming $\partial S \restrict (0,1) \times \R^d = 0$, the second condition in the definitions of $\Irm^{\AC}_{1+k}([0,1] \times \R^d)$ and $\Irm^{\rm Lip}_{1+k}([0,1] \times \R^d)$ will always be trivially satisfied.

\subsection{Comparison of variations} 

The following proposition entails that $\Fhom_\Irm\dash\eV(S;\frarg)$ and $\Fhom\dash\eV(U;\frarg)$  are finite measures and are always bounded from above by the space-time variation.

\begin{proposition}\label{prop:TV_leq_Var}  Let $S\in\Irm_{k+1}([0,1] \times \R^d)$ with $\partial S \restrict (0,1) \times \R^d = 0$. Then the following claims hold:
    \begin{itemize}
        \item[\rm (a)] The map $t \mapsto S|_t$ defines a path that has finite essential variation with values in the metric space $(\Irm_{k}([0,1] \times \R^d); \Fhom_\Irm)$.
        \item[\rm (b)] The map $t\mapsto S(t)$ defines a path that has finite essential variation with values in the metric space $(\Irm_{k}(\R^d); \Fhom_\Irm)$.
        \item[\rm (c)] $\Fhom_\Irm\dash\eV(S;\frarg)\leq\Var(S;\frarg)$ as measures on $\R$.
    \end{itemize}
\end{proposition}

\begin{proof}
We begin by showing that $t \mapsto S\vert_t$ has finite essential variation on every interval $[a,b]$. Indeed, fix an arbitrary (pointwise everywhere defined) representative and any finite subdivision times $a\leq t_1\leq\ldots\leq t_N\leq b$ such that the slices $S|_{t_i}$ are defined. By definition
\[
S\vert_{t_{i+1}}-S|_{t_i}=\partial (S\restrict (t_i, t_{i+1}) \times \R^d).
\]
This shows that $S\restrict {\{t_i< t\leq t_{i+1}\}}$ is a competitor for $\Fhom_{\Irm}(S|_{t_{i+1}}, S|_{t_i})$, and so
\[
\Fhom_{\Irm}(S|_{t_{i+1}},S|_{t_i})\leq \Mbf(S\restrict (t_i, t_{i+1}) \times \R^d).
\]
Summing over $i$ we get 
\begin{equation*}
    \sum_{i=0}^{N-1} \Fhom_{\Irm}(S\vert_{t_{i+1}}, S\vert_{t_i}) \le \sum_{i=0}^{N-1} \Mbf(S\restrict (t_i, t_{i+1}) \times \R^d)
    \leq \Mbf(S)< \infty 
\end{equation*}
and, since the subdivision $\{t_i\}_{i =0}^N$ is arbitrary (with the constraint of all slices being defined), we conclude that indeed $t \mapsto S\vert_t$ has finite essential variation. 

For the map $t \mapsto S(t)$ the proof is similar and uses the commutativity of $\pbf_*$ with $\partial$:  
we have (with the same notations as before)
\[
S(t_{i+1})-S(t_i)=\pbf_*\partial(S\restrict (t_i, t_{i+1}) \times \R^d)=\partial \pbf_*(S\restrict (t_i, t_{i+1}) \times \R^d), 
\]
showing that $\pbf_*(S\restrict (t_i, t_{i+1}) \times \R^d)$ is a competitor for $\Fhom_\Irm(S(t_{i+1})-S(t_i))$, and so
\begin{align*} 
\sum_{i=0}^{N-1} \Fhom_{\Irm}(S(t_{i+1}), S(t_i)) & \le \sum_{i=0}^{N-1} \Mbf(\pbf_*(S\restrict (t_i, t_{i+1}) \times \R^d)) \\ 
& \le \sum_{i=0}^{N-1} \int_{(t_i, t_{i+1}) \times \R^d} \| \Wedge^{k+1}\pbf (\vec{S})\|\; \dd\|S\| \\ 
& =\sum_{i=0}^{N-1}  \Var(S;(t_i, t_{i+1}) \times \R^d)\\ 
& \leq \Var(S,(a,b)) \\
&< \infty, 
\end{align*}
where we have used~\eqref{eq:mass_of_pushforward}. The same argument can also be used to show the inequality 
\[
\Fhom_\Irm\dash\eV(S; I) \le \Var(S;I)
\]
for an interval $I \subset [0,1]$ and therefore the proof is complete. 
\end{proof}

For the sake of completeness, we state also the following version of Proposition~\ref{prop:TV_leq_Var} for normal currents:

\begin{proposition} Let $U\in\Nrm_{k+1}([0,1] \times \R^d)$ with $\partial U \restrict (0,1) \times \R^d = 0$. Then the following claims hold: 
\begin{itemize}
	\item[\rm (a)] The map $t \mapsto U|_t$ defines a path that has finite essential variation with values in the metric space $(\Nrm_{k}([0,1] \times \R^d); \Fhom)$.
	\item[\rm (b)] The map $t\mapsto U(t)$ defines a path that has finite essential variation with values in the metric space $(\Nrm_{k}(\R^d); \Fhom)$.
	\item[\rm (c)] $\Fhom\dash\eV(U;\frarg)\leq\Var(U;\frarg)$ as measures on $\R$.
\end{itemize}
\end{proposition}

\begin{proof}
The proof is similar to the one for Proposition~\ref{prop:TV_leq_Var} and is therefore omitted.
\end{proof}

In general, the opposite inequality in Proposition~\ref{prop:TV_leq_Var}, namely $\Fhom_\Irm\dash\eV(S;\frarg) \le \Var(S;\frarg)$, may not hold, due to the possible presence of jumps in the path $t\mapsto S|_t$. Indeed, whenever a jump occurs at a certain time $t_0$, $\Var(S;\frarg)$ depends on the particular current that connects $S|_{t_0^-}$ and $S|_{t_0^+}$, while $\Fhom_{\Irm}\dash\eV$ always measures the optimal connection (given by the solution to Plateau's problem).

The next theorem entails that jumps are in fact the only obstructions to the equality between $\Var$ and $\Fhom_\Irm\dash\eV$. Its proof is postponed to Section~\ref{subsec:equivalence}, as it builds upon tools that will be introduced in the next section.

\begin{theorem}[Equality of variations]\label{thm:Var_equals_TV}
Let $S\in\Irm_{k+1}([0,1] \times \R^d)$ with $\partial S \restrict (0,1) \times \R^d = 0$ and such that $\Var(S;\frarg)$ is non-atomic. Then,
\[
\Var(S;\frarg)=\Fhom\dash\eV(S;\frarg)=\Fhom_\Irm \dash\eV(S;\frarg).
\]
\end{theorem}

The equality of the various concepts of variation in the class of space-time currents with non-atomic variation occupies a central place in this work and can be seen as a generalisation to any codimension of the following formula, valid for a function $u \colon [0,1] \to \R$ that is continuous and of bounded variation: 
\[
\pV(u,\R)
= \int_{\text{graph}(u)} |\pbf(\tau)|\; \dd \H^{1}
= \Var(S_u,\R), 
\]
where $S_u := \tau \H^{1} \restrict \text{graph}(u)$ and $\tau$ is the forward-pointing unit tangent to $\text{graph}(u)$, see \cite[Example 3.1]{Rindler21a?}.

\subsection{Zero-slice lemma}

We now establish the following \emph{zero-slice lemma}, which describes normal or integral space-time currents with trivial slices; it will be used several times throughout the paper.  

\begin{lemma} \label{lemma:zero_slices}
	Let $W \in \Nrm_{k+1}([0,1] \times \R^d)$ be a normal current with $\partial W \restrict (0,1) \times \R^d = 0$. 
	Then, the following are equivalent:
	\begin{enumerate}
	    \item[{\rm(i)}] $\displaystyle W\vert_t = 0$ for $\L^1$-a.e.\ $t \in [0,1]$.
    	\item[{\rm(ii)}] With $\kappa := \tbf_\# \|W\|$, for $\kappa$-a.e.\ $t$ there exists $W_t \in \Nrm_{k+1}([0,1] \times \R^d)$ with $W_t$ supported in $\{t\}\times\R^d$ such that
    	\begin{equation}\label{eq:W_zero_slices}
    	W= \int_0^1 W_t \; \dd \kappa(t)   \qquad\text{and}\qquad
    	\partial W_t=0 \quad\text{for $\kappa$-a.e.\ $t$.}
        \end{equation}
	\end{enumerate}
In addition, in~{\rm(ii)} it holds that $\vec W_t(x)\in\Wedge_k(\{0\}\times \R^d)$ for $\kappa$-a.e.\ $t$ and $\|W_t\|$-a.e.\ $x$. Moreover, if $W$ is integral and {\rm(ii)} holds, then $\kappa$ is purely atomic and the $W_t$'s are integral, so that
\[
W=\sum_{j} W_{t_j}
\]
for an at most countable collection of points $t_j \in [0,1]$.
\end{lemma}

\begin{proof}
	${\rm(ii)}\implies {\rm(i)}$ We use the cylinder formula to infer that for $\L^1$-a.e.\ $t\in [0,1]$
	\[
	\langle W|_t,\omega\rangle =\langle \partial(W\restrict\{\tbf<t\}),\omega\rangle =\langle W\restrict\{\tbf<t\},d\omega\rangle=\int_{(0,t)}\langle W_t,d\omega\rangle \; \dd \kappa(t)=0
	\]
	because $\partial W_t=0$ for $\kappa$-a.e.\ $t$.
	
	${\rm(i)}\implies {\rm(ii)}$ We consider the disintegration of $W$ (as a multivector-valued measure) with respect to $\kappa :=\tbf_\# \|W\|$, i.e.,
	\[
	W=\int_0^1 W_t \; \dd \kappa(t),
	\]
	where $W_t$ is a $(k+1)$-multivector-valued measure with total mass $1$ supported on $\{t\}\times \R^d$. 
	By Remark~\ref{rem:disint_explicit}, we can write for $\kappa$-a.e.\ $t \in (0,1)$,
	\[
	W_t = \wslim_{h\to 0} \frac{W\restrict {(t-h,t+h)\times\R^d}}{\kappa((t-h,t+h))}.
	\]
	By the cylinder formula~\eqref{eq:definition_slicing} and the assumption on the slices, we can choose a sequence $h_j\to 0$ so that $\partial [W\restrict {(t-h_j,t+h_j)\times\R^d}]=W|_{t+h_j}-W|_{t-h_j}=0$, and as a consequence we must also have $\partial W_t=0$. In particular, $W_t$ is a normal $(1+k)$-current.
	Since $W_t$ is normal for $\kappa$-a.e.\ $t$ and it is supported in $\{t\}\times\R^d$, its orienting multi-vector must belong to $\Wedge_k(\{0\}\times \R^d)$.
	
	Finally, if $W$ is integral then it is supported on a $(k+1)$-rectifiable set $R$, and thus there can be only countably many $t$'s for which $\|W_t\|(R)>0$, therefore $\kappa$ must be atomic. By the same limiting argument as for normal currents, and the closure theorem for integral currents, we also deduce that the $W_t$'s are integral.
\end{proof}

The zero-slice lemma is often combined with the following simple reasoning.

\begin{lemma}\label{lemma:W_perp_S}
Let $S\in \Irm_{k+1}([0,1] \times \R^d)$ with $\partial S \restrict (0,1) \times \R^d = 0$ such that $\Var(S;\frarg)$ is non-atomic. Then, $\| S\| \perp \|W\|$ for any $W$ of the form~\eqref{eq:W_zero_slices}. 
\end{lemma}

\begin{proof}
Let 
\[
R:=\left\{x:\limsup_{\rho\to 0}\frac{\|S\|(B(x,\rho))}{\rho^{k+1}}>0\right\},
\]
which is a $(k+1)$-rectifiable carrier of $\|S\|$. Since $\Var(S;\frarg)$ is non-atomic we know that $\|S\|(\{t\}\times\R^d)=0$ for every $t$. Moreover, the $W_t$'s are normal $(k+1)$-currents, hence $\|W_t\|\ll \H^{k+1}$ (see~\cite[4.1.20]{Federer69book}), and in particular $\|W_t\|\restrict R\ll \H^{k+1}\restrict R\leq \|S\|\restrict R$. Since $\|S\|\restrict R(\{t\}\times\R^d)=0$, it follows $\|W_t\|(R)=0$. From this we deduce
\[
\|W\|(R)=\int_0^1 \|W_t\|(R)\; \dd \kappa(t)=0.
\]
This proves that $ \|W\|\perp \|S\|$.
\end{proof}

From the previous lemma it follows that $\AC$ integral space-time currents are determined by their slices:

\begin{corollary}\label{cor:S=S'}
Let $S,S'\in \Irm_{k+1}([0,1] \times \R^d)$ with $\partial S \restrict (0,1) \times \R^d = 0 = \partial S' \restrict (0,1) \times \R^d$ be such that $\Var(S;\frarg)$  and $\Var(S';\frarg)$ are non-atomic (which is the case if, e.g., $S,S'\in\Irm_{k+1}^{\AC}([0,1] \times \R^d)$). Suppose that 
\[
S|_t=S'|_t\qquad\text{for $\L^1$-a.e.\ $t$}.
\]
Then, $S=S'$.
\end{corollary}

\begin{proof} By Lemma~\ref{lemma:zero_slices},
\[
S- S'=\sum_{j\in\mathbb{N}} W_{t_j} =: W
\]
where each $W_{t_j}$ is boundaryless, integral, and supported in $\{t_j\}\times\R^d$. By Lemma~\ref{lemma:W_perp_S}, $\|W\|\perp \|S\|+\|S'\|$. It follows that $W=0$, hence $S=S'$. 
\end{proof}

\begin{corollary}
Let $S\in \Irm_{k+1}([0,1] \times \R^d)$ with $\partial S \restrict (0,1) \times \R^d = 0$ be such that $\Var(S;\frarg)$ is non-atomic. Then, $S$ solves
\[
    \min \bigl\{\Var(U):  U \in \Nrm_{k+1}([0,1] \times \R^d), \, U\vert_t = S\vert_t \text{ for $\L^1$-a.e.\ $t$} \bigr\}. 
\] 
\end{corollary}

\begin{proof} Let $U$ be any competitor, i.e., $U \in \Nrm_{k+1}([0,1] \times \R^d)$ with $U\vert_t = S\vert_t \text{ for $\L^1$-a.e.\ $t$}$. By the Zero-slice Lemma~\ref{lemma:zero_slices} we have 
\[
S - U = W=\int_0^1 W_t \; \dd t, \qquad \partial W_t = 0,
\]
where $W_t \in \Nrm_{k+1}([0,1] \times \R^d)$ is supported in $\{t\}\times\R^d$. By Lemma~\ref{lemma:W_perp_S}, $\|W\|\perp \|S\|$, which entails $\|U\|=\|S\|+\|W\|$. As a consequence, $\Var(U)=\Var(S)+\Var(W)\ge \Var(S)$. 
\end{proof}

\subsection{Comparison of different flat norms}

As we already discussed, the normal and the integral homogeneous flat distance between two integral currents may differ, and it is an open problem whether they are even equivalent up to some dimensional constant (see \cite{Brezis_Mironescu} and \cite{Young18} for some recent developments). Theorem~\ref{thm:Var_equals_TV} is relevant also for this problem, since the equivalence between the two norms can be reformulated as an equivalence between the total variation of curves. The results in this section are not used anywhere else in this paper.

\begin{lemma}\label{lemma:equiv_bv_curves_norms}
The following are equivalent:
\begin{enumerate}
    \item[{\rm(i)}] $\Fhom_\Irm$ and $\Fhom$ are equivalent norms on $\Irm_k(\R^d)$, i.e., there exists a constant $C$ such that
    \[
    \Fhom_\Irm(T)\leq C \Fhom(T)\qquad\text{for every $T\in\Irm_k(\R^d)$}.
    \]
    \item[{\rm(ii)}] Every curve $[0,1]\ni t\mapsto T_t\in\Irm_k(\R^d)$ that is $\BV$ with respect to $\Fhom$ is also $\BV$ with respect to $\Fhom_\Irm$.
\end{enumerate}
\end{lemma}

\begin{proof}
Note first that without loss of generality we only need to consider the boundaryless case.

${\rm(i)}\implies {\rm(ii)}$ is clear. Let us prove the converse. We will show the contrapositive: if $\Fhom$ and $\Fhom_\Irm$ are \emph{not} equivalent, then there exists a curve that is $\BV$ with respect to $\Fhom$ but not with respect to $\Fhom_\Irm$. Supposing that $\Fhom$ and $\Fhom_\Irm$ are not equivalent we can find a sequence $T_j\in\Irm_k(\R^d)$, $j=1,2,\ldots$ such that
\[
\Fhom(T_j)=\frac{1}{j^2},\qquad \Fhom_\Irm(T_j)\geq \frac{1}{j}.
\]
Let us define $\gamma:[0,1]\to\Irm_k(\R^d)$ by
\[
\gamma(t):=
\begin{cases}
T_j & \text{if $t\in(2^{-2j},2^{-2j+1}]$}\\
0 & \text{if $t\in(2^{-2j+1},2^{-2j+2}]$} 
\end{cases}.
\]
Then 
\[
\Fhom\dash\eV(t\mapsto T_t)=\sum_j \frac{1}{j^2}<+\infty,\qquad \Fhom_\Irm\dash\eV(t\mapsto T_t)=\sum_j \frac{1}{j}=+\infty,
\]
therefore $\gamma$ is $\BV$ with respect to $\Fhom$, but not $\Fhom_\Irm$.
\end{proof}

What Theorem~\ref{thm:Var_equals_TV} achieves is proving the following: every curve $[0,1]\ni t\mapsto T_t\in\Irm_k(\R^d)$ such that
\begin{enumerate}
    \item it comes from the slicing of a space-time $\AC$ current (and thus it is $\BV$ with respect to $\Fhom_\Irm$), with $\Var(S;\frarg)$ non-atomic; 
    \item $\Mbf(T_t)\leq C$ for every $t\in [0,1]$, for some constant $C$;
\end{enumerate}
is also $\BV$ with respect to $\Fhom$, with additionally $\Fhom\dash\eV(t\mapsto T_t)=\Fhom_\Irm\dash\eV(t\mapsto T_t)$.

The proof of the following result is an application of the Zero-slice lemma. However its proof also requires space-time arguments that will be developed in the next section and is thus postponed to Section~\ref{subsec:equivalence}.

\begin{proposition}[Generic asymptotic equivalence of flat norms]\label{prop:asymptotic_equivalence} 
Let $S\in \Irm_{k+1}([0,1] \times \R^d)$ be such that $\partial S \restrict (0,1) \times \R^d = 0$ and such that $\Var(S;\frarg)$ is non-atomic, and let $t\mapsto T_t=S(t)$ be the corresponding pointwise-defined continuous representative. Then, for $\Var(S; \frarg)$-a.e.\ $t$, 
\begin{equation}\label{eq:asymptotic_equivalence}
\lim_{h\to 0} \frac{\Fhom_\Irm(T_{t+h}-T_t)}{\Fhom(T_{t+h}-T_t)}=1, 
\end{equation}
where we stipulate that the ratio is $1$ when the denominator is $0$ (in that case, $T_t = T_{t+h}$). 
\end{proposition}

\begin{remark}
The only case when the previous proposition does not say anything is when $\Var(S;[0,1])=0$. However in this case, by Proposition~\ref{prop:TV_leq_Var}, we would have $\Fhom_\Irm\dash\pV(t\mapsto S(t);[0,1])=0$, hence the path would be constant and the flat norms appearing in~\eqref{eq:asymptotic_equivalence} would both be zero.

It is also interesting to compare the conclusion of Proposition~\ref{prop:asymptotic_equivalence} with the following fact (see~\cite[Corollary~2.7.5]{BuragoBuragoIvanov}): if $\gamma:[0,1]\to (X,\mathsf{d})$ is a curve with finite pointwise variation, then for $\L^1$-a.e.\ $t\in[0,1]$ either
\[
\liminf_{h\to 0}\frac{\mathsf{d}\dash\pV(\gamma;[t,t+h])}{|h|}=0
\]
or
\[
\lim_{h\to 0} \frac{\mathsf{d}(\gamma(t),\gamma(t+h))}{\mathsf{d}\dash\pV(\gamma; [t,t+h])} = 1.
\]
From this one can deduce that, setting 
\[
N := \left\{t \in [0,1]: \liminf_{h\to 0} \frac{\Fhom\dash\pV(t\mapsto T_t; [t,t+h])}{|h|} = 0\right\},
\]
for $\L^1$-a.e.\ $t\in[0,1]\setminus N$ it holds that
\[
\lim_{h\to 0} \frac{\Fhom_\Irm(T_{t+h}-T_t)}{\Fhom(T_{t+h}-T_t)}=1.
\]
\end{remark}

\section{Space-time rectifiability}\label{sec:space-time}

We next turn our attention to understanding which paths of integral currents are ``space-time rectifiable'', meaning that the path consists of the slices of a space-time \emph{integral} current. The main result is the Rectifiability Theorem~\ref{thm:rect}, which requires bounded $\Fhom_{\Irm}$-variation. As a complementary result, in Proposition~\ref{prop:kampschulte}, we also investigate the case where we \emph{a-priori} know that the path solves the geometric transport equation for some driving vector field.

\subsection{Rectifiability under a BV assumption}\label{sc:rect}

The following theorem is space-time rectifiability result for general integral currents, i.e., we do not a-priori assume that the path is the solution of the geometric transport equation (for some vector field). In this general setting, we need to require a BV-type continuity property in time.

\begin{theorem}[Rectifiability] \label{thm:rect}
Let $t \mapsto T_t \in \Irm_k(\R^d)$, $t \in [0,1]$, with $\partial T_t = 0$ for every $t \in [0,1]$, and such that 
    \[
    \sup_{t \in [0,1]} \Mbf(T_t)  < \infty, \qquad  \Fhom_{\Irm}\dash\eV(t \mapsto T_t;[0,1])<\infty. 
    \]
    Then, there exists $S \in \Irm_{k+1}(\R\times \R^d)$ with $\partial S \restrict {(0,1)\times \R^d}=0$ such that:
    \begin{enumerate}
        \item[{\rm(a)}] $S\vert_t = T_t$ for $\L^1$-a.e.\ $t \in [0,1]$; 
        \item[{\rm(b)}] $\Var(S; \frarg) = \Fhom_\Irm\dash\eV(t \mapsto T_t; \frarg)$ as measures on $[0,1]$. 
    \end{enumerate}
\end{theorem}

\begin{proof}
The basic idea is to construct a piecewise constant (in time) approximation with good properties.

{\rm(a)}.
Let $\eps > 0$ and assume without loss of generality that 
\[
  \Fhom_{\Irm}\dash\pV(t \mapsto T_t;[0,1]) \leq \Fhom_{\Irm}\dash\eV(t \mapsto T_t;[0,1]) + \eps.
\]
For each $N \in \N$, define the uniform partition $0 = t_0 < t_1 < \ldots < t_N=1$ of $[0,1]$, where $t_i := \frac{i}{N}$ for $i=0,\ldots, N$. For ease of notation, we set for $i= 0,\ldots, N$,
\[
T_i := T_{t_i} 
\]
and take for $i=0,\ldots, N-1$ the currents $W_i \in \mathrm I_{k+1}(\R^d)$ that are optimal for $\Fhom_\Irm(T_{i+1}-T_i)$, i.e. such that
\[
T_{i+1} - T_i = \partial W_i \quad \text{ and } \quad \mathbf \Fhom_{\Irm}(T_{i+1} - T_i) = \mathbf{M}(W_i).
\]
We can now define the following $(k+1)$-integral current in $[0,1]\times \R^d$:
\begin{equation}\label{eq:def_S_N}
S_N:= -\sum_{i=0}^{N-1} \big(\bbb{t_i, t_{i+1}} \times T_i + \delta_{t_{i+1}} \times W_i\big).
\end{equation}
Here, $\bbb{t_i, t_{i+1}}$ denotes the (integral) $1$-current $\ee_t \L^1 \restrict [t_i, t_{i+1}]$ ($\ee_t$ pointing in the positive time direction).

Recalling formula~\eqref{eq:boundary_of_product} for the boundary of the product, we have 
\begin{align}
    \partial S_N
    &=-\sum_{i=0}^{N-1} \big(\delta_{t_{i+1}} \times T_i - \delta_{t_{i}} \times T_i + \delta_{t_{i+1}} \times \partial W_i\big)\nonumber \\
    & = -\sum_{i=0}^{N-1}\bigl( - \delta_{t_{i}} \times T_i + \delta_{t_{i+1}} \times T_{i+1} \bigr)\nonumber\\
    & = -\delta_{1} \times T_{1} + \delta_{0} \times T_{0}. \label{eq:boundary_S_N}
\end{align}
For the mass and boundary mass of $S_N$ we estimate
\begin{align*}
\mathbf{M}(S_N) &\le \frac{1}{N}\sum_{i=0}^{N-1}  \mathbf{M}(T_i) + \sum_{i=0}^{N-1}\mathbf{M}(W_i),  \\
\mathbf{M}(\partial S_N) &\le \mathbf{M}(T_{1}) + \mathbf{M}(T_{0}).
\end{align*}
Since we are assuming a uniform bound on the normal mass of $T_i$, we have 
\[
 \frac{1}{N}\sum_{i=0}^{N-1}\mathbf{M}(T_i) \leq C < \infty.
\]
For the two other terms we rely on the bound on the pointwise variation with respect to the homogeneous integral flat norm, which gives 
\[
\sum_{i=0}^{N-1} \mathbf{M}(W_i) = \sum_{i=0}^{N-1} \Fhom_{\Irm}(T_{i+1} - T_{i}) \leq \Fhom_{\Irm}\dash\pV(t \mapsto T_t;[0,1]) < \infty. 
\]
This inequality also yields a bound on the space-time variation:
\begin{equation}\label{eq:bound_on_var} 
\Var(S_N; [0,1])
= \sum_{i=0}^{N-1} \mathbf{M}(W_i) \leq \Fhom_{\Irm}\dash\pV(t \mapsto T_t;[0,1])
\leq \Fhom_{\Irm}\dash\eV(t \mapsto T_t;[0,1]) + \eps.
\end{equation}
We have therefore shown a uniform bound on the normal mass of $S_N$, which is clearly an integral current by definition. By the Helly selection principle for space-time integral currents in~\cite[Theorem~3.7]{Rindler21a?},
we obtain that there exists an integral current $S$ and a subsequence $N_j$ such that
\begin{equation}\label{eq:compactness_of_slices}
\left\{\begin{aligned}
S_{N_j} &\weaksto S\\
S_{N_j}(t)&\weaksto S(t)\quad\text{for all but countably many $t\in[0,1]$}.
\end{aligned}\right.
\end{equation}
We are left to show that $S(t)=T_t$ for all but at most countably many $t\in[0,1]$. To this aim, we fix $ t\in [0,1]$ with $S_{N_j}(t)\weaksto S(t)$ and use formula~\eqref{eq:definition_slicing} for the slicing of $S_{N_j}$ to get
\[
S_{N_j}|_{t}=(\partial S_{N_j})\restrict {\{\tbf< t\}}-\partial(S_{N_j}\restrict{\{\tbf< t\}})
\]
Let us define $\bar i$ to be the index such that $t_{\bar i}< t<t_{\bar{i}+1}$, that is, $\bar i=\lfloor N t\rfloor$. Then, recalling~\eqref{eq:def_S_N}, we obtain
\[
S_{N_j} \restrict {\{\tbf < t\}}=-\bbb{t_{\bar i}, t} \times T_{\bar i}-\sum_{i=0}^{\bar i-1} \big(\bbb{t_i, t_{i+1}} \times T_i + \delta_{t_{i+1}} \times W_i\big),
\]
and thus
\[
\partial(S_{N_j}\restrict {\{\tbf < t\}})=-(\delta_{t}-\delta_{t_{\bar i}})\times T_{\bar i}-\partial\left(\sum_{i=0}^{\bar i-1} \big(\bbb{t_i, t_{i+1}} \times T_i + \delta_{t_{i+1}} \times W_i\big)\right).
\]
On the other hand, from~\eqref{eq:boundary_S_N} we have
\[
(\partial S_{N_j})\restrict {\{\tbf< t\}}=\delta_{t_{\bar i}}\times T_{\bar i}-\partial \left(\sum_{i=0}^{\bar i-1} \big(\bbb{t_i, t_{i+1}} \times T_i + \delta_{t_{i+1}} \times W_i\big)\right).
\]
It follows that 
\[
S_{N_j}|_t=\delta_t\times T_{\bar i}=\delta_{t}\times T_{\lfloor N_j t\rfloor/N_j}\qquad\text{and}\qquad S_{N_j}( t)=T_{\lfloor N_j t\rfloor/N_j}.
\]
Since $t \mapsto T_t$ has finite variation, for all but at most countably many values of $ t\in [0,1]$ the left limit exists, that is
\[
T_{ t}=\wslim_{s\to t^-} T_s.
\]
Taking also~\eqref{eq:compactness_of_slices} into account, we deduce that for all but countably many $ t\in[0,1]$ it holds that
\[
S( t)=\wslim_{j\to\infty} S_{N_j}( t)=\wslim_{j\to\infty} T_{\lfloor N_j t\rfloor/N_j}=T_{t}.
\]
In the case that $t\mapsto T_t$ is continuous, equality follows for every $ t\in [0,1]$.

{\rm(b)}. The inequality $\Fhom_{\Irm}\dash\eV(t \mapsto T_t; \frarg) \le \Var(S; \frarg)$ follows from Proposition~\ref{prop:TV_leq_Var}. To prove the reverse inequality it is sufficient to show that the total masses of the two measures are equal. This is indeed a simple consequence of the lower semicontinuity of the variation and of the piecewise constant construction. We have 
\[
\Var(S;[0,1]) \leq \liminf_{N \to \infty} \Var(S_N;[0,1]) \leq \Fhom_{\Irm}\dash\eV(t \mapsto T_t;[0,1]) + \eps,
\]
where we have used~\eqref{eq:bound_on_var}. We conclude after letting $\eps \to 0$.
\end{proof}

\begin{remark} A straightforward adaptation of the proof shows that space-time currents satisfying {\rm(a)} can be constructed also without the assumption $\partial T_t =0$, provided that
\[
\sup_{t \in [0,1]} \Mbf(T_t) +\Mbf(\partial T_t) < \infty \qquad \text{and} \qquad \Fbf_{\Irm}\dash\eV(t \mapsto T_t;[0,1]) < \infty.
\]
Notice that we have replaced the \emph{homogeneous} flat norm with the \emph{inhomogeneous} flat norm
\begin{equation*}
 \Fbf_\Irm(T) := \inf\bigl\{ \Mbf(Q)+\Mbf(R) : Q\in\Irm_{k+1}(\R^d), \, R\in\Irm_k(\R^d), \, T=\partial Q+R \bigr\}.
\end{equation*}
\end{remark}

\begin{remark} \label{rem:rect_variational}
In the context of Theorem~\ref{thm:rect}, one can show that the current
$S$ constructed above is a minimiser for the following problem: 
        \[
        \min \bigl\{\Var(R;[0,1]):  R \in \Irm_{k+1}([0,1] \times \R^d), \,\partial R=0,\, R\vert_t = T_t \text{ for $\L^1$-a.e.\ $t$} \bigr\}. 
        \]
Indeed, for any competitor $R$, since the slices of $R$ coincide with those of $S$, we have 
\[
\Var(S;[0,1]) = \Fhom_{\Irm}\dash \eV(t \mapsto S\vert_t;[0,1])  = \Fhom_{\Irm}\dash \eV(t \mapsto R\vert_t;[0,1]) \le \Var(R;[0,1]),
\]
where we have used Proposition~\ref{prop:TV_leq_Var}~(c).
If $\Var(S;\frarg)$ is non-atomic, this conclusion can be strengthened in the following sense: the current $S$ constructed above is a minimiser of the following problem: 
\[
\min \bigl\{\Var(R;[0,1]):  R \in \Nrm_{k+1}([0,1] \times \R^d), \,\partial R=0,\, R\vert_t = T_t \text{ for $\L^1$-a.e.\ $t$} \bigr\},
\]
where now the class of competitors includes all \emph{normal} boundaryless currents with slices equal to the $T_t$'s. Indeed, for any such normal competitor $R$, by the Zero-slice Lemma~\ref{lemma:zero_slices} in conjunction with Lemma~\ref{lemma:W_perp_S}, we have 
\[
\|S\|\perp \|S - R\|,
\]
whence
\[
\Var(R;[0,1]) = \Var(S;[0,1])+\Var(S-R;[0,1]) \ge \Var(S;[0,1]).
\]
\end{remark}

Let us also state the following version of Theorem~\ref{thm:rect} for normal currents.

\begin{proposition} \label{prop:gluing_normal}
Let $t \mapsto T_t \in \Nrm_k(\R^d)$, $t \in [0,1]$, with $\partial T_t = 0$ for every $t \in [0,1]$, and such that 
    \[
   \sup_{t \in [0,1]} \Mbf(T_t)  < \infty, \qquad \Fhom\dash\eV(t \mapsto T_t;[0,1])<\infty. 
    \] 
    Then, there exists $U \in \Nrm_{k+1}(\R\times \R^d)$ with $\partial U \restrict {(0,1)\times \R^d}=0$ such that:
    \begin{enumerate}
        \item[{\rm(a)}] $U\vert_t = T_t$ for $\L^1$-a.e.\ $t \in [0,1]$; 
        \item[{\rm(b)}] $\Var(U;\frarg) = \Fhom\dash\eV(t \mapsto T_t; \frarg)$.
    \end{enumerate}
\end{proposition}

\begin{proof}
The proof is the same as the one for Theorem~\ref{thm:rect} with minor adaptations.   
\end{proof}

\subsection{On the equivalence of variations and flat norms}\label{subsec:equivalence}

We are now finally ready to prove Theorem~\ref{thm:Var_equals_TV} and Proposition~\ref{prop:asymptotic_equivalence}.

\begin{proof}[Proof of Theorem~\ref{thm:Var_equals_TV}]
Let $S \in \Irm_{k+1}([0,1] \times \R^d)$ with $\Var(S; \frarg)$ non-atomic. Let us consider the \emph{normal} current $U$ obtained through Proposition~\ref{prop:gluing_normal} starting from the map $t \mapsto S(t)=\pbf_* (S|_t)$.
By the Zero-slice Lemma~\ref{lemma:zero_slices}, there is a positive measure $\kappa$ on $[0,1]$ and $W_t \in \Nrm_{k+1}([0,1] \times \R^d)$ for $\kappa$-a.e.\ $t$, which is supported in $\{t\}\times\R^d$, such that
\[
S-U=W=\int_0^1 W_t\; \dd \kappa(t) \qquad\text{and}\qquad
    	\partial W_t=0 \quad\text{for $\kappa$-a.e.\ $t$.}
\]
By Lemma~\ref{lemma:W_perp_S}, $\|S\|\perp \|W\|$. Then, using Proposition~\ref{prop:gluing_normal}~{\rm(b)} and Proposition~\ref{prop:TV_leq_Var}, we obtain 
\begin{align}
\Var(S;[0,1])+\Var(W;[0,1])&=\Var(U;[0,1]) \notag\\
& = \Fhom\dash\eV(U;[0,1]) \notag\\
& = \Fhom\dash\eV(S;[0,1]) \notag\\
& \leq \Fhom_{\Irm}\dash\eV(S;[0,1]) \notag\\
&\leq \Var(S;[0,1]), \label{eq:inequalities_in_variations}
\end{align}
hence $\Var(W;[0,1])=0$. The orienting multi-vector of $W$ always lies in $\Wedge_{k+1}(\{0\}\times \R^d)$, which gives the equality $\Var(W;[0,1])=\|W\|([0,1] \times \R^d)$, and thus $W=0$. Therefore, the inequalities in~\eqref{eq:inequalities_in_variations} are all equalities and the conclusion is reached.
\end{proof}

As a byproduct of the previous proof we also obtain the following.

\begin{corollary}\label{cor:S=U}
Let $S \in \Irm_{k+1}([0,1] \times \R^d)$ with $\partial S \restrict (0,1) \times \R^d = 0 $ and $\Var(S; \frarg)$ non-atomic. Let us consider the \emph{normal} current $U$ obtained through Proposition~\ref{prop:gluing_normal} starting from the map $t \mapsto S(t)=\pbf_* (S|_t)$. Then, $S=U$.
\end{corollary}

\begin{proof}[Proof of Proposition \ref{prop:asymptotic_equivalence}]
	Suppose the conclusion is not true. Then, there exists a set $B\subset\R$ with $\Var(S;B)>0$, a number $\delta>0$, and for every $t\in B$ a sequence $h_j(t)\to 0$ of positive numbers such that 
	\[
	\Fhom_\Irm(T_{t+h_j(t)}-T_t)\geq (1+\delta) \Fhom(T_{t+h_j(t)}-T_t)
	\]
	and we can assume without loss of generality that $\Fhom(T_{t+h_j(t)}-T_t)>0$ for every $j$ (if this were not true, we would have $T_{t+h_j(t)}=T_t$ and therefore the ratio would be $1$ by the convention agreed in the statement). 
	
	By a covering argument, for any $\rho>0$ we can select finitely many $t_1^\rho,\ldots,t_m^\rho$ and $h_i^\rho\leq \rho$ such that
	\[
	\sum_i \Var(S;[t_i^\rho,t_i^\rho+h_i^\rho])\geq\frac12 \Var(S;B).
	\]
	We consider a modification of the space-time piecewise-constant construction used in the proof of Theorem~\ref{thm:rect}. We build the currents $S^\rho$ in the following way: outside the slabs $[t_i^\rho,t_i^\rho+h_i^\rho]\times\R^d$ we set $S^\rho=S$; inside each slab $[t_i^\rho,t_i^\rho+h_i^\rho]\times\R^d$ we instead replace $S$ with 
	\[
	\bbb{t_i^\rho,t_i^\rho+h_i^\rho}\times S(t_i^\rho)+\delta_{t_i^\rho+h_i^\rho}\times R_i^\rho,
	\]
	where $R_i^\rho$ denotes an optimal \textit{normal} filling between $S(t_i^\rho)$ and $S(t_i^\rho+h_i^\rho)$. Then, $S^\rho$ converge, up to a subsequence, to some normal current $U$, with the property that $U|_t=S|_t$ for $\L^1$-a.e.\ $t$. By the Zero-slice Lemma~\ref{lemma:zero_slices},
	\[
	S-U=\int_0^1 W_t\; \dd\lambda(t)=:W.
	\]
	By Lemma~\ref{lemma:W_perp_S}, $\|W\|\perp \|S\|$.
	From this it follows that $\|U\|=\|S\|+\|W\|$, and therefore
	\begin{equation}\label{eq:variation_USW}
		\Var(U)=\Var(S)+\Var(W).
	\end{equation}
	Observe now that
	\begin{align*}
		\Var(S^\rho;[t_i^\rho,t_i^\rho+h_i^\rho])&=\Fhom(T_{t_i^\rho+h_i^\rho}-T_{t_i^\rho})\\
		&\leq (1+\delta)^{-1} \Fhom_\Irm(T_{t_i^\rho+h_i^\rho}-T_{t_i^\rho})\\
		&\leq (1+\delta)^{-1} \Var(S;[t_i^\rho,t_i^\rho+h_i^\rho]),
	\end{align*}
	which by summing yields
	\begin{align*}
		\sum_i \Var(S^\rho;[t_i^\rho,t_i^\rho+h_i^\rho])&\leq \frac{1}{1+\delta}\sum_i \Var(S;[t_i^\rho,t_i^\rho+h_i^\rho])\\
		&=\sum_i \Var(S;[t_i^\rho,t_i^\rho+h_i^\rho])-\frac{\delta}{1+\delta}\sum_i \Var(S;[t_i^\rho,t_i^\rho+h_i^\rho])\\
		& \leq \sum_i \Var(S;[t_i^\rho,t_i^\rho+h_i^\rho])-\frac{\delta}{2(1+\delta)}\Var(S;B).
	\end{align*}
	Thus,
	\[
	\Var(S^\rho;[0,1])\leq \Var(S;[0,1])-\frac{\delta}{2(1+\delta)}\Var(S;B).
	\]
	From the lower semicontinuity of the variation it holds that
	\[
	\Var(U;[0,1])\leq \liminf_{\rho\to 0} \Var(S^\rho;[0,1])\leq \Var(S;[0,1])-\frac{\delta}{2(1+\delta)}\Var(S;B),
	\]
	but this contradicts~\eqref{eq:variation_USW}.
\end{proof}

\subsection{Gluing of transported currents}

The following proposition constitutes another approach to turn a path of integral currents into a space-time current and shares some similarities with \cite[Theorem 6.11]{Kampschulte17PhD}. It only applies when we know \emph{a-priori} that our path of integral currents satisfies the geometric transport equation. In this case, the resulting space-time current is in general only normal.

\begin{proposition}\label{prop:kampschulte}
    Let $t \mapsto T_t \in \Irm_k(\R^d)$, $t \in [0,1]$, such that $\partial T_t=0$ for every $t \in [0,1]$ and such that there exists a vector field $b \in {\rm L}^1(\L^1\otimes \|T_t\|)$ with $b_t$ normal to $\vec{T}_t$ for $\|T_t\|$-a.e.\ $x$ and $\L^1$-a.e.\ $t$, and such that 
    \begin{equation}\label{eq:usual_PDE}
     \frac{\dd}{\dd t} T_t + \Lcal_{b_t} T_t = 0.
\end{equation}
Let $\vec{T} \colon [0,1] \times \R^d \to \R \times \R^d$ be the space-time $k$-vector field defined by $\vec{T}(t,x) := \iota_t(\vec{T}_t(x))$, where $\iota_t:x\mapsto (t,x)$. We set
\begin{equation*}
    	U := [(1, b) \wedge \vec{T}] \L^1 \otimes \|T_t\|. 
\end{equation*}
Then, the following statements hold true: 
\begin{enumerate}
    \item[\rm{(i)}] $U$ is normal, i.e., $U \in \Nrm_{k+1}([0,1] \times \R^d)$ and $(\partial U) \restrict {(0,1) \times \R^d} = 0 $; 
    \item[\rm{(ii)}] $U$ has simple unit orienting $(k+1)$-vector and $\|U\|=|(1, b)|\L^1 \otimes \|T_t\|$ as measures on $[0,1] \times \R^d$;  
    \item[\rm{(iii)}] $\tbf_{\#}\|U\| \ll \L^1$;  
    \item[\rm{(iv)}] $\pbf_{*}(U\vert_t) = T_t$ for $\L^1$-a.e.\ $t \in (0,1)$. 
\end{enumerate} 
\end{proposition}

Notice that the assumption that $b_t$ is normal to $\vec{T}_t$ is not restrictive, as one can always reduce to this case by adding a tangential component to $b_t$, without changing \eqref{eq:usual_PDE}. 

\begin{proof}
{\rm(i)}. We first show that $U$ has finite mass. Indeed, for every smooth $(k+1)$-form $\omega$ we have (here and in the following $\L^1$ is always understood with respect to $t$)
\begin{align*}
	|\langle U,\omega \rangle | & = |\langle [(1, b) \wedge \vec{T}] \L^1 \otimes \|T_t\| ,\omega \rangle |  \\ 
	& = \left| \int_0^1 \int_{\R^d} \langle (1,b(t,x))\wedge \vec{T}(t,x), \omega(t,x)\rangle  \; \dd \|T_t\|(x) \; \dd t \right| \\ 
	& \le \|\omega\|_\infty  \int_0^1 \int_{\R^d} |(1,b(t,x))\wedge \vec{T}(t,x)|  \; \dd \|T_t\|(x) \; \dd t  \\ 
	&\le \|\omega\|_\infty  \int_0^1 \int_{\R^d} |(1,b(t,x))|  \; \dd \|T_t\|(x) \; \dd t
\end{align*} 
and the claim follows since $b \in {\rm L}^1(\L^1\otimes \|T_t\|)$.

One can check that 
\begin{equation}\label{eq:int_delta_times_T_t}
\vec{T} \L^1\otimes \|T_t\| = \int_0^1 \delta_{t} \times T_t \; \dd t \qquad \text{as $k$-currents in $\R\times \R^d$.} 
\end{equation}
Indeed, for any $\psi \in \Crm^\infty_c((0,1))$ and every smooth $k$-form $\omega \in \Dscr^k(\R^d)$ we have 
\begin{align*}
\int_0^1 \langle \delta_{t} \times T_t,\psi\omega\rangle \; \dd t & = \int_0^1 \psi(t) \int_{\R^d} \langle \vec{T}_t(x) ,\omega(x)\rangle \; \dd\|T_t\|(x) \; \dd t = \langle \vec{T} \L^1\otimes \|T_t\| , \psi\omega\rangle, 
\end{align*}
and this suffices by the usual density arguments. 
Together with the formula for the boundary of a product,~\eqref{eq:int_delta_times_T_t} yields 
\[
\partial [\vec{T} \L^1\otimes \|T_t\|] = \int_0^1 \partial(\delta_{t} \times T_t )\; \dd t =0.
\]
Therefore, we have
\[
\partial U = \partial \bigl[ [(1, b) \wedge \vec{T}] \L^1 \otimes \|T_t\|\bigr] = - \Lcal_{(1,b_t)} (\vec{T} \L^1 \otimes \|T_t\|) = - \Lcal_{(1,b_t)}\left(\int_0^1\delta_{t} \times T_t \; \dd t\right). 
\]
Resorting to Corollary~\ref{corollary:product} and using also~\eqref{eq:non_compact_support} with $\psi = \1_{[0,1]}$, we obtain 
\[
\partial U = - \Lcal_{(1,b_t)}\left(\int_0^1\delta_{t} \times T_t \; \dd t\right) = -\int_0^1\Lcal_{(1,b_t)} \left( \delta_{t} \times T_t\right) \; \dd t = \delta_{0} \times T_0 - \delta_{1} \times T_1.
\]

{\rm(ii)}. By (i), $U$ can be written in the form $U = \vec{U} \|U\|$, where $\vec{U}$ is a unit $(k+1)$-vector field and $\|U\|$ is a non-negative measure. It is then clear from the definition that
\[
\vec{U} = \frac{(1, b) \wedge \vec{T}}{|(1,b)|}, \qquad \|U\| = |(1,b)|\L^1 \otimes \|T_t\|.
\]
Hence, we immediately see that $\vec{U}$ is a simple $(k+1)$-vector.

{\rm(iii)}. We have already shown that $\|U\| = |(1,b)|\L^1 \otimes \|T_t\|$ and therefore for every $\L^1$-null set $N \subset \R$ we have 
\[
\tbf_{\#}(\|U\|)(N) = \|U\|(N \times \R^d) = \int_{N \times \R^d} |(1,b)| \; \dd\|T_t\| \; \dd t
= 0. 
\]

{\rm(iv)}. By definition, for $\L^1$-a.e.\ $t \in (0,1)$ we have 
\[
U\vert_t = \partial(U \restrict {\{\tbf < t\}}).  
\]
For every smooth $k$-form $\eta$ in $\R^d$ the following chain of equalities holds:
\begin{align*}
\langle \pbf_{*}(U\vert_t),\eta \rangle &= \langle \pbf_{*}\partial (U \restrict {\{\tbf < t\}}),\eta \rangle \\
&= \langle\partial  \pbf_{*}(U \restrict {\{\tbf < t\}}),\eta \rangle \\ 
& = \langle \pbf_{*} (U\1_{\{\tbf < t\}}), d\eta \rangle \\
& = \langle U\1_{\{\tbf < t\}} , \pbf^{*}d\eta \rangle \\
& = \int_{0}^{t} \int_{\R^d} \langle (1,b(s,x))\wedge \vec{T}(s,x), \pbf^{*}d\eta(x)\rangle \; \dd \|T_s\|(x) \; \dd s\\ 
& = \int_{0}^t \int_{\R^d} \langle b(s,x)\wedge \vec{T}_s(x), d\eta(x) \rangle \; \dd \|T_s\|(x) \; \dd s\\ 
& = \int_{0}^t \langle b_s \wedge T_s, d\eta \rangle \,  \; \dd s \\
& = \int_{0}^t \langle  \partial (b_s \wedge T_s), \eta \rangle \,  \; \dd s \\ 
& = \int_{0}^t \frac{\dd}{\dd s}\langle T_s,\eta\rangle \; \dd s  \\ 
& = \langle T_t, \eta \rangle,
\end{align*}
where the second-to-last equality holds by~\eqref{eq:usual_PDE} and the last equality holds by the absolute continuity of the path $t\mapsto \langle T_t,\eta\rangle$, which follows from~\eqref{eq:usual_PDE}.
\end{proof}

\section{Negligible criticality and advection theorem} \label{sc:advection}

We are now ready to understand one of the core aspects of this work, namely the dynamics of slices of $\AC$ integral currents. Given $S\in\Irm^{\AC}_{k+1}([0,1] \times \R^d)$, with $\partial S \restrict (0,1) \times \R^d = 0$,   we prove the \emph{advection theorem}, namely, that the slices $(S\vert_t)_{t}$ are advected by a vector field (i.e., they solve a transport equation with a vector field in ${\rm L}^1( \int_0^1  \| S\vert_t \| \; \dd t)$) if and only if $S$ has the \emph{negligible criticality property} or, equivalently, if and only if $S$ has the \emph{Sard property}, both of which are defined below.

\subsection{Negligible criticality condition and Sard-type property}

The considerations of the previous section inspire the following definitions, which turn out to be central:

\begin{definition}\label{def:NH}
A space-time current $S \in \Irm_{1+k}([0,1] \times \R^d)$ is said to satisfy the \textbf{negligible-criticality condition} (with respect to the map $\tbf$) if 
\begin{equation}\label{eq:NC}\tag{NC}
\| S\| \restrict \Crit(S) = 0,
\end{equation}
where $\Crit(S)$ is the critical set of $S$ defined in~\eqref{eq:definition_critical_set}. 
\end{definition}

The condition~\eqref{eq:NC} means that for $\|S\|$-a.e.\ $(t,x)$, $\mathrm{span}\,(\vec{S}(t,x))\not\subseteq \{0\}\times \R^d$ (i.e., the approximate tangent space to $S$ has almost always a non-trivial time component). Equivalently, this can be stated by requiring that
\begin{equation*}
\vec \S(t,x) \restrict d \tbf(t,x) \neq 0\quad\text{for $\|\S\|$-a.e.\ $(t,x)$}.
\end{equation*}

We also consider the following (in general weaker) condition, for which we recall that we denote by $\{\mu_t\}$ the disintegration of $\|S\|$ with respect to $\tbf$ and
\[
  \lambda := \L^1 +(\tbf_{\#}\|S\|)^s,
\]
as in~\eqref{eq:disintegration}.

\begin{definition}\label{def:sard}
A current $S \in \Irm_{1+k}([0,1] \times \R^d)$ is said to have the \textbf{Sard property} (with respect to the map $\tbf$) if
\begin{equation}\label{eq:Sard}\tag{S}
    \mu^s_t = 0 \qquad \text{for $\L^1$-a.e.\ $t \in \R$}.
\end{equation}
\end{definition}

It might not be evident from Definition~\ref{def:sard} why the vanishing of $\mu_t^s$ is called a ‘‘Sard-type’’ property. The following lemma and example explain our choice of terminology and clarify the relationship between~\eqref{eq:NC} and the Sard property. 

\begin{lemma} \label{lemma:sard_equivalent} Let $S \in \Irm_{1+k}([0,1] \times \R^d)$. Then, the following statements are equivalent: 
	\begin{enumerate}
		\item[{\rm(i)}] $S$ has the Sard property, i.e., $\mu^s_t = 0$ for $\L^1$-a.e.\ $t \in \R$; 
		\item[{\rm(ii)}] $\tbf_{\#} (\| S \|\restrict \Crit(S) )$ is singular with respect to $\L^1$.  
	\end{enumerate}
Furthermore,~\eqref{eq:NC} implies both of them. 
\end{lemma}

\begin{proof} It is clear that~\eqref{eq:NC} implies {\rm(ii)}, therefore it suffices to show the equivalence between {\rm(i)} and {\rm(ii)}.
For this, it is useful to recall the definition of the measure $\nu:=\|S\|\restrict C$, where $C := \Crit(S)$, introduced in the proof of Theorem~\ref{thm:struct}.

${\rm(i)} \implies {\rm(ii)}$. Assume first that $S$ has the Sard property, i.e., $\mu^s_t = 0$ for $\L^1$-a.e.\ $t \in \R$. By the very definition of $\mu^s_t$ (see~\eqref{eq:def_of_sing_sard}), we have that $\nu_t = 0$ for $\L^1$-a.e.\ $t \in \R$ and thus the disintegration of the measure $\nu$ with respect to $\tbf$ and $\lambda$ has the form
\begin{equation}\label{eq:disint_of_nu}
    \nu = \int_0^1 \nu_t \; \dd\lambda^s(t).
\end{equation} 
Hence, $\tbf_{\#}\nu \ll \lambda^s$ is indeed singular with respect to Lebesgue measure on $\R$.  
	
${\rm(ii)} \implies {\rm(i)}$. The converse is very similar: if $\tbf_{\#} (\| S \|\restrict C) \perp \L^1$, then we must have $\lambda\perp \L^1$ and therefore the disintegration of $\nu$ with respect to $\tbf$ and $\lambda$ has again the form~\eqref{eq:disint_of_nu}, i.e., $\nu_t = 0$ for $\L^1$-a.e.\ $t$. \end{proof}

Normally one defines Sard-like properties for maps, meaning that the image of the critical set is negligible. The connection with our notion given by Definition~\ref{def:sard} is best explained by the following example:

\begin{example}
Let $f:\R^d\to \R$ be $\Crm^1$, let $\Gamma\subset [0,1] \times \R^d$ denote its graph and let $S$ be the natural integral $d$-current with multiplicity $1$ associated with $\Gamma$. Then $\mathrm{Crit}(S)$ coincides with the set of points $(f(x),x)\in\Gamma$ with $x$ critical point of $f$, that is, $\nabla f(x)=0$. Indeed, the tangent space at $(f(x),x)$ is spanned by $\partial_i((f(x),x))=(\partial_i f(x),\ee_i)$ for $i=1,\ldots,d$, and thus $\nabla^S\tbf=0$ at $(f(x),x)$ if and only if $\partial_i f(x)=0$ for $i=1,\ldots,d$. By Lemma~\ref{lemma:sard_equivalent} it follows that $f$ has the classical Sard property (namely, $f(\{x:\nabla f(x)=0\})$ is $\L^1$-negligible) if and only if $S$ has the Sard property in the sense of Definition~\ref{def:sard}.
\end{example}

\subsection{The absolutely continuous case}\label{ss:ac_currents}

We now restrict our analysis to the case where the current $S$ is boundaryless and $\AC$ in the sense introduced in Section~\ref{sc:variation}. The following proposition contains a characterisation of $\AC$ space-time current via time projections.

\begin{proposition}\label{prop:projection_AC}
Let $S \in \Irm_{k+1}([0,1] \times \R^d)$ with $\partial S \restrict (0,1) \times \R^d = 0$. Then, $S \in \Irm^{\AC}_{1+k}([0,1] \times \R^d)$ if and only if $\lambda^s = 0$, i.e., 
\begin{equation*}
\tbf_{\#} (\| S\| \restrict \Crit(S))\ll \L^1.
\end{equation*} 
\end{proposition}

\begin{proof}
	Assume first $S \in \Irm^{\AC}_{1+k}([0,1] \times \R^d)$. Observe that, for every $(t,x) \in C := \Crit(S)$, we have $|\pbf(\vec{S}(t,x))|=1$. Thus, for every open interval $(a,b) \subset [0,1]$ we have 
	\begin{align*}
		\tbf_{\#} (\| S\| \restrict C) ((a,b)) & =  \| S \|(C \cap ((a,b) \times \R^d)) \\
		& =   \int_{C \cap ((a,b) \times \R^d)}  |\pbf(\vec{S})| \; \dd\|S\|  \\ 
		& \le  \int_{(a,b) \times \R^d}  |\pbf(\vec{S})| \; \dd\|S\|  \\ 
		& = \Var(S, (a,b))
	\end{align*}
and the claim follows from the definition of $S \in \Irm^{\AC}_{1+k}([0,1] \times \R^d)$. 

For the opposite direction, suppose $\tbf_{\#} (\| S\| \restrict C)\ll \L^1$. Then, for every measurable set $A \subset [0,1]$ we have 
\begin{align*}
	\Var(S; A) &= \int_{A \times \R^d} |\pbf(\vec{S})| \; \dd \|S\|  \\
	& \le \|S\|(A \times \R^d)  \\ 
	& = \|S\| \restrict {C^c}(\tbf^{-1}(A) ) + \|S\| \restrict {C}(\tbf^{-1}(A)) \\ 
	 & = \tbf_{\#}(\|S\| \restrict C^c)(A) + \tbf_{\#}(\|S\|\restrict C)(A),
\end{align*}
whence 
\[
	\Var(S;\frarg)  \le \tbf_{\#}(\|S\| \restrict C^c) + \tbf_{\#}(\|S\|\restrict C) \qquad \text{as measures on $\R$}.
\]
It remains to observe that the first summand is always absolutely continuous with respect to $\L^1$ by Lemma~\ref{lemma:disint_non_critical_mass}, while the second is also absolutely continuous by assumption.
\end{proof}

The following structure result offers some equivalent conditions to the negligible criticality property for $\AC$ currents.

\begin{proposition} \label{prop:sard_equivalent_AC}
Let $S \in \Irm^{\AC}_{1+k}([0,1] \times \R^d)$. Then, the following are equivalent: 
	\begin{enumerate}
		\item[{\rm(i)}] $S$ has the~\eqref{eq:NC} property, i.e., $\| S \|\restrict C = 0$; 
		\item[{\rm(ii)}] it holds that $\tbf_{\#} (\| S \|\restrict C) \perp \L^1$; 
		\item[{\rm(iii)}] $S$ has the Sard property, i.e., $\mu^s_t = 0$ for $\L^1$-a.e.\ $t \in \R$; 
		\item[{\rm(iv)}] $\| S \| \ll  \int_0^1  \| S\vert_t \| \; \dd t$, that is, for every Borel set $A \subset [0,1] \times \R^d$
		\begin{equation*}
			\int_0^1 \|S\vert_t \|(A) \; \dd t = 0 \implies \|S\|(A) = 0. 
		\end{equation*}
	\end{enumerate}
Furthermore, if any of the above conditions holds, then the disintegration of the mass measure $\|S\|$ with respect to $\tbf$ and $\L^1$ has the form 
\[
	\| S \| = \int_0^1  |\nabla^S\tbf|^{-1}  \| S\vert_t \| \; \dd t. 
\]
\end{proposition}

\begin{proof}
	{\rm(i)} $\implies$ {\rm(ii)}. This is obvious. 
	
	{\rm(ii)} $\implies$ {\rm(iii)}. This was proven for $S \in \Irm_{k+1}([0,1] \times \R^d)$ in Lemma~\ref{lemma:sard_equivalent}.  
	
	{\rm(iii)} $\implies$ {\rm(iv)}. Assume that $S$ has the Sard property. By Proposition~\ref{prop:projection_AC} we also know $\lambda^s = 0$. Combining this with~\eqref{eq:structure_disint}, we obtain 
	\begin{equation*}
		\| S \|  = \int_0^1 |\nabla^S\tbf|^{-1}  \| S\vert_t \| \; \dd t, 
	\end{equation*}
	whence the claim follows.
	
	{\rm(iv)} $\implies$ {\rm(i)}.  By Lemma~\ref{lemma:disint_bella} we have $\|S\vert_t\|(C) = 0$ for $\L^1$-a.e.\ $t \in \R$. Thus,
	\begin{equation*}
		\int_0^1 \|S\vert_t \|(C) \; \dd t = 0
	\end{equation*}
	and this implies $\|S\|(C) = 0 $ by the assumed implication.
\end{proof}

\begin{remark}\label{rmk:case_k_zero}
If $k=0$, the Sard property is always satisfied, that is, for every $S \in \Irm_{1}(\R\times \R^d)$ it holds that $\tbf_{\#} (\| S \|\restrict C) \perp \L^1$, since by the area formula $\H^{1}(\tbf(C)) = 0$. So, the effects related to the criticality are not present in all of the classical theory of BV- or AC-maps (which can be recovered as the $k = 0$ endpoint of our theory).
\end{remark}

For $k > 1$, the Sard property is not always satisfied and it is possible to construct $\AC$ (even Lipschitz) integral currents that do not have the Sard property. We refer the reader to Section~\ref{sc:counterexample} for the construction of such an example.

\subsection{Geometric derivative and advection theorem}

In the following we will prove what we call the \emph{advection theorem}, namely that for space-time currents satisfying the negligible criticality condition~\eqref{eq:NC}, there exists an advecting vector field. Furthermore, also the converse holds. These results should be compared with e.g. \cite[Theorem 8.3.1]{AGS}, where a similar advection theorem is established within the class of probability measures.

Recall that, if $S\in\Irm^{\AC}_{k+1}([0,1] \times \R^d)$ then, by Proposition~\ref{prop:sard_equivalent_AC}, the condition~\eqref{eq:NC} is equivalent to $S$ having the Sard property or also to 
\[
\|S\| \ll \int_0^1  \| S\vert_t \| \; \dd t.
\]
We define the \textbf{geometric derivative} of such an $S$ as
\begin{equation}\label{eq:geom_deriv} 
  b(t,x) := \frac{\DD}{\DD t} S(t,x) := \frac{\pbf(\xi(t,x))}{\abs{\nabla^S \tbf(t,x)}},  \qquad (t,x) \in \Crit(S)^c,
\end{equation}
where $\xi = |\nabla^S \tbf|^{-1}\nabla^S \tbf$ (on $\Crit(S)^c$) was defined in~\eqref{eq:def_eta}. Observe that $b$ is well-defined $(\int_0^1  \| S\vert_t \| \; \dd t)$-almost everywhere (by Lemma~\ref{lemma:disint_bella}~(ii)) and under~\eqref{eq:NC}, both $\xi$ and $b$ are well-defined also $\|\S\|$-almost everywhere. Clearly,
\begin{equation} \label{eq:geom_deriv_est}
  \absBB{\frac{\DD}{\DD t} S} \leq \frac{1}{\abs{\nabla^S \tbf}}.
\end{equation}

\begin{theorem}[Advection] \label{thm:advection}
Let $S\in\Irm^{\AC}_{k+1}([0,1] \times \R^d)$ with $(\partial \S) \restrict {(0,1) \times \R^d}=0$ satisfy~\eqref{eq:NC}. Then, the geometric derivative $b := \frac{\DD}{\DD t} S$ defined in~\eqref{eq:geom_deriv} belongs to ${\rm L}^1( \int_0^1  \| S\vert_t \| \; \dd t)$ and it holds that 
\begin{equation}\label{eq:advection}
     \frac{\dd}{\dd t} \S(t) + \Lcal_{b_t} \S(t) = 0.
\end{equation}
Conversely, if there exists a vector field $b \in {\rm L}^1( \int_0^1  \| S\vert_t \| \; \dd t)$ such that~\eqref{eq:advection} holds, then $S$ satisfies~\eqref{eq:NC}. 
\end{theorem}

We will prove the theorem in the following two sections.

\subsection{Existence of the geometric derivative}

We first prove that~\eqref{eq:NC} implies the existence of  the geometric derivative defined in~\eqref{eq:geom_deriv}.

\begin{proposition} \label{prop:slices_are_advected}
Let $S\in\Irm^{\AC}_{k+1}([0,1] \times \R^d)$ with $(\partial \S) \restrict {(0,1) \times \R^d}=0$ satisfy~\eqref{eq:NC}. Then, the geometric derivative $b := \frac{\DD}{\DD t} S$ defined in~\eqref{eq:geom_deriv} belongs to ${\rm L}^1( \int_0^1  \| S\vert_t \| \; \dd t)$ and it holds that 
\begin{equation}\label{eq:PDEslices_new}
     \frac{\dd}{\dd t} \S(t) + \Lcal_{b_t} \S(t) = 0. 
\end{equation} 
\end{proposition}

\begin{proof} The integrability of $b$ is a simple consequence of coarea formula. Indeed, we have %\mnote{F: $\leq$}
	\begin{equation*}
			\int_0^1 \int_{\{t\} \times \R^d} |b| \; \dd \|S|_t\| \; \dd t 
			\leq \int_0^1 \int_{\{t\} \times \R^d} \frac{1}{|\nabla^S \tbf|} \; \dd \|S|_t\| \; \dd t = \Mbf(S)
	\end{equation*}
where we have applied~\eqref{eq:geom_deriv_est} and then~\eqref{eq:coarea_integral} with $f=\tbf$ and $g=|\nabla^S \tbf|^{-1}$.

We now show that~\eqref{eq:PDEslices_new} holds. From Lemma~\ref{lemma:disint_non_critical_current} we have the disintegration 
\begin{equation*}
    S \restrict \Crit(S)^c = \int_0^1 \frac{\xi}{|\nabla^S \tbf|} \wedge  S\vert_t \; \dd t.
\end{equation*}
In particular, if $S$ has~\eqref{eq:NC}, then $S = S \restrict \Crit(S)^c$ and therefore 
\begin{equation*}
    S = \int_0^1 \frac{\xi}{|\nabla^S \tbf|} \wedge  S\vert_t \; \dd t. 
\end{equation*}
It will be convenient to write this conclusion in the following form:  
\begin{equation}\label{eq:dynamic_coarea} 
\begin{split}
    \int_0^1 \left( \frac{ \xi}{|\nabla^\S \tbf|}\wedge \S|_{t} \right) \psi(t)\; \dd t  = (\psi \circ \tbf)\S\qquad\text{for every $\psi\in \Crm^\infty_c((0,1))$}.
\end{split}
\end{equation} 
We now want to prove that for every $\psi \in \Crm^\infty_c((0,1))$ and for every $\omega \in \mathscr D^k(\R^d)$ it holds that
\begin{equation*}
    \int_0^1 \langle S(t),\omega\rangle \psi'(t)-\langle \Lcal_{b_t}S(t),\omega\rangle \psi(t) \; \dd t = 0. 
\end{equation*}
Moreover,
\[
	- \int_0^1 \langle \Lcal_{b_t}S(t),\omega\rangle \psi(t)\; \dd t =  \int_0^1 \langle b_t \wedge S(t),d\omega\rangle \psi(t)\; \dd t.
\]
By~\eqref{eq:geom_deriv_est} and~\eqref{eq:coarea_integral} this integral is finite:
\begin{align*} 
\left|  \int_0^1 \langle b_t \wedge S(t),d\omega\rangle \psi(t)\; \dd t \right| & \le 
\|d\omega\|_\infty \|\psi\|_{\infty} \int_0^1 \Mbf(b_t \wedge S(t))\dd t \\ 
& = \|d\omega\|_\infty \|\psi\|_{\infty} \int_0^1 |b_t \wedge \vec{S}\vert_t| \; \dd \|S\vert_t \|\dd t \\ 
&\leq \|d\omega\|_\infty \|\psi\|_{\infty} \Mbf(S) \\
&< \infty. 
\end{align*} 
We then have
\begin{align*}
    \int_0^1 \langle b_t \wedge S(t),d\omega\rangle \psi(t)\; \dd t
     &= \left\langle\int_0^1  b_t \wedge S(t) \psi(t)\; \dd t,d\omega\right\rangle  \notag\\ 
     &= \left \langle\int_0^1  \frac{\pbf(\xi)}{|\nabla^S \tbf|} \wedge \pbf_{*}(S|_{t}) \psi(t)\; \dd t,d\omega\right\rangle  \notag\\
     &=  \left\langle\int_0^1  \pbf_{*}\biggl(\frac{\xi}{|\nabla^S \tbf|} \wedge S|_{t}\biggr) \psi(t)\; \dd t,d\omega\right\rangle  \notag\\
     &= \left \langle \pbf_{*}\left(\int_0^1  \frac{\xi}{|\nabla^S \tbf|} \wedge S|_{t} \psi(t)\; \dd t\right),d\omega\right\rangle  \notag\\       
     &= \langle \pbf_{*}((\psi \circ \tbf)S),d\omega\rangle  \notag\\
     &= \langle \partial\pbf_{*}((\psi \circ \tbf)S),\omega\rangle  \label{eq:lie_derivative1} 
\end{align*}
where in the second-to-last equality we used~\eqref{eq:dynamic_coarea}. Using also the commutativity between boundary and pushforward, we have thus shown 
\begin{equation}\label{eq:lie_derivative2}
   -\int_0^1 \Lcal_{b_t}(\S(t)) \psi(t) \, dt =  \partial \pbf_{*}[(\psi\circ \tbf) \S]=  \pbf_{*}\partial[(\psi\circ \tbf) \S]. 
\end{equation}
Observe now that for any $k$-current $\T$, for any $f \in \Crm_c^\infty(\R)$, and any $\omega \in \Dcal_k(\R^d)$, the Leibniz rule holds in the form 
\[
\langle \partial( f\T), \omega \rangle = \langle f\partial \T, \omega \rangle - \langle \T \restrict df ,  \omega \rangle.  
\]
Using this in~\eqref{eq:lie_derivative2} and also taking into account that $d(\psi\circ \tbf)=(\psi'\circ \tbf)dt$ as well as~\eqref{eq:slice_general},
\begin{align*}
 -\int_0^1 \langle \Lcal_{b_t}S(t),\omega\rangle \psi(t)\; \dd t 
 &=  \langle \partial [(\psi \circ \tbf)S], \pbf^{*}\omega \rangle \\
 &=  \langle (\psi \circ \tbf) \partial S, \pbf^{*}\omega \rangle - \langle S \restrict d(\psi\circ\tbf) , \pbf^{*}\omega \rangle \\ 
&=  \langle (\psi \circ \tbf) \partial S, \pbf^{*}\omega \rangle - \langle S \restrict (\psi'\circ\tbf) dt , \pbf^{*}\omega \rangle \\ 
&=   \langle (\psi \circ \tbf) \partial S, \pbf^{*}\omega \rangle - \int_0^1 \langle S|_t , \pbf^{*}\omega \rangle \psi'(t) \; \dd t. 
\end{align*}
Since $\psi \in \Crm^\infty_c((0,1))$ and $\partial S\restrict {(0,1)\times\R^d}=0$, the term $(\psi\circ \tbf)\partial S$ vanishes. Recalling that, by definition, $S(t)=\pbf_*(S|_t)$, we finally obtain
\[
\int_0^1 \langle \Lcal_{b_t}\S(t),\omega\rangle \psi(t)\; \dd t = \int_0^1 \langle S(t) ,\omega \rangle \psi'(t) \; \dd t,
\]
which proves that $S(t)$ is indeed a distributional solution to~\eqref{eq:PDEslices_new}.
\end{proof}

\subsection{Necessity of~\eqref{eq:NC}} 

In this section we prove that~\eqref{eq:NC} is a necessary condition for $S$ if its slices solve the geometric transport equation with a vector field  $b \in {\rm L}^1( \int_0^1  \| S\vert_t \| \; \dd t)$. This completes the proof of the Advection Theorem~\ref{thm:advection}.

\begin{proposition}\label{prop:converse_if_PDE_then_Sard_for_slices}
Let $S\in\Irm^{\AC}_{k+1}([0,1] \times \R^d)$ with $(\partial \S) \restrict {(0,1) \times \R^d}=0$. Suppose that there exists a vector field $b \in {\rm L}^1( \int_0^1  \| S\vert_t \| \; \dd t)$ with
\begin{equation}\label{eq:PDE_solved_by_slices}
     \frac{\dd}{\dd t} S(t) + \Lcal_{b_t} S(t) = 0,
\end{equation}
where $S(t) = \pbf_{*}(S\vert_t)$ are the projected slices of $S$. Then, $S$ satisfies~\eqref{eq:NC}. 
\end{proposition}

\begin{proof} Since $S$ is an $\AC$ integral current, the conclusion is equivalent to showing that $S$ has the Sard property, i.e., $\mu_t^s=0$ for $\L^1$-a.e.\ $t$ (see Proposition~\ref{prop:sard_equivalent_AC}). This will be achieved comparing $S$ with another normal current with the same slices, constructed via Proposition~\ref{prop:kampschulte}. 

{\emph{Step 1.}} 
Consider the map $t \mapsto S(t)$. By assumption it satisfies~\eqref{eq:PDE_solved_by_slices}, therefore we can resort to Proposition~\ref{prop:kampschulte} to ‘‘bundle’’ together the currents $S(t)$. We obtain a normal current $U \in \Nrm_{k+1}([0,1] \times \R^d)$ with a simple unit $(k+1)$-vector $\vec{U}$ and mass $\|U\|=|(1, b)|\L^1 \otimes \|S(t)\|$. In addition, we also know that $(\partial U) \restrict {(0,1) \times \R^d} = 0$, $\pbf_{*}(U\vert_t) = S(t)$ for $\L^1$-a.e.\ $t \in (0,1)$, and $\tbf_{\#}\|U\| \ll \L^1$.

In particular, the slices of the space-time current $U$ coincide with those of $S$ (because $\pbf_{*}(U\vert_t) = S(t) = \pbf_{*}(S\vert_t)$, and both $U\vert_t$ and $S\vert_t$ live in the fiber $\{t\}\times \R^d$). Applying Lemma~\ref{lemma:zero_slices} we conclude that 
\begin{equation}\label{eq:s-u_is_w}
    S - U = \int_0^1 W_t \; \dd \kappa(t) =: W
\end{equation}
where $\kappa:=\tbf_{\#}(\|S - U\|)=\tbf_{\#}\|W\|$ is a finite measure in $\R$, $W_t \in \Nrm_{k+1}([0,1] \times \R^d)$, $\partial W_t=0$, $W_t$ is supported in $\{t\}\times\R^d$ 
for $\kappa$-a.e.\ $t$.  
Observe that since both $\tbf_{\#}\|S\|$ and $\tbf_{\#}\|U\|$ are absolutely continuous with respect to $\L^1$, so is $\kappa = \tbf_{\#}\|W\|$, because for every set $A\subset [0,1] \times \R^d$ it holds that 
\[
\kappa(A) = \tbf_{\#}\|W\|(A) = \|W\|(\tbf^{-1}(A)) \le (\|S\|+\|U\|)(\tbf^{-1}(A)). 
\]
Hence, we write from now onwards $\kappa = h \L^1$, where $h$ denotes the density with respect to $\L^1$.  

{\emph{Step 2.}} We now prove that $S$ has the Sard property, i.e., $\mu_t^s=0$ for $\L^1$-a.e.\ $t$. By~\eqref{eq:s-u_is_w},
\begin{equation*}
    \|S \| \le \|W \| + \|U\| \le \int_0^1\| W_t \|h(t)\; \dd t + \int_0^1 |(1,b)|\|S|_t\|\; \dd t
\end{equation*}
whence by the Disintegration Structure Theorem~\ref{thm:struct}, see in particular~\eqref{eq:structure_disint}, and the fact that $S$ is $\AC$, we obtain  
\begin{equation*}
\int_0^1 \left( |\nabla^S\tbf|^{-1}  \| S\vert_t \| + \mu^s_t  \right) \; \dd t \le \int_0^1 \| W_t\|h(t) \; \dd t + \int_0^1 |(1,b)|\|S|_t\|\; \dd t
\end{equation*}
in the sense of measures on $[0,1] \times \R^d$. Computing this expression on the set $C:=\Crit(S)$ and rearranging terms, we deduce 
\begin{equation*}
\int_0^1 \mu^s_t(C) \; \dd t \le \int_0^1 (|(1,b)| - |\nabla^S\tbf|^{-1}) \| S\vert_t \|(C) + h(t)\| W_t\|(C)  \; \dd t.
\end{equation*}
By Lemma~\ref{lemma:disint_bella}~(ii),
\[
 \|S\vert_t \|(C) = 0 \qquad \text{for $\L^1$-a.e.\ $t$.}
\]
For the other term, since $W_t \in \Nrm_{k+1}([0,1] \times \R^d)$ we have $\| W_t \| \ll \H^{k+1}$ (see~\cite[4.1.20]{Federer69book})
and hence $\| W_t \|(C) = \| W_t \|(C \cap (\{t\} \times \R^d))=0$, because by Remark~\ref{remark:coarea_di_alberti_mu_t_perp} it holds that $\H^{k}(C \cap (\{t\} \times \R^d)) = 0 $ for $\L^1$-a.e.\ $t$.
We therefore obtain $\mu_t^s(C)=0$ for $\L^1$-a.e.\ $t$, and this yields $\mu_t^s = 0$ because $\mu_t^s$ is (by definition) concentrated on $C$ (see Theorem~\ref{thm:struct}). 
\end{proof}

\begin{remark} In the proof of Proposition~\ref{prop:converse_if_PDE_then_Sard_for_slices}, it can actually be proved that $W=0$, whence $S=U$. Indeed, we have already observed that $\kappa = \tbf_{\#}\|W\| \ll \L^1$. Denoting by $R$ the $(k+1)$-rectifiable carrier of $S$, observe that $W\restrict {R^c}=0$: Indeed, for every $A \subset R^c$ we have  
\[
W\restrict A = (S - U)\restrict A=- U\restrict A.
\]
but $\vec{W}(t,x) \in\Wedge_k(\{0\}\times \R^d)$, while $\vec{U}(t,x)$ is instead parallel to $(1,b(t,x))\wedge \vec{S}|_t(t,x)$.
Therefore, $W\restrict A=0$, hence $W = W\restrict {R}$. Let now $B \subset R$: We have (see, e.g.,~\cite[Corollary 2.29]{AFP})
\[
\|W\|(B) = \int_0^1 \|W_t\|(B)\; \dd \kappa(t) = \int_0^1 \|W_t\|(B)h(t)\; \dd t. 
\]
Since the $W_t$ are $(k+1)$-normal currents, we have 
\[
\{t: \|W_t\|(B) >0 \} \subset \{t: \H^{k+1}(B \cap( \{t\} \times \R^d)) >0 \} \]
and since $\H^{k+1}(B)<\infty$ the latter set must be at most countable. Therefore, $\|W\|(B)=0$ and, since $B$ is arbitrary, we conclude $W=0$. In particular, $S=U$, and this can be considered as a rectifiability result for the current $U$ obtained in Proposition~\ref{prop:kampschulte}.
\end{remark}

\section{Rademacher-type differentiability theorems}

As the culmination of our efforts, we now prove two Rademacher-type differentiability theorems for paths of currents. We start from a path $t\mapsto T_t\in \Irm_k(\R^d)$, $\partial T_t=0$, which is absolutely continuous in time with respect to the homogeneous integral flat norm $\Fhom_\Irm$
, and we ask when we can find a vector field $b_t:\R^d\to\R^d$ that solves the geometric transport equation. Finding such a vector field gives that the path $t\mapsto T_t$ is differentiable in the geometric sense, hence the moniker ``Rademacher''. Indeed, in this way we connect the derivative taken using outer variations and the linear structure of $\Irm_k(\R^d)$ (namely, $\frac{\dd}{\dd t}T_t$) with the derivative taken with respect to inner variations (namely, $\Lcal_{b_t}T_t$). Note that the Advection Theorem~\ref{thm:advection} also had a similar differentiability conclusion, but there we started with a space-time $(1+k)$-current and not merely a path of $k$-currents.

\subsection{Weak differentiability theorem}

First, we derive a weak version of differentiability: Given an $\Fhom_\Irm$-absolutely continuous path of integral $k$-currents, we can solve the transport equation, provided that in~\eqref{eq:PDE} we replace the Lie derivative $\Lcal_{b_t}T_t$ by $\partial R_t$ for some $(k+1)$-currents $R_t$ of finite mass (note that if $R_t = -b_t\wedge T_t$ for some vector field $b_t$, then $\partial R_t = \Lcal_{b_t}T_t$). This is basically a consequence of compactness for finite-mass currents.

\begin{theorem}[Weak differentiability] \label{thm:weak_rademacher}
Let $t \mapsto T_t \in \Irm_k(\R^d)$, $t \in [0,1]$, with $\partial T_t=0$ for every $t \in [0,1]$, be a path that is absolutely continuous with respect to the homogeneous integral flat norm $\Fhom$ (which is implied by $\Fhom_\Irm$-absolute continuity), that is,
    \begin{equation} \label{eq:Fhom_AC}
    \Fhom(T_s-T_t)\leq \int_s^t g(r)\; \dd r
    \end{equation}
    for some $g\in {\rm L}^1([0,1])$ and all $s < t$. Then, there exists a finite-mass $(k+1)$-current $R_t\in \Mrm_{k+1}(\R^d)$ that solves the equation
\begin{equation}\label{eq:weak_rademacher}
\frac{\dd}{\dd t}T_t-\partial R_t=0.
\end{equation}
\end{theorem}

\begin{proof}
    In an analogous fashion to Lemma~\ref{lemma:equiv_def_of_solutions}, proving that the PDE holds in the sense of distributions amounts to showing that, for every test $k$-form $\omega$, the function $t\mapsto \langle T_t,\omega\rangle$ has as distributional derivative the function $t\mapsto \langle \partial R_t,\omega\rangle$. We first identify $R_t$ for $\L^1$-a.e.\ $t$, then we prove that they solve the PDE.
    
    From the absolute continuity assumption~\eqref{eq:Fhom_AC} and by Lebesgue's theorem,
    \[
    g(t)=\lim_{h\to 0}\fint_t^{t+h}g(r)\; \dd r\qquad \text{for $\L^1$-a.e.\ $t$}.
    \]
    Let now $R_{t,h}\in\Nrm_{k+1}(\R^d)$ be such that $\partial R_{t,h}=T_{t+h}-T_t$ and $\Mbf(R_{t,h})=\Fhom(T_{t+h}-T_t)$, whose existence is guaranteed by the definition of $\Fhom$ and the usual compactness and lower semicontinuity results. It follows that
    \[
    \Mbf\left(\frac{R_{t,h}}{h}\right)=\frac{1}{h}\Fhom(T_{t+h}-T_t)\leq \fint_t^{t+h} g(r)\; \dd r
    \]
    which converges to $g(t)$ as $h\to 0$, for $\L^1$-a.e.\ $t$. Thus, $\tfrac{1}{h}R_{t,h}$ is equibounded in mass and by compactness there is an infinitesimal sequence $h_j\to 0$ such that $R_{t,h_j}\weaksto R_t$ for some $R_t\in \Nrm_{k+1}(\R^d)$. We shall now prove that these $R_t$ satisfy~\eqref{eq:weak_rademacher}.
    
    We first argue that the scalar map $t\mapsto \langle T_t,\omega\rangle$ is absolutely continuous for every test $k$-form $\omega \in \mathscr D^k(\R^d)$. Indeed,
    \[
    |\langle T_t,\omega\rangle-\langle T_s,\omega\rangle|=|\langle T_t-T_s,\omega \rangle|\leq \Fhom(T_t-T_s) \|d\omega\|_\infty\leq \|d\omega\|_\infty\int_s^t g(r)\; \dd r.
    \]
    This proves the absolute continuity, As a consequence, in order to identify the distributional derivative of $t\mapsto \langle T_t,\omega\rangle$, it is sufficient prove that
    \begin{equation}\label{eq:pointwise_derivative}
    t\mapsto \langle T_t,\omega\rangle\qquad\text{has as pointwise $\L^1$-a.e.\ derivative}\qquad t\mapsto \langle \partial R_t,\omega\rangle.
    \end{equation}
    
    Let us fix a countable set $\Dcal$ of test $k$-forms that is dense in $\mathscr D^k(\R^d)$ the $\Crm^1$-norm, and prove~\eqref{eq:pointwise_derivative} for every $\eta\in\Dcal$. From the absolute continuity, there is a negligible set $N \subset \R$ such that outside $N$
    \[
    t\mapsto \langle T_t,\eta\rangle\qquad\text{is differentiable for every $\eta\in\Dcal$.}
    \]
    This implies that outside $N$,
    \begin{equation}\label{eq:pointwise_explicit}
    \langle \partial R_t,\eta\rangle=\lim_{h\to 0}\left\langle\frac{\partial R_{t,h}}{h},\eta\right\rangle=\lim_{h\to 0} \left\langle \frac{T_{t+h}-T_t}{h},\eta\right\rangle 
    \end{equation}
    where the existence of the limit is guaranteed by the differentiability of $t\mapsto \langle T_t,\eta\rangle$. This proves~\eqref{eq:pointwise_derivative} for every $\eta\in\Dcal$.
    
    Let us consider now any test $k$-form $\omega \in \mathscr D^k(\R^d)$, and any $\eps>0$. By density we can find $\eta\in\Dcal$ such that $\|\omega-\eta\|_{\Crm^1}<\eps$. Then,
    \[
    \left \langle \frac{T_{t+h}-T_t}{h},\omega\right\rangle=\left\langle \frac{T_{t+h}-T_t}{h},\eta\right\rangle+\left\langle \frac{T_{t+h}-T_t}{h},\omega-\eta\right\rangle
    \]
    and
    \[
    \langle \partial R_t,\omega\rangle=\langle \partial R_t,\eta\rangle+\langle \partial R_t, \omega-\eta\rangle.
    \]
    We can bound both error terms as follows:
    \[
    \left\vert \left\langle \frac{T_{t+h}-T_t}{h},\omega-\eta\right\rangle\right\vert\leq \|\omega-\eta\|_{\Crm^1}\fint_t^{t+h}g(r)\; \dd r
    \]
    and
    \[
    |\langle \partial R_t, \omega-\eta\rangle|\leq \Mbf(R_t) \|\omega-\eta\|_{\Crm^1}\leq \|\omega-\eta\|_{\Crm^1}\fint_t^{t+h}g(r)\; \dd r.
    \]
    Using that~\eqref{eq:pointwise_explicit} holds for $\eta$ we deduce that
    \[
    \limsup_{h\to 0}\left\vert \left\langle \frac{T_{t+h}-T_t}{h},\omega\right\rangle-\langle \partial R_t,\omega\rangle\right\vert \leq 2\eps g(t).
    \]
    Since $\eps $ is arbitrary we obtain that~\eqref{eq:pointwise_derivative} holds for $\omega$ as well, and this concludes the proof.
\end{proof}

\subsection{Strong differentiability theorem}

We next ask when the currents $R_t$ are actually of the form $-b_t\wedge T_t$ for some vector field $b_t$, so that we can solve the geometric transport equation in the original formulation~\eqref{eq:PDE}. This question is more subtle, and after the previous sections it is perhaps not surprising that a positive answer is strictly related to the property~\eqref{eq:NC} or, equivalently, the Sard property~\eqref{eq:Sard} (for the space-time current obtained through the procedure described in Section~\ref{sc:rect}). 

\begin{theorem}[Strong differentiability] \label{thm:strong_rademacher}
Let $t \mapsto T_t \in \Irm_k(\R^d)$, $t \in [0,1]$, be a path that is absolutely continuous with respect to the homogeneous integral flat norm $\Fhom_\Irm$ and such that
\[\partial T_t = 0, \; t\in [0,1] \qquad \text{and}\qquad \sup_{t\in[0,1]} \Mbf(T_t)<\infty.
\]
Let $S$ be the unique current given by Theorem~\ref{thm:rect} in this setting (uniqueness follows from Corollary~\ref{cor:S=S'}). 
If $S$ satisfies~\eqref{eq:NC}, then there exists $b \in {\rm L}^1(\L^1 \otimes \|T_t\|)$ such that~\eqref{eq:PDE} holds, that is,
\[
\frac{\dd}{\dd t} T_t + \Lcal_{b_t} T_t = 0. 
\]
\end{theorem}

\begin{proof} Let $S\in \Irm^{\AC}_{k+1}([0,1] \times \R^d)$ be the current obtained by Rectifiability Theorem~\ref{thm:rect} applied to $t \mapsto T_t$. Since by assumptions $S$ satisfies~\eqref{eq:NC} we can apply the Advection Theorem~\ref{thm:advection} (we only need the direction of Proposition~\ref{prop:slices_are_advected}), obtaining that there exists $b \in {\rm L}^1(\L^1 \otimes \|S(t)\|)$ such that $t\mapsto S(t)=T_t$ solves the geometric transport equation associated with $b$.
\end{proof}

Recall that in the case $k=0$ (that is, essentially, for paths of atomic measures), the Sard property is automatically satisfied by the area formula, see Remark~\ref{rmk:case_k_zero}. Thus, for $k=0$ strong differentiability holds without any further assumption. See also \cite[Theorem~1.13]{DS_revisited} and \cite[Lemma~1.1]{DS_multiple}.

\subsection{Particular dimensions}

Finally, we will look at the special cases $k\in\{0,d-2,d-1,d\}$, where it is known that $\Fhom_{\Irm}$ coincides with $\Fhom$, and thus all the results can be stated using the latter norm.

\begin{proposition}\label{prop:equality}
Let $T\in\mathrm I_k(\R^d)$ with $\partial T=0$. If $k\in \{0,d-2,d-1,d\}$ then 
\[
	\Fhom(T)=	\Fhom_{\Irm}(T).
\]			
\end{proposition}

The case $k=d$ is trivial, since the only boundaryless normal $d$-current in $\R^d$ is the zero current. 
In the case $k=d-1$ the proof of Proposition~\ref{prop:equality} is an easy consequence of the isoperimetric inequality and the constancy lemma. Indeed, one can show that in this case there exists an integral current $S$ such that $\partial S = T$ (by the isoperimetric inequality, see~\cite[Theorem 7.9.1]{KP}). Such a current $S$ is also unique (in the class of normal currents with boundary $T$): suppose $N_1, N_2 \in \mathrm{N}_{d}(\R^d)$ are such that $\partial N_1 = \partial N_2 = T$, Then, the normal $d-$current $N := N_1 - N_2$ would be boundaryless, hence by the constancy lemma (see, e.g.,~\cite[Proposition 7.3.1]{KP}), $N$ would be equivalent to a constant. Since both $N_1,N_2$ have finite mass, we necessarily have $N_1 - N_2 = 0$ and so uniqueness holds.

For a proof of Proposition~\ref{prop:equality} for the case $k=d-2$ we refer the reader to, e.g.,~\cite{HP} or~\cite[Remark 5]{Brezis_Mironescu}, which also contains the proof of the case $k=0$.

Combining Proposition~\ref{prop:equality} with Theorem~\ref{thm:strong_rademacher} we obtain the following:

\begin{corollary}
Let $k \in \{0,d-2,d-1,d\}$ and let $t \mapsto T_t \in \Irm_k(\R^d)$, $t \in [0,1]$, be a path that is absolutely continuous with respect to the homogeneous (non-integral) flat norm $\Fhom$ and such that 
\[\partial T_t = 0, \; t\in [0,1] \qquad\text{and}\qquad \sup_{t\in[0,1]} \Mbf(T_t)<\infty.
\]
Let $S$ be the unique current given by Theorem~\ref{thm:rect} in this setting. If $S$ satisfies~\eqref{eq:NC}, then there exists $b \in {\rm L}^1(\L^1 \otimes \|T_t\|)$ such that~\eqref{eq:PDE} holds.
\end{corollary}

For general $k$ it does not seem to be known whether, for $T\in \Irm_k(\R^d)$, the definitions of homogeneous flat norm and homogeneous integral flat norm give rise to equivalent norms.

\section{A Lipschitz space-time current with many critical points}\label{sc:counterexample} 

In this final section we construct an example, which we call the ``Flat Mountain'', of a Lipschitz-in-time (so, in particular, AC) integral current $S\in \Irm_2^{\Lip} (\R\times\R^2)$ that does not possess the negligible criticality property~\eqref{eq:NC} or, equivalently, the Sard property~\eqref{eq:Sard}. This implies (via Proposition~\ref{prop:converse_if_PDE_then_Sard_for_slices}) that there is \emph{no} vector field $b\in {\rm L}^1(\L^1(\dd t)\otimes \|S|_t\|)$ such that $S$ that solves the Lie derivative equation~\eqref{eq:PDE} with $b$. In this sense, our Strong Differentiability Theorem~\ref{thm:strong_rademacher} is sharp.

The current $S$ will be the graph of a $\BV$ function $u:[0,1]^2\to \R$. We present a few different ways to construct $u$. The inspiration for this construction comes from~\cite[Section~4]{ABC}, and indeed, the function $u$ can be seen as a suitable limit of the construction performed there. See Figure~\ref{fig:Sard} for an illustration of our construction.

\subsection{Fixed point} We start by defining the cubes
\[
Q_0:=[0,\tfrac12)\times [0,\tfrac12),\quad Q_1:=[\tfrac12,1)\times [0,\tfrac12),\quad Q_2:=[0,\tfrac12)\times [\tfrac12,1),\quad Q_3:=[\tfrac12,1)\times [\tfrac12,1)
\]
and the maps $A_i:Q_i\to [0,1)^2$ given by 
\[
A_0(x):=2x,\qquad A_1(x)=2x-(1,0),\qquad A_2(x):=2x-(0,1),\qquad A_3(x):=2x-(1,1)
\]
together with their inverses $B_i:[0,1)^2\to Q_i$,
\[
B_0(x)=\tfrac12 x,\qquad B_1(x)=\tfrac12 x +(\tfrac12,0),\qquad B_2(x)=\tfrac12 x+(0,\tfrac12),\qquad B_3(x)=\tfrac12 x +(\tfrac12,\tfrac12).
\]
Consider the function 
\[
\bar{u}(x):=\sum_{i=0}^3 \frac{i}{4}\chi_{Q_i}(x).
\]
We define the map $L:{\rm L}^1([0,1]^2)\to {\rm L}^1([0,1]^2)$ given by
\begin{equation}\label{eq:def_L}
(Lv)(x):=\bar{u}(x)+\frac14 \sum_{i=0}^3 v(A_i(x))\chi_{Q_i}(x).
\end{equation}

\begin{figure}\centering
\begin{subfigure}[b]{0.45\textwidth}
         \centering
         \begin{tikzpicture}[scale=4]
       %%% BIG SQUARE %%%
       \draw (0,0)--(1,0)--(1,1)--(0,1)--(0,0); 
       
      \draw (0.5,0)--(0.5,1);
      \draw (0,0.5)--(1,0.5);
      
       %%% NODES %%%
       \node at (0.25,0.25) {$Q_0$};
       \node at (0.75,0.25) {$Q_1$};
       \node at (0.25,0.75) {$Q_2$};
       \node at (0.75,0.75) {$Q_3$};
       
    \end{tikzpicture}
    
    \caption{The subsquares $Q_0,Q_1,Q_2,Q_3$ used in the construction of $u$.}
    \label{fig:Q0123}
     \end{subfigure}
     \hfill
     \begin{subfigure}[b]{0.45\textwidth}
         \centering
         \begin{tikzpicture}[scale=4]
         %%% BIG SQUARE %%%
       \draw (0,0)--(1,0)--(1,1)--(0,1)--(0,0); 
       
          %%% VERTICAL LINES %%%
       \draw[dotted] (0.5,0)--(0.5,1);
       \draw[dotted] (0.25,0)--(0.25,1);
       \draw[dotted] (0.75,0)--(0.75,1);
       
       %%% HORIZONTAL LINES %%%
       \draw[dotted] (0,0.5)--(1,0.5);
       \draw[dotted] (0,0.75)--(1,0.75);
       \draw[dotted] (0,0.25)--(1,0.25);
       
       %%% SUBSQUARE Q_{3,1} %%%
       \draw[thick] (0.75,0.5)--(1,0.5)--(1,0.75)--(0.75,0.75)--(0.75,0.5);
       \draw[->,dashed] (1.1,0.9)--(1.02,0.78);
       \node[right] at (1.1,0.9) {$Q_{3,1}$};
       
       %%% SUBSQUARE Q_{2,1,1} %%%
       \draw[thick] (0.375,0.5)--(0.5,0.5)--(0.5,0.625)--(0.375,0.625)--(0.375,0.5);
       \draw[->,dashed] (-0.1,0.75)--(0.35,0.6);
       \node[left] at (-0.1,0.75) {$Q_{2,1,1}$};
       
       %%% POINT x_{1,3,0,0,...} %%%
       \draw[fill] (0.75,0.25) circle [radius=0.01];
       \draw[->,dashed] (1.1,0.5)--(0.8,0.28);
       \node[right] at (1.1,0.5) {$x_{1,3,0,0,\ldots}$};
       
         \end{tikzpicture}
         \caption{A few subsquares and points identified by a sequence of numbers in $\{0,1,2,3\}$.}
     \end{subfigure} \\ 
    \caption{}
\end{figure}

\begin{lemma}\label{lemma:contraction} We have the following:
\begin{enumerate}
    \item[{\rm(i)}] The map $L$ is a contraction in ${\rm L}^p$ for every $p\in[1,\infty]$, and also in $\Crm^0$ (actually, it is a scaled isometry with factor $\frac14$).
    \item[{\rm(ii)}] The map $L$ is a contraction in $\BV$.
    \item[{\rm(iii)}] Let $u$ be the unique fixed point of $L$ in $\BV$. Then, $Du=D^ju$. In particular, $\nabla u(x)=0$ for a.e.\ $x\in [0,1]^2$.
\end{enumerate}
\end{lemma}

\begin{proof}
(i) For any $p\in[1,\infty)$ we have
\[
\|Lv-Lw\|_{{\rm L}^p}=\left(\frac{1}{4^p}\sum_{i=0}^3 \frac14\|v-w\|_{{\rm L}^p}^p\right)^{1/p}=\frac{1}{4} \|v-w\|_{{\rm L}^p}.
\]
The case $p=\infty$ and $\Crm^0$ is very similar.

(ii) We have that
\begin{align*}
|D(Lv-Lw)|([0,1]^2)& \leq \frac14 \sum_{i=0}^3 |D(v\circ A_i-w\circ A_i)|(\overline{Q_i})\\
&= \frac14 \sum_{i=0}^3 \frac12 |D(v-w)|([0,1]^2)\\
&=\frac12 |D(v-w)|([0,1]^2).
\end{align*}
Together with point (i) for $p=1$, this proves that $L$ is a contraction in $\BV$.

(iii)
We use the fact that the fixed point can be identified by iteration: We choose as a starting point the function $u_0=0$, and define $u_{n+1}:=Lu_n$. Then, set $u:=\lim_{n\to\infty}u_n$.
Claim~(ii) ensures that the convergence happens in the $\BV$ norm, and in particular, $Du_n$ converges strongly to $Du$ in the sense of measures. Since $Du_n$ is supported on the set $D_n\times D_n$, where $D_n=\{\frac{j}{2^{n+1}}:j=0,\ldots, 2^{n+1}\}$, the measure $Du$ is supported on 
\[
J:=\bigcup_{n=1}^\infty D_n\times D_n.
\]
Since this set is $\H^1$-$\sigma$-finite, by the dimensional properties of the total variation, see~\cite[Proposition~3.92]{AFP}, we must have $Du=D^ju$, and therefore $\nabla u(x)=0$ for $\L^2$-a.e.\ $x\in [0,1]^2$.
\end{proof}

\begin{remark} An alternative way to prove point (iii) is by using the compactness theorem in $\SBV$. The jump set of $u_n$ is contained in $D_n\times D_n$, where $D_n=\{\frac{j}{2^{n+1}}:j=0,\ldots, 2^{n+1}\}$. Let $\ell_n=2^{n+1}$ denote the total length of $D_n\setminus D_{n-1}$. The jump of $u_n$ is of order $d_n\approx \frac{1}{4^n}$ on $D_n\setminus D_{n-1}$. As a consequence, setting $\theta(t):=t^\alpha$ for some $\alpha>\frac12$, we have
\[
\int_{J_{u_n}}\theta([u_n](x))\; \dd \H^1(x)\leq \sum_{j=0}^n \ell_j \theta(d_j)\lesssim\sum_{j=0}^n 2^{j(1-2\alpha)}\leq \frac{1}{1-2^{1-2\alpha}}.
\]
We can thus apply the closure theorem in $\SBV$~\cite[Theorem~4.7]{AFP} to conclude that $\nabla u=0$ almost everywhere and $Du=D^j u$.
\end{remark}

\begin{figure}
   \centering
     \begin{subfigure}[b]{0.45\textwidth}
	        	\centering
	         	\includegraphics[width=\textwidth]{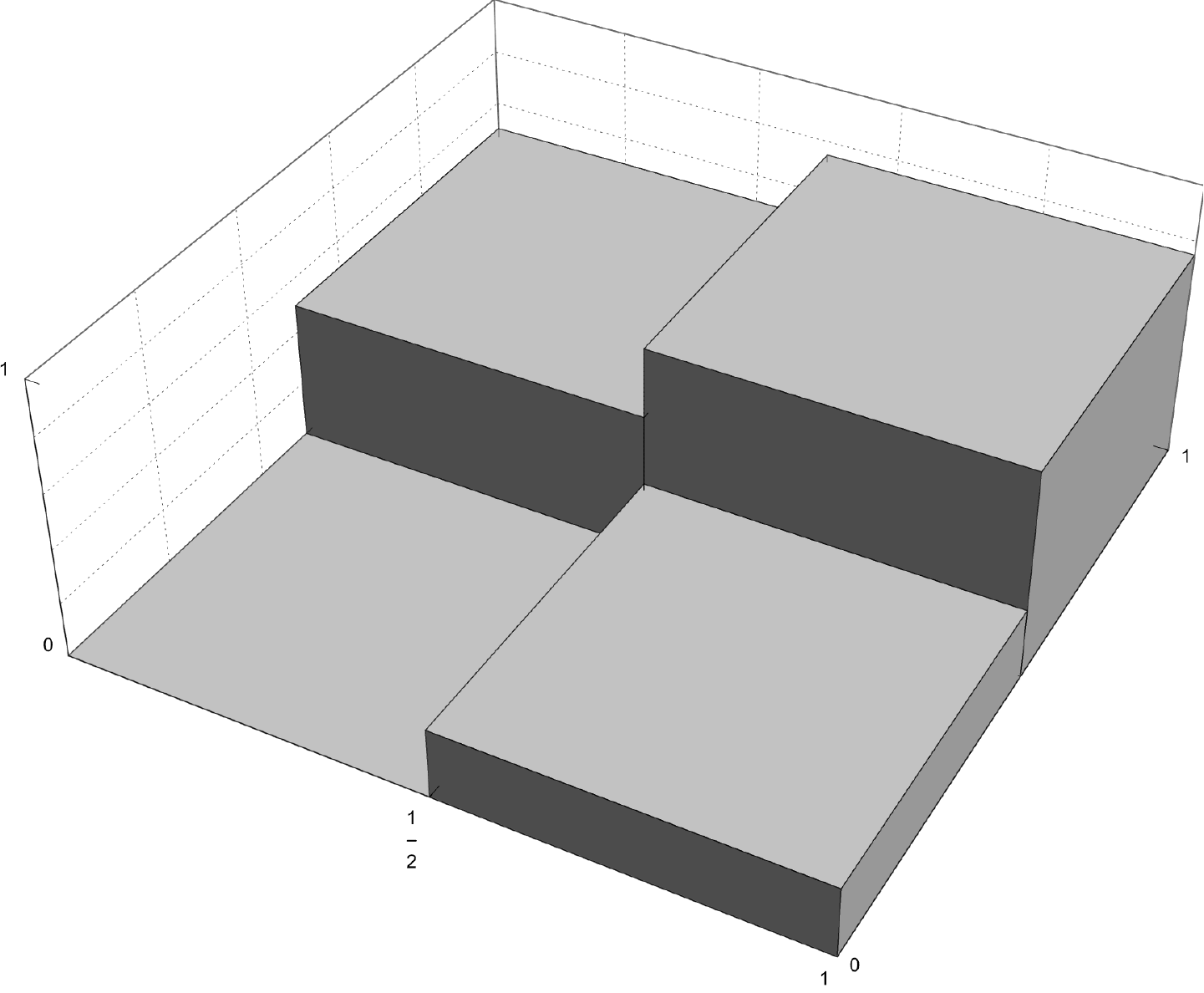}
	         	 \caption{The function $\bar u$, which also coincides with $u_1=Lu_0=L0$.}
	     \end{subfigure}
     \hfill
     \begin{subfigure}[b]{0.45\textwidth}
         \centering
	      \includegraphics[width=\textwidth]{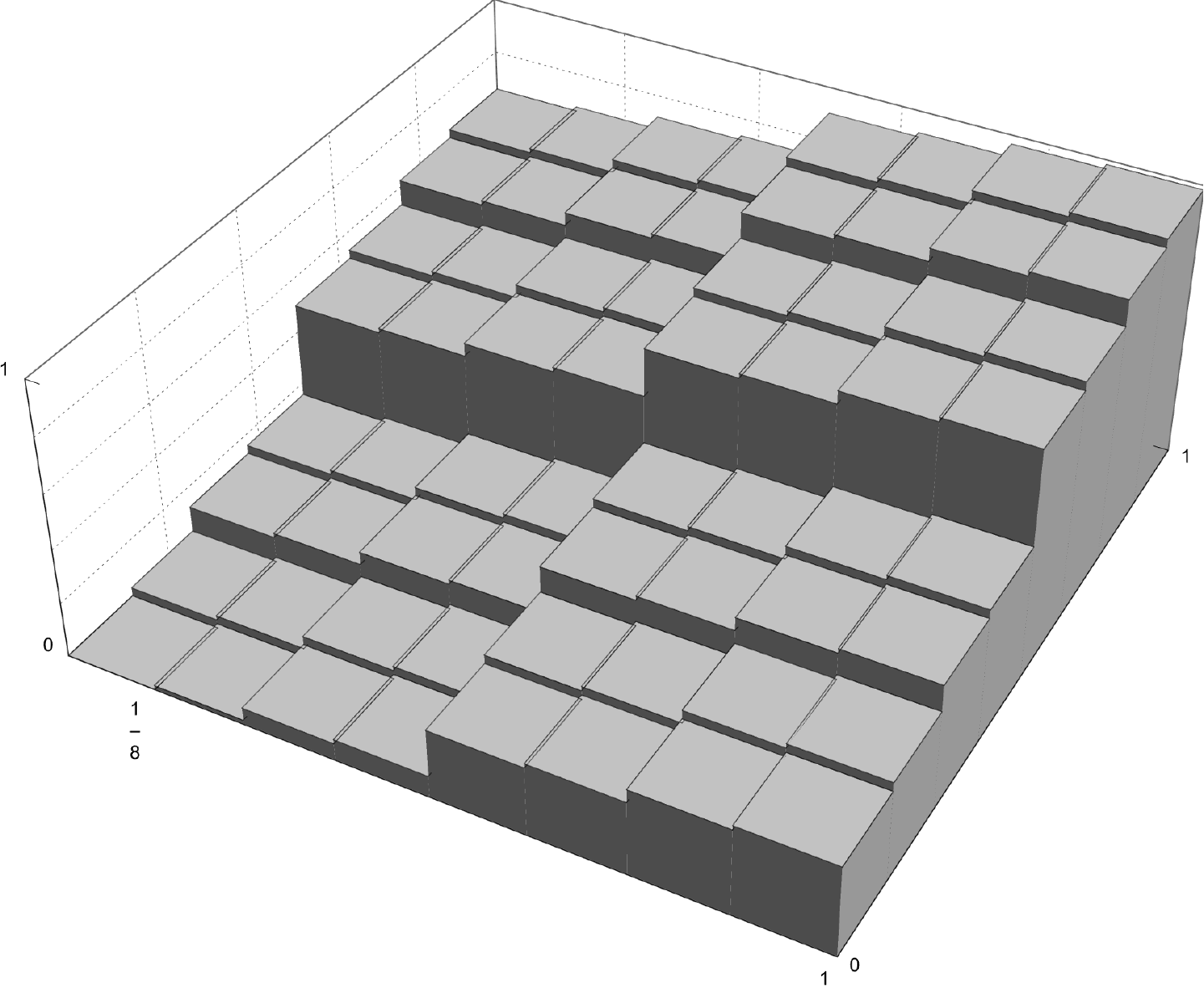}
	         \caption{The function $u_3$.}
	     \end{subfigure} \\ 
     \caption{}
	 \label{fig:Sard}
\end{figure}

\subsection{Base-4 expansion}
We now consider the family of dyadic cubes inside $[0,1)^2$. For simplicity we use half-open cubes, so that they are disjoint. It will be useful to describe these cubes by an expansion in base-$4$.

Given any $n\in\Nbb$, to any finite sequence $\ibf=(i_j)_{j=1}^n\in \{0,1,2,3\}^n$ we associate the dyadic cube $Q_{\ibf}$ of sidelength $2^{-n}$ given by 
\begin{equation}\label{eq:def_Q_i}
Q_\ibf=Q_{i_1,\ldots,i_n}:= B_{i_1}\circ \ldots\circ B_{i_n} ([0,1)^2).
\end{equation}
One can think of $Q_\ibf$ in the following way: $i_1$ tells us which subcube of $[0,1)^2$ we have to select, $i_2$ which sub-subcube of the previous one, and so on. At each step, index $0$ means the lower-left cube, index $1$ the lower-right cube, index $2$ the upper-left cube, and index $3$ the upper-right cube (see Figure~\ref{fig:Q0123}). 

In the same way, to any infinite sequence $\ibf\in\{0,1,2,3\}^\Nbb$ we associate the unique point $x_\ibf\in[0,1]^2$ identified by the corresponding infinite sequence of cubes:
\[
\{x_\ibf\}=\bigcap_{n\in \Nbb} \overline{Q}_{i_1,\ldots,i_n}.
\]
Vice versa, to any $x\in[0,1)^2$ we can associate the corresponding infinite family of cubes whose intersection is $\{x\}$, and therefore the  associated index $\ibf(x)$.

Recall the functions $u_n:=L(u_{n-1})$ obtained by iteration starting from $u_0=0$.

	\begin{lemma}\label{lemma:explicit_u}
		The following hold:
		\begin{enumerate}
			\item[{\rm(i)}] The function $u_n$ is constant on each dyadic cube $Q_\ibf$ of sidelength $2^{-n}$, and, more precisely,
			\begin{equation}\label{eq:u_n_base4}
				u_n(x)=\sum_{j=1}^n \frac{i_j}{4^j}\qquad\text{for every $x\in Q_\ibf$, $\ibf=(i_1,\ldots,i_n)$.} 
			\end{equation}
			\item[{\rm(ii)}] The function $u$ satisfies
			\begin{equation}\label{eq:u_base4}
				u(x)=\sum_{j=1}^\infty \frac{i_j}{4^j}\qquad\text{if $\ibf=\ibf(x)$.}
			\end{equation}
		\end{enumerate}
	\end{lemma}

\begin{proof}
(i) Let us prove this by induction. The claim is clear if $n=1$. Let us suppose that~\eqref{eq:u_n_base4} holds for $n$, and let us prove it for $n+1$. If $x\in Q_{i_1,\ldots,i_{n+1}}$, then in particular $x\in Q_{i_1}$, and moreover from the definition~\eqref{eq:def_Q_i} it holds that $A_i(x)\in Q_{i_2,\ldots,i_{n+1}}$. Therefore, by~\eqref{eq:def_L}, if $x\in Q_{i_1,\ldots,i_{n+1}}$, then
\begin{align*}
u_{n+1}(x)&=\bar u(x)+\frac14\sum_{i=0}^3 u_n(A_i(x))\chi_{Q_i}(x)\\
& =\frac{i_1}{4}+\frac14 u_n(A_{i_1}(x))\\
&= \frac{i_1}{4}+\frac14\sum_{j=1}^n\frac{i_{j+1}}{4^j}\\
&= \sum_{j=1}^{n+1} \frac{i_j}{4^j},
\end{align*}
and the claim is proven.

(ii)
Since $L$ is a contraction in $\Crm^0$, the convergence $u_n\to u$ is uniform. Then, passing to the limit as $n\to\infty$ in~\eqref{eq:u_n_base4}, we deduce~\eqref{eq:u_base4}.
\end{proof}

\subsection{A space-filling curve}\label{subsec:curve}
We also record here the definition of the following curve $\gamma:[0,1]\to [0,1]^2$, which is essentially surjective. This curve essentially coincides with the inverse of $u$, modulo the fact that $u$ is not exactly injective and thus not invertible (but it is so on a subset of $[0,1]^2$ of full measure). Given $t\in [0,1]$, we consider its base-$4$ expansion $t=\sum_{j=1}^\infty t_j 4^{-j}$, $t_j\in\{0,1,2,3\}$. We suppose that the expansion does not end with an infinite sequence of $3$'s, so that it is uniquely determined. Then we consider the sequence $\ibf(t)=(t_j)_{j=1}^\infty$ and we set
\[
\gamma(t):=x_{\ibf(t)}.
\]
More concretely, setting
\[
t^1_j:=\begin{cases}
0&\text{if $t_j\in\{0,2\}$,}\\
1&\text{if $t_j\in\{1,3\}$,}
\end{cases}\qquad
t^2_j:=\begin{cases}
0&\text{if $t_j\in\{0,1\}$,}\\
1&\text{if $t_j\in\{2,3\}$,}
\end{cases}
\]
then $\gamma(t)=(\gamma_1(t),\gamma_2(t))$ is given by
\[
\gamma_1(t)=\sum_{j=1}^\infty \frac{t^1_j}{2^j},\qquad\gamma_2(t)=\sum_{j=1}^\infty \frac{t_j^2}{2^j}.
\]
We record here that this curve (and slight variations) has been known under the name of \textit{Lebesgue curve} or \textit{Z-order curve}.

\subsection{The Flat Mountain $S$ and its slices}
We are finally ready to consider the integral $2$-current $S$ associated to the graph of $u$, completed with jump parts. We define the following $2$-rectifiable sets:
\[
\Gamma_u:=\bigcup_{x\in[0,1]^2\setminus J_u} \{(x,u(x))\},\qquad \Gamma_{J_u}:= \bigcup_{x\in J_u} \{x\}\times [u^-(x),u^+(x)],
\]
where $[u^-(x),u^+(x)]$ stands for the segment between the traces of $u$ on the two sides of $J_u$. We emphasise that now we consider $u$ as a function on $\R^2$ supported on $[0,1]^2$, so that a big portion of $J_u$ lies on the boundary of $[0,1]^2$. Then we define $S$ to be the integral $2$-current with density $1$ associated to $\Gamma=\Gamma_u\cup \Gamma_{J_u}$. The layer cake decomposition and coarea formula for $\BV$ functions imply that
\[
u(x) e_1\wedge e_2=\int_0^1 \1_{\{u \geq t\}}(x) \; \dd t\qquad 
\text{and}
\qquad Du^\perp =\int_0^1 \partial (\1_{\{u \geq t\}}e_1\wedge e_2) \; \dd t 
\]
if we identify $Du^\perp$ with the $1$-current associated with the clockwise rotation $Du^\perp$ of the gradient measure.

	\begin{lemma}\label{lemma:properties_u_n} The following properties hold:
		\begin{enumerate}
			\item[{\rm(i)}] The superlevel sets $\{x:u_n(x)\geq t\}$ are piecewise constant in $t\in[0,1]$, with a discontinuity exactly when $t$ is a multiple of $\tfrac{1}{4^n}$.
			\item[{\rm(ii)}] For every $m\geq n$ and $j=0,\ldots,4^n$:
			\[
			\left\{u_n\geq \frac{j}{4^n}\right\}=\left\{u_m\geq \frac{j}{4^n}\right\}=\left\{u\geq \frac{j}{4^n}\right\}.
			\]
			\item[{\rm(iii)}] The following equality holds modulo $\L^2$-negligible sets:
			\[
			\left\{\sum_{j=1}^n \frac{i_j}{4^j}<u(x)<\frac{1}{4^n}+\sum_{j=1}^n \frac{i_j}{4^j}\right\}= Q_{i_1,\ldots,i_n}.
			\]
			\item[{\rm(iv)}] The slices $S(t):=\partial(\1_{\{u\geq t\}} e_1\wedge e_2)$ (seen as integral $1$-currents) satisfy
			\[
			\Fhom(S(t)-S(s))=|t-s|\qquad \text{for every $t,s\in[0,1]$ multiples of $\tfrac{1}{4^n}$.}
			\]
			\item[{\rm(v)}] The slices $S(t):=\partial(\1_{\{u\geq t\}} e_1\wedge e_2)$ (seen as integral $1$-currents) satisfy
			\[
			\Fhom(S(t)-S(s))=|t-s|\qquad\text{for every $t,s\in[0,1]$.}
			\]
		\end{enumerate}
	\end{lemma}
	
	\begin{proof}
		(i) This is clear from~\eqref{eq:u_n_base4}.
		
		(ii) Given any cube $Q_\ibf$, its subcubes are exactly those of the form $Q_{\ibf'}$, where $\ibf'$ is any sequence that agrees with $\ibf$ up to the $n$'th digit. Then again by~\eqref{eq:u_n_base4}, it follows that for every $m\geq n$ we have
		\[
		u_n(x)\leq u_m(x)< u_n(x)+\frac{1}{4^n}\qquad\text{for every $x\in Q_\ibf$, $\ibf=(i_1,\ldots, i_n)$.}
		\]
		From this the first equality follows. The second equality is derived from the uniform convergence of $u_n$ to $u$, which is a consequence of~\ref{lemma:contraction}~(i).
		
		(iii) If $x\in Q_{i_1,\ldots,i_n}$, then by Lemma~\ref{lemma:explicit_u}~(ii)
		\[
		\sum_{j=1}^n \frac{i_j}{4^j}\leq u(x)=\sum_{j=1}^n \frac{i_j}{4^j}+\sum_{j=n+1}^\infty \frac{i_j}{4^j}< \frac{1}{4^n}+ \sum_{j=1}^n \frac{i_j}{4^j},
		\]
		therefore one inclusion is proven up to removing the point whose index is $(i_1,\ldots,i_n,0,0,\ldots)$ (which realises the left-most equality).
		
		For the reverse inclusion, let $x\in Q_{\ell_1,\ldots,\ell_n}$, with $(\ell_1,\ldots,\ell_n)\neq (i_1,\ldots,i_n)$. Then either
		\[
		\sum_{j=1}^n\frac{\ell_j}{4^j}\leq \sum_{j=1}^n\frac{i_j}{4^j} -\frac{1}{4^n}
		\]
		or
		\[
		\sum_{j=1}^n\frac{\ell_j}{4^j}\geq \sum_{j=1}^n\frac{i_j}{4^j} +\frac{1}{4^n}.
		\]
		By Lemma~\ref{lemma:explicit_u}~(ii), in the first case we have 
		\[
		u(x)< \sum_{j=1}^n\frac{i_j}{4^j},
		\]
		while in the second case we have
		\[
		u(x)\geq \sum_{j=1}^n\frac{i_j}{4^j}+\frac{1}{4^n}.
		\]
		
		(iv) By the previous point, up to $\L^2$-negligible sets, it holds
		\[
		\left\{\frac{k}{4^n} < u(x)< \frac{h}{4^n}\right\}=\bigcup_{m=k}^{h-1} \left\{\frac{m}{4^n}< u(x) < \frac{m}{4^n}+\frac{1}{4^n}\right\}=\bigcup_{(i_1,\ldots,i_n)\in \Ical} Q_{i_1,\ldots,i_n},
		\]
		where $\Ical$ is a suitable set (of cardinality $h-k$) of indices, and the corresponding squares are essentially disjoint. Therefore,
		\[
		\Fhom(S(\tfrac{k}{4^n})-S(\tfrac{h}{4^n}))=\sum_{(i_1,\ldots,i_n)\in\Ical} \L^2(Q_{i_1,\ldots,i_n})=\frac{h}{4^n}-\frac{k}{4^n}.
		\]
		
		(v) From Points (ii) and (iv) we obtain that
		\[
		\Fhom(S(t)-S(s))=|t-s|
		\]
		for every $t,s$ that are multiples of $\tfrac{1}{4^n}$ for some $n$. Since the latter is a dense set, and the map $t\mapsto S(t)$ is easily seen to be continuous, we reach the conclusion for every $t,s \in[0,1]$.
	\end{proof}

\subsection{Summary of the example}
The Flat Mountain $S$ has the following properties:
\begin{enumerate}
    \item[(i)] It is an integral $2$-current in $\R^3$.
    \item[(ii)] Its projected slices $S(t)\in \Irm_1(\R^2)$ are boundaryless and $\Fhom(S(t)-S(s))=|t-s|$.
    \item[(iii)] $\|S\|(\Crit(S))=1$ because $\nabla u(x)=0$ for a.e.\ $x\in[0,1]^2$.
\end{enumerate}
Therefore, $S$ is a Lipschitz integral space-time current without the Sard property or equivalently, by Proposition~\ref{prop:sard_equivalent_AC}, without~\eqref{eq:NC}.

As a consequence of Proposition~\ref{prop:converse_if_PDE_then_Sard_for_slices}, there is no vector field $b\in {\rm L}^1(\L^1(\dd t)\otimes \|S|_t\|)$ such that $S$ that solves the Lie derivative equation~\eqref{eq:PDE} with $b$. Nevertheless, the Weak Differentiability Theorem~\ref{thm:weak_rademacher} ensures the existence of currents $R_t$ of finite mass that solve
\[
\frac{\dd}{\dd t}S(t)-\partial R_t=0.
\]
We finish by identifying these $R_t$ in our example. 

\begin{lemma}
For the Flat Mountain $S$, with the notation of Theorem~\ref{thm:weak_rademacher}, we have that
\[
R_t=\delta_{\gamma(t)} e_1\wedge e_2\qquad\text{for $t\in[0,1]$,}
\]
where $\gamma$ is the space-filling curve of Section~\ref{subsec:curve}.
\end{lemma}
Observe that, in the notation of~\eqref{eq:structure_disint}, for the current $S$ it holds that
\[
\mu_t^\perp=\delta_{\gamma(t)}\perp \H^1\restrict \Gamma_t= \|S(t)\|,
\]
which is another way of saying that $S$ does not have the Sard property.

\begin{proof}
For any pair $t<t+h\in[0,1]$, by the constancy lemma there exists only one minimiser of the homogeneous flat norm $\Fhom(T_{t+h}-T_t)$, given by (in the notation of the proof of Theorem~\ref{thm:weak_rademacher}) $R_{t,h}=\1_{\{t\leq u(x)\leq t+h\}}e_1\wedge e_2$. Given $t\in[0,1)$, let us consider its base-$4$ expansion 
\[
t=\sum_{j=1}^\infty \frac{t_j}{4^j},\qquad t_j\in\{0,1,2,3\},
\]
where the expansion does not end with infinitely many $3$'s (so that it is uniquely determined). Given $n\in \mathbb{N}$, for all $h>0$ sufficiently small we necessarily have that
\[
\sum_{j=1}^n \frac{t_j}{4^j}\leq t\leq t+h < \frac{1}{4^n}+\sum_{j=1}^n \frac{t_j}{4^j}.
\]
By Lemma~\ref{lemma:properties_u_n} (iii), $\tfrac{1}{h}R_{t,h}$ is thus supported in $Q_{i_1,\ldots,i_n}$. Sending $n\to\infty$ we discover that $R_t$, being the weak limit of $\tfrac{1}{h}R_{t,h}$, must be supported in 
\[
\bigcap_{n\in\mathbb{N}} \overline{Q}_{i_1,\ldots,i_n}=\{\gamma(t)\}.
\]
Moreover, since the $\tfrac{1}{h}R_{t,h}$ have mass equal to $1$, it is easy to see that the limit must satisfy
\[
R_t =\delta_{\gamma(t)}e_1\wedge e_2
\]
and the proof is thus complete.
\end{proof}

\bibliographystyle{plain}
\bibliography{biblio_currents,Plast}
\end{document}